\documentclass[a4paper]{amsart}
\usepackage{sabbah_wild-twistor}

\begin{document}
\frontmatter
\title{Wild twistor $\mathcal D$-modules}

\author[C.~Sabbah]{Claude Sabbah}
\address{UMR 7640 du CNRS\\
Centre of Math\'ematiques Laurent Schwartz\\
\'Ecole polytechnique\\
F--91128 Palaiseau cedex\\
France}
\email{sabbah@math.polytechnique.fr}
\urladdr{http://www.math.polytechnique.fr/~sabbah}

\begin{abstract}
We propose a definition of (polarized) wild twistor $\mathcal D$-modules, generalizing to objects with irregular singularities that of (polarized) regular twistor $\mathcal D$-modules. We give a precise analysis in dimension one.
\end{abstract}


\maketitle
\tableofcontents
\mainmatter
\section*{Introduction}
In \cite{Bibi01c}, we have introduced the notion of \emph{polarized regular twistor $\cD$-module} on a complex manifold, and one of the main results was a \emph{decomposition theorem} for the direct image of such objects by a projective morphism between complex manifolds. A consequence of this theorem was the proof of a particular case of a conjecture of M\ptbl Kashiwara, saying that the direct image by a projective morphism of an irreducible holonomic $\cD$-module on a projective manifold should decompose into direct sums of irreducible holonomic $\cD$-module on the target manifold. The particular case treated was that of a \emph{smooth} twistor $\cD$-module.

T\ptbl Mochizuki, through a very precise analysis of harmonic metrics on the complement of a normal crossing divisor in a projective complex manifold \cite{Mochizuki07}, succeeded in proving the equivalence
\begin{gather*}
\textit{polarized regular twistor $\cD$-module}\\ \Big\updownarrow\\ \textit{semisimple regular holonomic $\cD$-modules}
\end{gather*}
on a projective manifold, giving therefore a proof of the conjecture of M\ptbl Kashiwara for semisimple regular holonomic $\cD$-modules.

On the other hand, using the Riemann-Hilbert correspondence, one can state the conjecture for perverse sheaves, and it is equivalent to the conjecture in the regular case. From this point of view, the conjecture has been proved by V\ptbl Drinfeld \cite{Drinfeld01}, modulo a conjecture of de~Jong, partly proved later by G\ptbl B\"ockle and C\ptbl Khare \cite{B-K03} on the one hand and proved by D\ptbl Gaitsgory \cite{Gaitsgory04} modulo results not yet written up on the other hand.

The goal of this article is to introduce a category of (polarized) wild twistor $\cD$-modules. Conjecturally, on any projective manifold, this category (in the polarized case) would be equivalent to the category of semisimple holonomic $\cD$-modules, and this would provide us with a tool for an analytic proof of the conjecture of M\ptbl Kashiwara for (possibly non regular) semisimple holonomic modules. We develop, in this context, an idea of P\ptbl Deligne \cite{Deligne83} for defining nearby cycles for irregular $\cD$-modules.

However, we do not give here any result in the direction of the previous conjecture. One would need to develop an analysis of harmonic metrics analogous to that developed by T\ptbl Mochizuki for tame harmonic metrics. Nevertheless, in dimension one, such kind of results have been obtained by O\ptbl Biquard and Ph\ptbl Boalch \cite{B-B03}.

The main result of this article (Theorem \ref{th:orbnilp}) is an analysis of the behaviour of wild twistor $\cD$-modules on a Riemann surface in the neighbourhood of the singularities.

\begin{remarque*}
In the recent preprint \cite{Mochizuki08}, T\ptbl Mochizuki enlarges the framework developed here and gives complete results on the theory.
\end{remarque*}

\subsubsection*{Acknowledgements}
I~thank the referee for his useful comments.

\section{Preliminaries}
In \oldS\S\ref{subsec:notadef}, \ref{subsec:lochyp} and \ref{subsec:strictspe}, we quickly review definitions and results from \cite{Bibi01c} concerning $\cR_\cX$-modules, which we refer to for more details. In \S\ref{subsec:Sdec}, we give complements on the notion of strict S-decomposability which was used in \cite{Bibi01c}.

\subsection{Notation and basic definitions}\label{subsec:notadef}
We fix a coordinate $\hb$ on the complex line~$\CC$. We denote by $\Delta_0$ the closed disc $\mhb\leq1$ and by $\bS$ its boundary $\mhb=1$. We will usually denote by $\Omega_0$ some open disc of radius $r>1$, and $r-1$ can be chosen arbitrarily small.

Let $X$ be a complex manifold. We denote by a curly $\cX$ the product $X\times\Delta_0$ and by $\cO_\cX$ the sheaf on~$\cX$ of germs on~$\cX$ of holomorphic functions on $X\times\Omega_0$. The sheaf $\cR_\cX$ of holomorphic differential operators is locally defined in coordinates as $\cO_\cX\langle\partiall_{x_1},\dots,\partiall_{x_n}\rangle$ with $\partiall_{x_j}\defin\hb\partial_{x_j}$ and, for any $f\in\cO_\cX$, $[\partiall_{x_j},f]=\hb\partial f/\partial x_j$. The category of left $\cR_\cX$-modules is equivalent to the category of $\cO_\cX$-modules equipped with a $\hb$-connection (\ie a $\CC$-linear endomorphism satisfying Leibniz rule with $\hb d$ instead of the usual differential $d$).

A $\cO_\cX$-module is said to be \emph{strict} if it has no $\cO_{\Omega_0|\Delta_0}$-torsion. The word ``strict'' always refer to such a property.

We often use the notion of coherent, good, and holonomic (left or right) $\cR_\cX$-module (\cf \cite[Def\ptbl1.2.4]{Bibi01c}).

Regarding the projective line $\PP^1$ as the union of two charts $\Omega_0$ and $\Omega_\infty$, we denote by $\sigma$ the anti-linear involution $\hb\mto-1/\ov\hb$, where $\ov\hb$ denotes the usual conjugate of~$\hb$. We use the notation $\ovv{\phantom{X}}$ for the ``twistor conjugation'': if $f\in\cO(\Omega)$, we define $\ovv f\in\cO(\sigma(\Omega))$ by the formula $\ovv f(\hb)=\ov{f(-1/\ov\hb)}$ (on the right-hand side, the conjugation is the usual one on complex numbers). If $\cH$ is a holomorphic bundle on~$\Omega_0$, then its ``conjugate'' $\ovv\cH\defin\sigma^*\ov\cH$ is a holomorphic bundle on~$\Omega_\infty$.

Twistor conjugation on~$\cX$ is meant as the usual conjugation on functions on~$X$ and twistor conjugation with respect to $\hb$. We denote by $\ovv\cX$ the product $\ov X\times\Delta_\infty$ and by $\cO_{\ovvv\cX}$ the sheaf of holomorphic functions on $\ov X\times\Delta_\infty$ (\ie anti-holomorphic with respect to $X$). The conjugate $\ovv\cM$ of a left $\cR_\cX$-module is a left $\cR_{\ovvv\cX}$-module.

We denote by $\cCh{\cX}$ the sheaf of $C^\infty$ functions on~$\cX$ which are holomorphic with respect to $\hb$. Similarly, $\cCc{\cX}$ denotes the sheaf of continuous functions on~$\cX$ which are $C^\infty$ with respect to $X$. We denote by $\Dbh{X}$ the sheaf on $X\times\bS$ of distributions which are continuous with respect to $\bS$.

An object $\cT$ of the category $\RTriples(X)$ consists of a triple $\cT=(\cM',\cM'',C)$, where $\cM',\cM''$ are left $\cR_\cX$-modules and $C$ is a sesquilinear pairing $\cMS'\otimes_{\cOS}\cMS''\to\Dbh{X}$, with $\cOS\defin\cO_{\Omega_0|\bS}$. We say that $\cT$ is smooth if $\cM',\cM''$ are $\cO_\cX$-locally free. In such a case, $C$ takes values in $\cCc{\cX|\bS}$ (\cf \cite[Lemma 1.5.3]{Bibi01c} where $\cCh{\cX|\bS}$ should read $\cCc{\cX|\bS}$).

The Tate twist by $k\in\hZZ$ of $\cT$ is defined by
\[
\cT(k)=(\cM',\cM'',(i\hb)^{-2k}C).
\]

\subsection{Localization away from a hypersurface}\label{subsec:lochyp}
Let $X_0$ be a complex manifold, let $D$ be an open disc in $\CC$ centered at the origin with coordinate $t$, and let $X$ be an open set in $X_0\times\nobreak D$. We also regard $t$ as a function on~$X$, and we denote by $\partiall_t$ the corresponding vector field. For simplicity, we may also denote by $X_0$ the divisor $t^{-1}(0)\subset X$ (which is open in the original~$X_0$).

We can extend the previous definitions to $\cR_\cX[\tm]$-modules:

\begin{definition}\label{def:wtRtriples}
The category $\wtRTriples(X)$ consists of objects $\wt\cT=(\wt\cM',\wt\cM'',\wt C)$, where $\wt\cM',\wt\cM''$ are $\cR_\cX[\tm]$-coherent and~$\wt C$ is a sesquilinear pairing between them taking values in the sheaf $\Dbh{X}[\tm]$ of distributions on $\{t\neq0\}$ depending continuously on $\hb\in\bS$ and having moderate growth along $\{t=0\}$.
\end{definition}

There is a natural functor (localization away from $\{t=0\}$) from $\RTriples(X)$ to $\wtRTriples(X)$.

\subsection{Strict specializability}\label{subsec:strictspe}
We keep notation of \S\ref{subsec:lochyp}. Let $\cM$ be a coherent $\cR_\cX$-module. We say that $\cM$ is \emph{strictly specializable} along $\{t=0\}$ (\cf \cite[Def\ptbl3.3.8]{Bibi01c}) if it has locally a decreasing Kashiwara-Malgrange filtration $V^\cbbullet_\phbo\cM$ indexed by $\RR$ with Bernstein polynomial having the special form\footnote{One can introduce more general kinds of Bernstein relations, in order to take into account various parabolic filtrations, as in \cite{Mochizuki07}. We will not do this here.} of a product of terms $t\partiall_t-\beta\star\hb$, where $\beta\star\hb\defin\reel\beta+i(\hb^2+1)\im\beta/2$ (\cf \cite[(3.3.3)]{Bibi01c} where we replace $\partiall_tt+\alpha\star\hb$ with $t\partiall_t-\beta\star\hb$ with $\beta=-\alpha-1$) for which the graded pieces have no $\hb$-torsion, and are generated through the action of $t$ or $\partiall_t$ by those for which the real part of the index belongs to $[-1,0]$. We will also use the notation $\beta=\beta'+i\beta''$ with $\beta',\beta''\in\RR$ and $\ell_\hb(\beta)\defin\reel(\beta'+i\hb\beta'')=\beta'-\beta''\im\hb$ (\cf \cite[\S0.9]{Bibi01c}, or \cite{Mochizuki07} for a different notation $\mathfrak{p},\mathfrak{e}$).

\begin{remarque}[Decreasing $V$-filtration]\label{rem:decVfil}
In this article, we use the decreasing convention for the $V$-filtrations. We indicate increasing filtrations with lower indices, and decreasing ones with upper indices. The correspondence with \cite{Bibi01c} is as follows, setting $\beta=-\alpha-1$,
\begin{align*}
V_\alpha^\phbo&=V^\beta_\phbo,\quad\psi_{f,\alpha}=\psi_f^\beta,\\
\tag*{and} \Psi_{f,\alpha}&=\Psi_f^\beta\quad\text{for $\reel\alpha\in[-1,0)$, so $\reel\beta\in(-1,0]$.}
\end{align*}
\end{remarque}

If $\cM$ is strictly specializable along $\{t=0\}$ then, for any $\beta\in\CC$, one defines the $\cR_{\cX_0}$-modules $\psi_t^\beta\cM$, equipped with a nilpotent endomorphism $\rN$ induced by $-(t\partiall_t-\beta\star\hb)$. The monodromy filtration $\rM_\bbullet$ attached to $\rN$ on $\psi_t^\beta\cM$, indexed by $\ZZ$, is defined by the properties that $\rN$ sends $\rM_k$ in $\rM_{k-2}$ and that, for each $\ell\geq0$, $\rN^\ell$ induces an isomorphism $\gr_\ell^{\rM}\isom\gr_{-\ell}^{\rM}$. (In general, $\gr_\ell^{\rM}\psi_t^\beta\cM$ is possibly not strict, but it is so when $\cM$ underlies a twistor $\cD$-module.) The primitive submodule $P\gr_\ell^{\rM}\psi_t^\beta\cM$ is defined, as usual, as the kernel of $\rN^{\ell+1}:\gr_\ell^{\rM}\psi_t^\beta\cM\to\gr_{\ell-2}^{\rM}\psi_t^\beta\cM$.

These notions can be extended (with a similar notation) to $\cR_\cX[\tm]$-modules $\wt\cM$, \cf \cite[\S3.4]{Bibi01c}.

We have functors between the categories of $\cR_\cX$- and $\cR_\cX[\tm]$-modules which are strictly specializable along $\{t=0\}$: the first one is the \emph{localization away from $\{t=0\}$}, and the other one is the \emph{minimal extension across $\{t=0\}$}.

If $\cM$ is strictly specializable along $\{t=0\}$, then the localized module $\cM[\tm]$ is $\cR_\cX[\tm]$-coherent and strictly specializable along $\{t=0\}$, and we define $\Psi_t^\beta\cM=\psi_t^\beta(\cM[\tm])$. The properties of $\Psi$ are given in \cite[\S3.4]{Bibi01c}.

Conversely, given any strictly specializable $\wt\cM$, we can define its minimal extension $\wt\cM_{\min_t}$ across $\{t=0\}$ as the $\cR_\cX$-submodule of $\wt\cM$ locally generated by $V_\phbo^{>-1}\wt\cM$ (\cf \cite[\S3.4.b]{Bibi01c}).

Starting from a strictly specializable $\cR_\cX$-module $\cM$, we then denote by $\cM_{\min_t}$ the $\cR_\cX$-module obtained from $\cM$ by composing both functors, namely $\cM_{\min_t}\defin(\wt\cM)_{\min_t}$. We say that $\cM$ is a \emph{minimal extension across $\{t=0\}$} if $\cM=\cM_{\min_t}$.

If $\cT=(\cM',\cM'',C)$ is an object of $\RTriples(X)$ and if $\cM',\cM''$ are strictly specializable along $\{t=0\}$, then the pairing $C$ can be specialized to each $\psi_t^\beta\cM',\psi_t^\beta\cM''$, defining thus $\psi_t^\beta\cT$ (\cf \cite[(3.6.10)]{Bibi01c}). For any $\ell\geq0$, the pairing $\psi_t^\beta C$ induces a pairing $\psi_t^{\beta,\ell}C$ between $\gr_{-\ell}^{\rM}\psi_t^\beta\cM'$ and the conjugate of $\gr_\ell^{\rM}\psi_t^\beta\cM''$. We denote by $P\psi_t^{\beta,\ell}C$ the pairing $\psi_t^{\beta,\ell}C((i\rN)^\ell\cbbullet,\ovv{\cbbullet})$ induced between $P\gr_\ell^{\rM}\psi_t^\beta\cM'$ and $P\gr_\ell^{\rM}\psi_t^\beta\cM''$. This defines an object $P\gr_\ell^{\rM}\psi_t^\beta\cT$. We will also have to consider the twisted object $P\gr_\ell^{\rM}\psi_t^\beta\cT(\ell/2)$, with $P\psi_t^{\beta,\ell}C(\ell/2)=(i\hb)^{-\ell}P\psi_t^{\beta,\ell}C$.

Similarly, if $\wt\cT=(\wt\cM',\wt\cM'',\wt C)$ is an object of $\wtRTriples(X)$ and if $\wt\cM',\wt\cM''$ are strictly specializable along $\{t=0\}$, then the pairing $\wt C$ can be specialized to each $\psi_t^\beta\wt\cM',\psi_t^\beta\wt\cM''$, with the same procedure as above. This defines $\psi_t^\beta\wt\cT$.

If $\wt\cT$ is obtained by localization of $\cT$ then, for $\reel\beta\in(-1,0]$, the specialized pairing $\psi_t^\beta\wt C$ is equal to $\psi_t^\beta C$ (this follows from \cite[(3.6.10) and Rem\ptbl3.4.4]{Bibi01c}). In other words, we can define $\Psi_t^\beta\cT$ as $\psi_t^\beta\wt\cT$, consistently with \cite[Def\ptbl3.6.11]{Bibi01c}.

Let now $f:X\to D$ be a holomorphic function. We denote by $i_f:X\hto X'\defin X\times D$ the graph inclusion $x\mto(x,f(x))$.

\begin{definition}\label{def:strictspef}
If $\cM$ is a coherent $\cR_\cX$-module (\resp if $\cT$ is an object of $\RTriples(X)$), we say that~$\cM$ (\resp $\cT$) is \emph{strictly specializable along $\{f=0\}$} if $i_{f,+}\cM$ (\resp $i_{f,+}\cT$) is so along $\{t=0\}$.
\end{definition}

A similar definition can be made for a $\cR_\cX[\tm]$-module $\wt\cM$ or an object $\wt\cT$ of $\wtRTriples(X)$. The notion of direct image $i_{f,+}$ is well-defined from the category of $\cR_\cX[1/f]$-modules to that of $\cR_{\cX'}[\tm]$-modules in a way compatible with the direct image of $\cR_\cX$-modules, that is, if $\wt\cM=\cM[1/f]$, then $i_{f,+}\wt\cM=(i_{f,+}\cM)[\tm]$. A similar notion applies to the corresponding categories $\wtRTriples(X)$ and $\wtRTriples(X')$.

\subsection{Strictly S-decomposable objects}\label{subsec:Sdec}
Let $\cM$ be a strictly specializable $\cR_\cX$-module along $\{t=0\}$. We say that it is \emph{strictly S-decomposable along $\{t=0\}$} if $\cM=\cM_{\min_t}\oplus\cM''$, with $\cM''$ having support in $\{t=0\}$. We notice that $\cM''$, being a direct summand of a strictly specializable object, is
also strictly specializable along $\{t=0\}$, hence take the form $i_+\cN''$ for some coherent $\cR_{\cX_0}$-module (Kashiwara's equivalence \cite[Cor\ptbl3.3.12]{Bibi01c}), where $i:X_0\hto X$ denotes the inclusion, and in fact $\cN''=\Ker[t:\cM''\to\cM'']$ (given any $\cM''$ supported in $\{t=0\}$, we always can define $\cN''$ by the previous formula, and the strict specializability of $\cM''$ insures that $\cM''=i_+\cN''$). In \cite[Prop\ptbl3.3.11]{Bibi01c}, we gave a characterization of such strictly specializable modules in terms of the morphisms $\can$ and $\var$, in a way analogous to \cite[Lemme 5.1.4]{MSaito86}.

\begin{remarque}
Let $f:X\to D$ be a holomorphic function and let $\wt\cM$ be a coherent $\cR_\cX[1/f]$-module. With the only assumption of strict specializability along $\{f=0\}$, it is not clear whether $(i_{f,+}\wt\cM)_{\min_t}$ is $i_{f,+}$ of some $\cR_\cX$-module. We will show below that, with the stronger assumption of strict S-decomposability, this property holds, and it enables us to define the minimal extension of $\cM$ across $\{f=0\}$.
\end{remarque}

We say that $\cM$ (\resp $\cT$) is \emph{strictly S-decomposable along $\{f=0\}$} if $i_{f,+}\cM$ (\resp $i_{f,+}\cT$) is strictly S-decomposable along $\{t=0\}$.

\begin{lemme}\label{lem:Sdecompf}
If $\cM$ is strictly S-decomposable along $\{f=0\}$, then the decomposition $i_{f,+}\cM=(i_{f,+}\cM)_{\min_t}\oplus(i_{f,+}\cM)''$ is the direct image by $i_f$ of a decomposition $\cM=\cM'\oplus\cM''$.
\end{lemme}

We will then denote $\cM'$ by $\cM_{\min_f}$ and call it the \emph{minimal extension of $\cM$ across $\{f=0\}$}. (This lemma is implicitly used in \cite[Prop\ptbl3.5.4]{Bibi01c}, when proving the existence of the decomposition of a strictly S-decomposable holonomic $\cR_\cX$-module with respect to the strict support.)

\begin{proof}[\proofname\ of Lemma \ref{lem:Sdecompf}]
We have $\cM=\Ker[t-f:i_{f,+}\cM\to i_{f,+}\cM]$. We set $\cM'=\Ker[t-f:(i_{f,+}\cM)_{\min_t}\to (i_{f,+}\cM)_{\min_t}]$ and $\cM''=\Ker[t-f:(i_{f,+}\cM)''\to (i_{f,+}\cM)'']$. Then we clearly have $\cM=\cM'\oplus\cM''$ and $i_{f,+}\cM'\subset (i_{f,+}\cM)_{\min_t}$, $i_{f,+}\cM''\subset (i_{f,+}\cM)''$. Both inclusions are equalities, as their direct sum is an equality by assumption.
\end{proof}

\begin{remarque}
If $\cM$ is strictly specializable along $\{f=0\}$, it is also strictly specializable along $\{f^r=0\}$ for any $r\geq1$ (\cf \cite[Prop\ptbl3.3.13]{Bibi01c}). If $\cM$ is strictly S-decomposable along $\{f=0\}$, then the decomposition $\cM=\cM_{\min_f}\oplus\cM''$ is also a decomposition relative to $f^r$, so in particular $\cM_{\min_f}$ is also the minimal extension relative to $f^r$.
\end{remarque}

We say that $\cM$ is strictly S-decomposable at $x_o\in X$ if it is strictly S-decomposable with respect to any germ at $x_o$ of holomorphic function on~$X$, and that $\cM$ is strictly S-decomposable if it is so at any point $x_o$ of $X$.

The following lemma is implicitly used in \loccit:

\begin{lemme}[Kashiwara's equivalence for strictly S-decomposable $\cR_\cX$-modules]\label{lem:kashequiv}
Let $i:Z\hto X$ be the inclusion of a submanifold. A coherent $\cR_\cX$-module $\cN$ is strictly S-decomposable and is supported in~$Z$ if and only if there exists a coherent strictly S-decomposable $\cR_\cZ$-module~$\cM$ such that $i_+\cM=\cN$. We then have $\cM=\cO_\cZ\otimes_{\cO_\cX}\cN$.
\end{lemme}

\begin{proof}
According to the second part, the problem is local, so we can reduce to the case where $Z$ is defined by an equation $t=0$, where $t$ is part of a coordinate system on~$X$.

Let $\cM$ be a coherent strictly S-decomposable $\cR_\cZ$-module. We will first prove that $i_+\cM$ is a strictly S-decomposable $\cR_\cX$-module.

On the one hand, let $f$ be a holomorphic function which vanishes identically on $Z$. Denote by $i_f:X\hto X\times\CC$ the natural inclusion, and by $u$ the coordinate on~$\CC$ (corresponding to $f$). Then $i(z)=(z,0)$ and $i_f\circ i(z)=(z,0,0)$, and the strict decomposition of $(i_f\circ i)_+\cM$ along $u=0$ is clear from that of $i_+\cM$ along $t=0$.

On the other hand, if $f$ does not vanish identically on~$Z$. Denote by $i_f:X\hto X\times\CC$ and $i_{f|Z}:Z\hto Z\times\CC$ the natural inclusions, and by $u$ the coordinate on~$\CC$ (corresponding to $f$ or $f|Z$). Denote by $i':Z\times\CC\hto X\times\CC$ the inclusion $i\times\Id_\CC$. Then $i_f\circ i=i'\circ i_{f|Z}$. By assumption, $i_{f|Z,+}\cM$ is strictly S-decomposable along $\{u=0\}$. Then one easily checks that $i'_+(i_{f|Z,+}\cM)=i_{f,+}(i_+\cM)$ is strictly S-decomposable along $\{u=0\}$, as the $V$-filtration relative to~$u$ is obtained from that of $i_{f|Z,+}\cM$ by applying $i'_+$.

Conversely, assume that $\cN$ is strictly S-decomposable on~$X$ and supported on~$Z$. It then strictly specializable along $\{t=0\}$ and, by \cite[Prop\ptbl 3.3.11(b)]{Bibi01c}, it takes the form $i_+\cM$ with $\cM=V^{-1}\cN$. Any holomorphic function~$f$ on~$Z$ locally extends as a holomorphic function on~$X$, and the $V$-filtrations of $i_{f,+}\cM$ and $i_{f,+}\cN$ along $\{t=0\}$ are easily related, showing that $\cM$ is strictly specializable along $\{f=0\}$ (\resp strictly S-decomposable) if~$\cN$ is~so.
\end{proof}

Let now $\cT=(\cM',\cM'',C)$ be an object of $\RTriples(X)$. We say that it is strictly S-decomposable (\resp holonomic) if $\cM',\cM''$ are so. If $\cT$ is holonomic and strictly S-decomposable, then $\cM',\cM''$ have a decomposition with respect to the strict support, and $C$ also decomposes, according to \cite[Prop\ptbl3.5.8]{Bibi01c}, hence $\cT$ also admits a decomposition with respect to the strict support.

\begin{lemme}\label{lem:kashequivT}
Kashiwara's equivalence \ref{lem:kashequiv} applies to strictly S-decomposable objects of $\RTriples(X)$.
\end{lemme}

\begin{proof}
This is Lemma 3.6.32 in \cite{Bibi01c}.
\end{proof}

\subsection[Minimal extension of strictly specializable objects]{Minimal extension of strictly specializable objects of~$\RTriples$}
\label{subsec:minext}
Let us take the setting of \S\ref{subsec:lochyp}. We have considered the two functors called ``localization along $\{t=0\}$'' and ``minimal extension across $\{t=\nobreak0\}$'' between strictly specializable\footnote{along $\{t=0\}$, if no other indication.} $\cR_\cX[\tm]$ and $\cR_\cX$-modules.

At many places it is simpler to work with $\cR_\cX[\tm]$-modules, for instance when considering ramification along $\{t=0\}$. However, when taking de~Rham complexes (or direct images), finiteness (or coherence) is obtained for coherent $\cR_\cX$-modules only. When trying to extend similar properties to objects of $\RTriples(X)$, we are led to define these functors at the level of the categories $\wtRTriples(X)$ and $\RTriples(X)$.

Let $\cT=(\cM',\cM'',C)$ be an object of $\RTriples(X)$ which is strictly specializable along $\{t=0\}$. Then we say that $\cT$ is a \emph{minimal extension across $\{t=0\}$} if $\cM',\cM''$ are so.

Given any strictly specializable object $\cT=(\cM',\cM'',C)$, the localization along $\{t=0\}$ is the object $\wt\cT\defin(\wt\cM',\wt\cM'',\wt C)$, where $\wt\cM',\wt\cM''$ are the localization of $\cM',\cM''$, and $\wt C$ is the natural extension of $C$ taking values in $\Dbh{X}[\tm]$.

On the other hand, we did not define in \cite{Bibi01c} the minimal extension $\cT_{\min_t}$, that is, we did not define a sesquilinear pairing $C_{\min_t}$ between the minimal extensions $\cM'_{\min_t}$ and $\cM''_{\min_t}$. The sesquilinear pairing $\wt C$ between $\wt\cM'$ and $\wt\cM''$ takes values in distributions having moderate growth along $\{t=0\}$, and one searches for a natural ``principal value'' $C_{\min_t}$ of $\wt C$ as being a sesquilinear pairing between $\cM'_{\min_t}$ and $\cM''_{\min_t}$ taking values in ordinary distributions (being understood that the distributions depend on $\hb\in\bS$ in a continuous way; in other words, we work with $\Dbh{X}$ and $\Dbh{X}[\tm]$). Such a result was not needed in \cite{Bibi01c}, as we mainly worked with strictly S-decomposable objects.

We will use the results of \cite[Appendix]{Bibi01c}, as detailed in \cite{Bibi05b}, to construct $C_{\min_t}$, hence $\cT_{\min_t}$. We introduce a new variable $\tau$ and denote by $\Afuh$ the corresponding complex line. We denote by $p$ the projection $Z=X\times\Afuh\to X$. We denote by $\cE^{-\tau/\hb t}$ the free $\cO_\cZ[\tm]$-module of rank one, with generator denoted by $\text{``}e^{-\tau/\hb t}\text{''}$, equipped with the action of $\cR_\cZ[\tm]$ defined by
\begin{align*}
t\partiall_\tau\,\text{``}e^{-\tau/\hb t}\text{''}&=-\text{``}e^{-\tau/\hb t}\text{''},\\
t^2\partiall_t\,\text{``}e^{-\tau/\hb t}\text{''}&=\tau\cdot\text{``}e^{-\tau/\hb t}\text{''},
\end{align*}
and we set $\cFcM\defin \cE^{-\tau/\hb t}\otimes p^+\wt\cM$, where the tensor product is taken over $\cO_\cZ[\tm]$ and is equipped with its natural structure of left $\cR_\cZ[\tm]$-module (see \loccit\ but be careful that the variable called $t$ here corresponds to the variable called $t'$ in \loccit, and we do not consider here the chart with variable $1/t$, called $t$ in \loccit).

\begin{proposition}\label{prop:minM}
Let $\wt\cM$ be $\cR_\cX[\tm]$-coherent and strictly specializable along $\{t=\nobreak0\}$. Then $\cFcM$ is $\cR_\cZ$-coherent and strictly specializable along $\{\tau=0\}$ and is a minimal extension across $\{\tau=0\}$. Moreover, we have a natural isomorphism
\[
\wt\cM_{\min_t}\isom P\gr_0^{\rM}\psi_\tau^0\cFcM.
\]
\end{proposition}

\begin{proof}
This is proved in \cite[Prop\ptbl4.1]{Bibi05b}. One has to notice that the assumption of regularity made in \loccit\ is only used to obtain regularity of $\cFcM$ along $\{\tau=0\}$ and to obtain properties of $\cFcM$ along $\tau=\tau_o\neq0$. These properties are not used here.
\end{proof}

\begin{remarque}
In \cite[Prop\ptbl4.1]{Bibi05b}, we start from a coherent $\cR_\cX$-module $\cM$ which is strictly specializable along $\{t=0\}$, but the proof only uses the localized object~$\wt\cM$.
\end{remarque}

Let $\wt\cT=(\wt\cM',\wt\cM'',\wt C)$ be a strictly specializable object of $\wtRTriples(X)$. In \cite[\S A.2.c]{Bibi01c}, we have defined the object $\cFcT=(\cFcM',\cFcM'',\cFC)$ of $\RTriples(Z)$ and we have proved (\cf \cite[Prop\ptbl5.8]{Bibi05b}):

\begin{proposition}
If $\cT$ is a minimal extension across $\{t=0\}$, then
$$
\cT\isom P\gr_0^{\rM}\psi_\tau^0(\cFcT).\eqno\qed
$$
\end{proposition}

In fact, the construction of $\cFcT$ in \loccit\ is done starting from an object of $\RTriples(X)$, but it is well-defined starting from an object $\wt\cT$ instead of an object $\cT$. It is therefore natural to define in general:

\begin{definition}[of $\wt\cT_{\min_t}$]\label{def:mint}
Let $\wt\cT$ be any strictly specializable object of $\wtRTriples(X)$. We define
\begin{equation}\tag{(\protect\ref{def:mint})$(*)$}\label{eq:defTmin}
\wt\cT_{\min_t}\defin P\gr_0^{\rM}\psi_\tau^0(\cFcT).
\end{equation}
\end{definition}

\section{Strict Deligne specializability}\label{sec:strictspeexp}

In \S\ref{subsec:deligne} we recall the notion of nearby cycles for irregular $\cD$\nobreakdash-modules introduced by P\ptbl Deligne \cite{Deligne83}, in order to explain the analogue for $\cR_\cX$\nobreakdash-modules.

\subsection{Irregular nearby cycles, after Deligne}\label{subsec:deligne}
In this section, we use the setting of \S\ref{subsec:lochyp}, but we work with holonomic $\cD_X$-modules. Let $M$ be a holonomic $\cD_X$-module and let $\wt M$ be its localization away from $\{t=0\}$. We denote by $V^\bbullet\wt M$ its Kashiwara-Malgrange filtration and we have well-defined holonomic $\cD_{X_0}$-modules $\psi_t^\beta\wt M$, for $\reel\beta\in(-1,0]$. It is known that $\psi_t^\beta\wt M$ are zero except for a locally finite number of $\beta$'s. But it may happen that all of them are zero. Therefore, they do not give any interesting information on~$\wt M$ along $\{t=0\}$.

We note that $\psi_t^\beta\wt M=\psi_t^0(t^+L_{\exp-2\pi i\beta}\otimes\wt M)$, where $L_{\exp-2\pi i\beta}$ is the rank one $\CC\{t\}[\tm]$-module with connection twisted by $t^{-\beta}$ and $t^+$ denotes the pull-back of connections by the map $t:X\to D$. In other words, $\psi_t\wt M=\oplus_L\psi_t^0(t^+L\otimes\wt M)$, where~$L$ runs over the $\CC\{t\}[\tm]$-modules with connection having \emph{regular singularity} and which are irreducible (\ie of rank one).

\begin{definition}\label{def:formirred}
Let $N$ be a free $\CC\{t\}[\tm]$-module with connection. We say that it is \emph{formally irreducible} if $\CC\lcr t\rcr[\tm]\otimes_{\CC\{t\}[\tm]}N$ is irreducible.
\end{definition}

According to a classical result of Turrittin and Levelt, the free $\CC\lcr t\rcr[\tm]$-module with connection~$\wh N$ is irreducible if and only if there exists an integer $q\geq1$ such that, denoting by $\rho_q:t_q\mto t=t_q^{q}$ the ramification of order $q$, $\wh N$ is the direct image by $\rho_q$ of a elementary $\CC\lcr t_q\rcr[\tqm]$-module with connection $\wh\cE^{-\varphi}\otimes \wh L$, where $L$ and $\cE^{-\varphi}$ satisfy the following properties:
\begin{enumerate}
\item
$L$ is regular and has rank one,
\item\label{enum:varphitirred}
$\cE^{-\varphi}\defin(\CC\{t_q\}[\tqm],d-d\varphi)$ with $\varphi\in\tqm\CC[\tqm]$, and for any $q$th root of unity $\zeta\neq1$, $\varphi(\zeta t_q)\neq\varphi(t_q)$.
\end{enumerate}

\begin{definition}[Irregular nearby cycles]\label{def:irregnearby}
If $M$ is left holonomic $\cD_X$-module, the irregular nearby cycles $\psi_t^\Del M$ are defined as
\begin{equation}\tag*{$(\protect\ref{def:irregnearby})(*)$}\label{eq:irregnearby}
\psi_t^\Del M\defin\tbigoplus_{\substack{N\text{ form.}\\ \text{irred.}}}\psi_t^0(t^+N\otimes M).
\end{equation}
\end{definition}

Let us note that $\psi_t^\Del M$ only depends on the localized module $\wt M$. In dimension one, the theorem of Levelt-Turrittin for $\wt M$ can be restated by saying that giving the formalized module $\wt M^\wedge$ is equivalent to giving $\psi_t^\Del\wt M$, which is a \emph{finite dimensional} graded vector space (the grading indices being the formally irreducible $N$'s) equipped with an automorphism.

It will be more convenient to use the following expression for $\psi_t^\Del M$. We say that $\varphi\in\tqm\CC[\tqm]$ is $t$-irreducible if it satisfies Condition \eqref{enum:varphitirred} above, \ie if $\rho_{q,+}\cE^{-\varphi}$ is irreducible. Then,
\begin{equation}
\psi_t^\Del M=\tbigoplus_{\varphi\text{ $t$-irred.}}\tbigoplus_{\reel\beta\in(-1,0]}\psi_t^\beta(t^+\rho_{q,+}\cE^{-\varphi}\otimes M).
\end{equation}

In dimension $\geq2$, the situation can be more complicated than in dimension one. Let us consider the following examples:

\begin{exemple}\label{ex:nearbyirreg}
$X_0=\Afu$ is the affine complex line with coordinate~$x$ and~$X$ is an open set in $X_0\times\CC$ (coordinates $x,t$). Let $\wt M$ be equal to $\cO_X[\tm]$ as a $\cO_X$-module, equipped with the connection $e^{-x/t}\circ d\circ e^{x/t}=d+dx/t-xdt/t^2$. We denote by $\ext$ the generator $1$ of $\wt M$. It satisfies thus both relations
\[
(t^2\partial_t+x)\ext=0\quad\text{and}\quad (t\partial_x-1)\ext=0.
\]
The second relation implies that the $V$-filtration along $t=0$ is constant equal to~$\wt M$. The same computation can be done after tensoring by any formally irreducible connection in dimension one. It follows that $\psi_t^\Del\wt M=0$, and the irregular nearby cycles along $\{t=0\}$ do not seem to bring any information on~$\wt M$ along $\{t=0\}$.

Instead of considering $\ext$ as in the previous example, we consider $\exxt$, defined in a similar way. Then one can show that $\psi_t^\Del\wt M$ is a $\cD_{X_0}$-module supported at $x=0$ and having $t$-monodromy equal to $-\id$.
\end{exemple}

\begin{remarque}\label{rem:finitude}
When the morphism $t:\Supp M\to\Afu$ is algebraic, Deligne \cite{Deligne83} proved that the sum in \ref{eq:irregnearby} is finite. One can conjecture that, in the analytic case we consider here, the same result holds in the neighbourhood of any compact set. This should be a consequence of the existence of a good formal structure after blowing-up.
\end{remarque}

\subsection{Strict Deligne specializability for $\cR_\cX[\tm]$-modules}\label{subsec:exptwist}
We keep notation as in \S\ref{subsec:deligne}. We denote by $D$ (\resp $D_q$) some open disc with coordinate $t$ (\resp $t_q$) and by $\rho_q:D_q\to D$, $t_q\mto t=t_q^{q}$, the ramification of order $q$. As above, we set~$\cD$ (\resp $\cD_q$) for $D\times\Omega_0$ (\resp $D_q\times\Omega_0$). We now define $\cE^{-\varphi/\hb}$ as the free rank-one $\cO_{\cD_q}[\tqm]$-module with generator denoted by $\ephi$ and with the $\hb$-connection $\hb d-d\varphi$. Then $\rho_{q,+}\cE^{-\varphi/\hb}$ is a free $\cO_{\cD}[\tm]$-module of rank~$q$ with $\hb$-connection.

Let $f:X\to D$ be a holomorphic function. The pull-back $f^+\rho_{q,+}\cE^{-\varphi/\hb}$ is a free $\cO_\cX[1/f]$-module with a $\hb$-connection, hence is a left $\cR_\cX[1/f]$-module.

\begin{definition}
We say that a $\cR_\cX[1/f]$-module $\wt\cM$ is \emph{strictly Deligne specializable} along \hbox{$f=0$} if, for any $t$-irreducible $\varphi\in\tqm\CC[\tqm]$, $f^+\rho_{q,+}\cE^{-\varphi/\hb}\otimes_{\cO_\cX[1/f]}\wt\cM$ is strictly specializable along $f=0$. We then set
\begin{equation}
\psi_t^\Del\wt\cM=\tbigoplus_{\varphi\text{ $t$-irred.}}\tbigoplus_{\reel\beta\in(-1,0]}\psi_f^\beta(f^+\rho_{q,+}\cE^{-\varphi/\hb}\otimes_{\cO_\cX[1/f]}\wt\cM).
\end{equation}

We say that a morphism $\mu:\wt\cM_1\to\wt\cM_2$ between strictly Deligne specializable $\cR_\cX[1/f]$-modules is \emph{strictly Deligne specializable} along $f=0$ if, for any $t$-irreducible $\varphi\in\tqm\CC[\tqm]$, the induced morphism
\[
\id\otimes\mu:f^+\rho_{q,+}\cE^{-\varphi/\hb}\otimes_{\cO_\cX[1/f]}\wt\cM_1\to f^+\rho_{q,+}\cE^{-\varphi/\hb}\otimes_{\cO_\cX[1/f]}\wt\cM_2
\]
is strictly specializable in the sense of \cite[Def\ptbl3.3.8(2)]{Bibi01c}, \ie $\ker$ and $\coker$ of $\psi_t^\beta(\id\otimes\mu)$ are strict for any $\beta$.
\end{definition}

This definition extends to $\cR_\cX$-modules and $\cR_\cX$-linear morphisms: we ask that such a module or morphism is strictly specializable in the sense of \cite[Def\ptbl3.3.8]{Bibi01c} and that the localized object satisfies the previous definition.

We now define the twist by $f^+\rho_{q,+}\cE^{-\varphi/\hb}$ on objects of $\wtRTriples(X)$. We have a natural pairing $\wt c_{q,-\varphi}$ on $\cE^{-\varphi/\hb}$ which takes values in functions with moderate growth on $D_q^*$ depending continuously on $\hb\in\bS$, defined, on the generator $\ephi$ of $\cE^{-\varphi/\hb}$ by
\begin{equation}\label{eq:wtcprime}
\wt c_{q,-\varphi}(\ephi,\ovv \ephi)=e^{\hb\ov\varphi-\varphi/\hb}.
\end{equation}
Since $\hb\in\bS$, we have $1/\hb=\ov\hb$, where $\ov\hb$ denotes the usual complex conjugate of~$\hb$, so the exponent in the exponential term reads $2i\im(\hb\ov\varphi)$ and the function $e^{\hb\ov\varphi-\varphi/\hb}$ has moderate growth, as well as all its derivatives, along $\{t_q=0\}$.

The direct image $\wt c_{-\varphi}$ of $\wt c_{q,-\varphi}$ by $\rho_q$ is defined as usual: a $\cO_{\cD}[\tm]$-basis of $\rho_{q,+}\cE^{-\varphi/\hb}$ consists of $\ephi,t_q\ephi,\dots,t_q^{q-1}\ephi$, and we set, for any test function~$\chi$ on~$D$ which is infinitely flat at $t=0$,
\begin{multline*}
\Big\langle\wt c_{-\varphi}(t_q^{k}\ephi,\ovv{t_q^\ell\ephi}),\chi(t)\itwopi \frac{dt}{t}\wedge \frac{d\ovt}{\ovt}\Big\rangle\\
=\int_{D_q}t_q^{k}\ovt_q^\ell e^{\hb\ov\varphi-\varphi/\hb}\chi(t_q^{q})q^2\itwopi\frac{dt_q}{t_q}\wedge \frac{d\ovt_q}{\ovt_q}.
\end{multline*}
We note that $\wt c_{-\varphi}$ is nothing but the trace of $\wt c_{q,-\varphi}$ by $\rho_q$, and is also a $C^\infty$ function with moderate growth on~$D^*$, depending holomorphically on~$\hb$.

If $\wt C:\wt\cMS'\otimes_{\cOS}\ovv{\wt\cMS''}\to\Dbh{X}[1/f]$ is a sesquilinear pairing, then one defines in a natural way a pairing $\wt C_{-\varphi}$ after twisting each term by $f^+\rho_{q,+}\cE^{-\varphi/\hb}$, using the fact that $f^*\wt c_{-\varphi}(t_q^{k}\ephi,\ovv{t_q^\ell\ephi})$ is a multiplier on $\Dbh{X}[1/f]$. This construction defines $f^+\rho_{q,+}\cE^{-\varphi/\hb}\otimes\nobreak\wt\cT$ as an object of $\wtRTriples(X)$.

For a strictly Deligne specializable $\wt\cM$ (or $\wt\cT$), we can therefore consider the $(\varphi,\beta)$-irregular nearby cycles defined as
\begin{align*}
\psi_f^{\varphi,\beta}\wt\cM&\defin\psi_f^\beta(f^+\rho_{q,+}\cE^{-\varphi/\hb}\otimes\wt\cM)\\
\tag*{or} \psi_f^{\varphi,\beta}\wt\cT&\defin\psi_f^\beta(f^+\rho_{q,+}\cE^{-\varphi/\hb}\otimes\wt\cT),
\end{align*}
where the functor $\psi_f^\beta$ is the functor used in \S\ref{subsec:strictspe} (that is, $\psi_t^\beta\circ i_{f,+}$).

\begin{remarque}
These definitions and constructions extend in a natural way to $\cR_\cX$-modules or objects of $\RTriples(X)$, as tensoring with $f^+\rho_{q,+}\cE^{-\varphi/\hb}$ does not distinguish between $\cM$ and $\wt\cM$. We denote by $\Psi_f^{\varphi,\beta}$ the corresponding functor.
\end{remarque}

\begin{proposition}[Regularity]\label{prop:strictspereg}
Let $\cM$ be a $\cR_\cX$-module which is strictly specializable along $f=0$ and which is regular along $f=0$ (in the sense of \cite[\S3.1.d]{Bibi01c}). Then $\cM$ is strictly Deligne specializable along $f=0$ and $\psi_f^{\varphi,\beta}\cM=0$ if $\varphi\neq0$.
\end{proposition}

\begin{proof}
It is enough to consider the case where $f$ is a projection, \ie the situation of an open set $X$ of $X_0\times D$ as in \S\ref{subsec:strictspe}. Let us first consider the case where $\varphi\in\tm\CC[\tm]\moins\{0\}$. Let us work near $\hbo\in\Omega_0$ and let us consider the canonical $V$-filtration $V_\phbo^\cbbullet\wt\cM$. By assumption, $V_\phbo^{>-1}\wt\cM$ is $\cR_{\cX/D}$-coherent. It is then enough to prove that
\[
V^0(\cR_\cX)\cdot\big(\cE^{-\varphi/\hb}\otimes V_\phbo^{>-1}\wt\cM\big)=\cE^{-\varphi/\hb}\otimes\wt\cM,
\]
because this will imply that the constant filtration equal to $\cE^{-\varphi/\hb}\otimes\nobreak\wt\cM$ is a good $V$-filtration, hence the Bernstein polynomial exists and is constant.

One proves by induction on $k\geq0$ that, for any local section~$m$ of $V_\phbo^{>-1}\wt\cM$, denoting by $\ephi$ the $\cO_\cD[\tm]$-generator of $\cE^{-\varphi/\hb}$, $\ephi\otimes t^{-k}m$ belongs to $V^0(\cR_\cX)\cdot\big(\cE^{-\varphi/\hb}\otimes V_\phbo^{>-1}\wt\cM\big)$. Let us show this for $k=1$ for instance. One has
\[
t\partiall_t(\ephi\otimes m)=\ephi\otimes[t\partiall_tm-(t\partial_t\varphi)m],
\]
so that, if $p$ is the order of the pole of $\varphi$, multiplying both sides by $t^{p-1}$, we find $\ephi\otimes t^{-1}m\in V^0(\cR_\cX)\cdot\big(\cE^{-\varphi/\hb}\otimes V_\phbo^{>-1}\wt\cM\big)$. The induction is then easy.

When $q\geq2$, the same argument can be applied after the ramification $\rho_q$, as $\rho_q^+\wt\cM$ remains strictly specializable and regular along $t_q=0$, as explained below.
\end{proof}

\subsection{Ramification}\label{subsec:ram}

We keep notation as in \S\ref{subsec:lochyp}. Let us fix an integer $q\geq2$ and let us denote by $\rho_q:X_0\times D_q\to X_0\times D$ the ramification $t_q\mto t=t_q^{q}$. We denote then by $X_q$ the inverse image of $X$ by $\rho_q$ ($X_q$ is smooth) and by $\rho_q:X_q\to X$ the restricted morphism. The functor $\rho_q^+$ (inverse image by $\rho_q$) is well defined from the category of (left) $\cR_\cX$-modules to that of $\cR_\cXq$-modules as follows: if $\cM$ is a left $\cR_\cX$-module,
\begin{itemize}
\item
as a $\cO_\cXq$-module, we set $\rho_q^+\cM=\rho_q^*\cM=\cO_\cXq\otimes_{\rho_q^{-1}\cO_\cX}\rho_q^{-1}\cM$;
\item
for coordinates $x_i$ on~$X_0$, the action of $\partiall_{x_i}$ is the natural one, \ie $\partiall_{x_i}(1\otimes m)=1\otimes\partiall_{x_i}m$;
\item
the action of $\partiall_{t_q}$ is defined, by a natural extension using Leibniz rule, from
\[
\partiall_{t_q}(1\otimes m)=qt_q^{q-1}\otimes\partiall_tm.
\]
\end{itemize}

However, given an object $\cT=(\cM',\cM'',C)$ of $\RTriples(X)$, the object $\rho_q^+\cT$ is not clearly defined: the pull-back by $\rho_q$ of a distribution on~$X$ is not defined, since the trace by $\rho_q$ of a test form on~$X_q$ can have singularities.

On the other hand, the pull-back of a $\cR_\cX[\tm]$-module is defined similarly and the pull-back of a \emph{moderate} distribution along $\{t=0\}$ is well-defined as a moderate distribution along $\{t_q=0\}$. Therefore, the functor $\rho_q^+$ is well defined from $\wtRTriples(X)$ to $\wtRTriples(X_q)$.

\begin{proposition}\label{prop:strictsperamif}
Assume that $\wt\cM$ is strictly specializable along $\{t=0\}$. Then $\rho_q^+\wt\cM$ is so along $\{t_q=0\}$.
\end{proposition}

\begin{proof}[\proofname\ of Proposition \ref{prop:strictsperamif}]
Near $(x_o,\hbo)$, one shows that any local section of $\rho_q^+\wt\cM$ satisfies a Bernstein functional equation, using that
\[
t_q^{k}\otimes t\partiall_tm=\frac1q(t_q\partiall_{t_q}-k\hb)(t_q^{k}\otimes m).
\]
If we identify $\rho_q^+\wt\cM$ with $\bigoplus_{k=0}^{q-1}t_q^{k}\otimes\wt\cM$ (with the suitable $\cR_\cXq[\tqm]$-module structure on the right-hand term), the filtration $V_\phbo^\cbbullet\rho_q^+\wt\cM$ defined by the formula
\begin{equation}\label{eq:Vpi+}
V_\phbo^b\rho_q^+\wt\cM=\tbigoplus_{k=0}^{q-1}(t_q^{k}\otimes V_\phbo^{(b-k)/q}\wt\cM),
\end{equation}
satisfies all properties required for the Kashiwara-Malgrange filtration (\cf \cite[Lemma~3.3.4]{Bibi01c}).
\end{proof}

\begin{remarque}\label{rem:ramif}
For any $\beta\in\CC$, we then have
\[
\psi_{t_q}^\beta\rho_q^+\wt\cM\simeq\tbigoplus_{k=0}^{q-1}\psi_t^{(\beta-k)/q}\wt\cM,
\]
and, under this identification, the nilpotent endomorphism $\rN_{t_q}$ corresponds to the direct sum of the nilpotent endomorphisms $q\rN_t$. Therefore, we have a similar relation for the graded modules with respect to the monodromy filtration and the corresponding primitive submodules.

Similarly, if $\wt\cT$ is strictly specializable, we have
\[
\psi_{t_q}^\beta\rho_q^+\wt\cT\simeq\tbigoplus_{k=0}^{q-1}\psi_t^{(\beta-k)/q}\wt\cT.
\]
Indeed, the point is to compute $\psi_{t_q}^\beta\rho_q^+\wt C$. Let $m'$ (\resp $m''$) be a local section of $V_\phbo^{(b-k')/q}\wt\cM'$ (\resp $V_\phbo^{(b-k'')/q}\wt\cM''$) with $k',k''\in[0,q-1]\cap\NN$. Up to translating $b$ by an integer, we can assume that $\wt C(m',\ovv{m''})$ takes values in $\Dbh{X}$. Let $\varphi$ be a form of maximal degree on~$X_0$ (with coefficients depending on $\hb\in\bS$) and let $\chi$ be a $C^\infty$ function on~$D_q$ with compact support, identically equal to $1$ near $t_q=0$. In the following, it will not matter to assume that $\chi$ is a function on~$D$.

Setting $\alpha=-\beta-1$, and denoting by $[m]$ the class of~$m$ in the graded piece of the $V$-filtration, we have by definition
\begin{multline*}
\big\langle\psi_{t_q}^\beta\rho_q^+\wt C([t_q^{k'}\otimes m'],\ovv{[t_q^{k''}\otimes m'']}),\varphi\big\rangle\\
=\res_{s=\alpha\star\hb/\hb}\big\langle\vert t_q\vert^{2s}\rho_q^+\wt C(t_q^{k'}\otimes m',\ovv{t_q^{k''}\otimes m''}),\varphi\wedge\chi(t)\itwopi dt_q\wedge d\ov{t_q}\big\rangle.
\end{multline*}
Up to a positive constant, the right-hand term is the residue of
\[
\Big\langle\wt C(m',\ovv{m''}),\varphi\wedge\chi(t)\mt^{2(s+1)/q}\tr(t_q^{k'}\ov{t_q^{k''}})\itwopi \frac{dt}{t}\wedge \frac{d\ovt}{\ovt}\Big\rangle.
\]
As $k',k''$ belong to $[0,q-1]$, the trace $\tr(t_q^{k'}\ov{t_q^{k''}})$ is zero unless $k'=k''$, giving therefore the desired formula for $\psi_{t_q}^\beta\rho_q^+\wt C$.
\end{remarque}

Let now $X$ be a complex manifold and let $f:X\to D$ be a holomorphic function. We set $X'=X\times D$ and $X'_0=X\times\{0\}\simeq X$. We denote by $i_f:X\hto X'$ the inclusion $x\mto (x,f(x))$ and by $t$ the coordinate on the factor $D$. For $q\geq1$, we denote by $X_q$ the inverse image of $i_f(X)$ in $X'_q$ by $\rho_q:X'_q\to X'$. If we identify $X'_q$ with $X\times D_q$ (coordinate $t_q$ on the factor $D_q$) so that $\rho_q(x,t_q)=(x,t_q^{q})$, then $X_q$ is the subset of $X'_q$ defined by $f(x)-t_q^{q}=0$. It can be singular.

If $\wt\cM$ is a $\cR_\cX[1/f]$-module (\resp if $\wt\cT$ is an object of $\wtRTriples(X)$) which is strictly specializable along $\{f=0\}$ (\cf Definition \ref{def:strictspef}), we define $\rho_q^+\wt\cM$ (\resp $\rho_q^+\wt\cT$) as $\rho_q^+(i_{f,+}\wt\cM)$ (\resp $\rho_q^+(i_{f,+}\wt\cT)$) as in the beginning of this subsection. This is a $\cR_{\cX'_q}[\tqm]$-module (\resp an object of $\wtRTriples(X'_q)$) supported on~$X_q$.

\begin{corollaire}
If $\wt\cM$ (\resp $\wt\cT$) is strictly specializable along $\{f=0\}$, then $\rho_q^+(i_{f,+}\wt\cM)$ (\resp $\rho_q^+(i_{f,+}\wt\cT)$) is so along $t_q=0$.\qed
\end{corollaire}

\begin{corollaire}
$\wt\cM$ (\resp $\wt\cT$) is strictly Deligne specializable along $\{f=0\}$ if and only if, for any $q\geq1$ and for any $\varphi\in\tqm\CC[\tqm]$, $\cE^{-\varphi/\hb}\otimes_{\cO_{\cX'_q}}\rho_q^+(i_{f,+}\wt\cM)$ (\resp $\cE^{-\varphi/\hb}\otimes_{\cO_{\cX'_q}}\rho_q^+(i_{f,+}\wt\cT)$) is so along $\{t_q=0\}$.
\end{corollaire}

\begin{proof}
Let us prove the `only if' part. We can reduce to the case where $f$ is a coordinate $t$. We use the projection formula to get that $t^+\rho_{q,+}\cE^{-\varphi/\hb}\otimes_{\cO_\cX[\tm]}\wt\cM=\rho_{q,+}(t_q^+\cE^{-\varphi/\hb}\otimes\rho_q^+\wt\cM)$. We are thus reduced to proving that strict specializability is preserved by the direct image $\rho_{q,+}$. This follows from \cite[Th\ptbl3.3.15]{Bibi01c}, as the restriction of $\rho_q$ to $t_q=0$ is equal to the identity.

For the `if' part, we apply Proposition \ref{prop:strictsperamif} and we use that $\rho_q^+\rho_{q,+}\cE^{-\varphi/\hb}$ decomposes as the direct sum $\oplus_{\zeta^q=1}\cE^{-\varphi\circ\mu_\zeta/\hb}$, where $\mu_\zeta$ is the multiplication by~$\zeta$.
\end{proof}

\subsection{Compatibility with proper direct images}
Let $g:X\to Y$ be a morphism between complex analytic manifolds and let $\cM$ be a good $\cR_\cX$-module (\ie which has good filtrations on compact sets of $X$, \cf \cite[\S1.1.c]{Bibi01c}). Let $f:Y\to\CC$ be a holomorphic function and assume that $\cM$ is strictly specializable along $(f\circ g)^{-1}(0)$. If~$g$ is proper on the support of $\cM$ then (\cf \cite[Th\ptbl3.3.15]{Bibi01c}) for all $j$ and $\beta$ with $\reel\beta\in(-1,0]$, there is a natural isomorphism $\Psi_f^\beta\cH^jg_+\cM\simeq\cH^jg_+\Psi_{f\circ g}^\beta\cM$ \emph{if we assume that $\cH^jg_+\Psi_{f\circ g}^\beta\cM$ are strict for all $j$ and $\beta$ with $\reel\beta\in(-1,0]$}. Under the same assumption, we have an isomorphism for vanishing cycles $\psi_f^{-1}\cH^jg_+\cM\simeq\cH^jg_+\psi_{f\circ g}^{-1}\cM$. Moreover, this isomorphism extends to objects of $\RTriples(X)$ (\cf \cite[\S3.6.c]{Bibi01c}).

We now wish to extend this result to the objects $\Psi_f^\Del \cM$ and $\Psi_f^\Del \cT$. Let us first note that we will not be concerned with the corresponding vanishing cycles (vanishing cycles are used in \cite{Bibi01c} mainly in \S6.3 to insure the S-decomposability of direct images; we will use them here in the same way, without any twist). It will therefore be more convenient to work with localized modules and to consider direct images of $\cR_\cX[1/(f\circ g)]$-modules.

\begin{proposition}
Let $\wt\cM$ be a good $\cR_\cX[1/(f\circ g)]$-module which is strictly Deligne specializable along $\{f\circ g=0\}$ and such that $g$ is proper on the support of $\wt\cM$. Then, if $\cH^jg_+\psi_{(f\circ g)}^\Del \wt\cM$ are strict for any~$j$, we have
\[
\psi_f^\Del \cH^jg_+\wt\cM\simeq\cH^jg_+\psi_{f\circ g}^\Del \wt\cM.
\]
A similar result holds for objects of $\wtRTriples(X)$.
\end{proposition}

\begin{proof}[Sketch of proof]
We apply the result recalled above to the module $(f\circ\nobreak g)^+\rho_{q,+}\cE^{-\varphi/\hb}\otimes\nobreak\wt\cM$ and we use a projection formula.
\end{proof}

\subsection{Integrability}
We now consider the behaviour of integrability under Deligne specialization. The integrability property is defined in \cite[Chap\ptbl7]{Bibi01c}, where it is proved that it is compatible with specialization (\cf Proposition~7.3.1 in \loccit). To prove the compatibility with Deligne specialization it is therefore enough to prove that $f^+\rho_{q,+}\cE^{-\varphi/\hb}\otimes\wt\cM$ (\resp $f^+\rho_{q,+}\cE^{-\varphi/\hb}\otimes\nobreak\wt\cT$) remains integrable when $\cM$ (\resp $\cT$) is~so.

Firstly, integrability is preserved by direct image (\cf \cite[Prop\ptbl7.1.4]{Bibi01c}), so we can replace $\cM$ with $i_{f,+}\cM$ and assume that $f$ is the coordinate~$t$. Moreover, using the projection formula, we are reduced to showing that $\cE^{-\varphi/\hb}\otimes\rho_q^+\wt\cM$ is integrable.

Let us note that $\cE^{-\varphi/\hb}$ is integrable, either as a $\cR_{\cD_q}[\tqm]$-module or as an object of $\wtRTriples(D_q)$. Indeed, one defines the action of $\hb^2\partial_\hb$ on the generator $\ephi$ of $\cE^{-\varphi/\hb}$ as $\hb^2\partial_\hb(\ephi)=\varphi$. On the other hand, with this definition and using \eqref{eq:wtcprime}, one has
\begin{multline*}
\hb\partial_\hb\wt c_{q,-\varphi}(\ephi,\ovv \ephi)\\
=\wt c_{q,-\varphi}(\hb\partial_\hb(\ephi),\ovv \ephi)-\wt c_{q,-\varphi}(\ephi,\ovv{\hb\partial_\hb(\ephi)}),
\end{multline*}
which gives the desired integrability (\cf \cite[(7.1.2)]{Bibi01c}) of $\cE^{-\varphi/\hb}$ and then of $\cE^{-\varphi/\hb}\otimes\rho_q^+\wt\cM$.\qed

\section{Polarizable wild twistor $\cD$-modules}

\subsection{Wild and regular (polarizable) twistor $\cD$-modules}
Let $X$ be a complex manifold and let $w\in\ZZ$. We will define by induction on $d\in\NN$ the category $\MTw_{\leq d}(X,w)$ of \emph{wild} \emph{twistor $\cD$-modules of weight $w$ on~$X$, having support of dimension $\leq d$}. This will be a full subcategory of the category $\RTriples(X)$.

\begin{definition}[Wild twistor $\cD$-modules]\label{def:wildtw}
The category $\MTw_{\leq d}(X,w)$ is the full subcategory of $\RTriples(X)$, the objects of which are triples $\cT=(\cM',\cM'',C)$ satisfying:

\smallskip\noindent
(HSD) $\cT$ is holonomic, strictly S-decomposable and has support of dimension $\leq d$.

\smallskip\noindent
($\MTw_{>0}$) For any open set $U\subset X$ and any holomorphic function $f:U\to\nobreak\CC$, $\cT$~is strictly Deligne specializable along $\{f=0\}$ and, for any integer $\ell\geq0$, the triple
$$
\gr_\ell^{\rM}\Psi_f^\Del \cT
$$
is an object of $\MTw_{\leq d-1}(X,w+\ell)$.

\smallskip\noindent
$(\MT_0)$
For any zero-dimensional strict component $\{x_o\}$ of $\cM'$ or $\cM''$, we have
\[
(\cM'_{\{x_o\}},\cM''_{\{x_o\}},C_{\{x_o\}})=i_{\{x_o\}+}(\cH',\cH'',C_o)
\]
where $(\cH',\cH'',C_o)$ is a twistor structure of dimension~$0$ and weight~$w$.
\end{definition}

One can define the notion of polarization and the category $\MTw_{\leq d}(X,w)^\rp$ of polarized objects as in \cite[Def\ptbl4.2.1]{Bibi01c}:

\begin{definition}[Polarization]\label{def:polarization}
A \emph{polarization} of an object $\cT$ of $\MTw_{\leq d}(X,w)$ is a ses\-qui\-li\-near Hermitian duality $\cS:\cT\to\cT^*(-w)$ of weight~$w$ such that:

\smallskip\noindent
$(\MTPw_{>0})$
for any open set $U\subset X$ and any holomorphic function $f:U\to\CC$, for any $\alpha$ with $\reel(\alpha)\in[-1,0)$ and any integer $\ell\geq 0$, the morphism $(P\gr_\bbullet^{\rM}\Psi^\Del_{f,\alpha}\cS)_\ell$ induces a polarization of $P_\ell\Psi^\Del_{f,\alpha}\cT$,

\smallskip\noindent
$(\MTP_0)$
for any zero-dimensional strict component $\{x_o\}$ of $\cM'$ or $\cM''$, we have $\cS=i_{\{x_o\}+}\cS_o$, where $\cS_o$ is a polarization of the zero-dimensional twistor structure $(\cH',\cH'',C_o)$.
\end{definition}

\begin{remarque}
According to Proposition \ref{prop:strictspereg}, the category $\MTr_{\leq d}(X,w)$ (\resp $\MTr_{\leq d}(X,w)^\rp$) of \emph{regular (\resp polarized) twistor $\cD$-modules} introduced in \cite[Def\ptbl4.1.2]{Bibi01c} (\resp in \cite[Def\ptbl4.2.1]{Bibi01c}) is a full sub-category of $\MTw_{\leq d}(X,w)$ (\resp of $\MTw_{\leq d}(X,w)^\rp$).
\end{remarque}

\subsection{Some properties of wild twistor $\cD$-modules}\label{subsec:someprop}
In this paragraph, we list some of the properties of the categories $\MT(X,w)$ or $\MT(X,w)^\rp$, as proved in \cite[\S\S4.1 and 4.2]{Bibi01c}, which are also shared by the categories $\MTw(X,w)$ or $\MTw(X,w)^\rp$. The extension is essentially straightforward.

\begin{enumerate}
\item
The category $\MTw_{\leq d}(X,w)$ is a full subcategory of the category $\MT_{\leq d}(X,w)$ defined in \cite[Def\ptbl4.1.1]{Bibi01c}.
\item
The category $\MTw_{\leq d}(X,w)$ is local (\ie checking that an object belongs to this category can be done locally on~$X$). Moreover, it is stable by direct summand in $\RTriples(X)$. Each object $\cT$ in $\MTw_{\leq d}(X,w)$ has a decomposition by the strict support.

\item
\emph{Kashiwara's equivalence}. If $i:X\hto X'$ denoted the inclusion of a closed analytic submanifold of $X'$, then $i_+$ is an equivalence between $\MTw(X,w)$ and $\MTw_X(X',w)$ (objects of $\MTw(X',w)$ supported in $X$). This follows from Lemma \ref{lem:kashequivT}.

\item
If $\cT$ in $\MTw(X,w)$ has strict support $Z$, then there exists a Zariski open dense smooth set $Z'$ of $Z$ on which $\cT$ is $i_+\cT'$, where $\cT'$ is a smooth twistor structure of weight $w$ on~$Z'$.

\item
There is no nonzero morphism in $\RTriples(X)$ from an object of $\MTw(X,w)$ to an object of $\MTw(X,w')$ if $w>w'$.

\item
The category $\MTw(X,w)$ is abelian, all morphisms are strict and strictly Deligne specializable.

\item
If $\cT_1$ is a sub-object in the category $\MTw(X,w)$ of polarizable object $\cT$ in $\MTw(X,w)^\rp$ with polarization~$\cS$, then the polarization induces a polarization~$\cS_1$ on~$\cT_1$ and $(\cT_1,\cS_1)$ is a direct summand of $(\cT,\cS)$ in $\MTw(X,w)^\rp$.

\item
The category $\MTw(X,w)^\rp$ is semi-simple.

\item
The conclusion of \cite[Prop\ptbl2.1.19 and 2.1.21]{Bibi01c} hold for graded Lefschetz polarized wild twistor $\cD$-modules.

\item
The spectral sequence degeneration argument of \cite[Cor\ptbl4.2.11]{Bibi01c} holds in the wild case.
\end{enumerate}

\section{Local properties of $\cR_\cX$-modules in dimension one}

In this section, we analyze local properties of $\cR_\cX$-modules which are strictly specializable, when $X$ has dimension one. Therefore, we will assume that $X$ is a disc with coordinate $t$. We set $X^*=X\moins\{0\}$. We denote by $K$ the field of convergent Laurent series $\CC\{t\}[\tm]$ and by $\wh K$ the field of formal Laurent series $\CC\lcr t\rcr[\tm]$. For $q\in\NN^*$, we denote by $K_q$ (\resp $\wh K_q$) the extension of $K$ (\resp $\wh K$) obtained by taking a $q$-th root $t_q$ of $t$.

\subsection{Meromorphic bundles with connection}\label{subsec:meromconn}

We recall here some classical results on meromorphic connections (\cf \eg \cite{Malgrange91}). Let $\wt M$ be a finite dimensional $K$-vector space with a connection~$\nabla$. There exists $q\in\NN^*$ such that the pull-back $\wt M_q=K_q\otimes_K\wt M$ has a \emph{formal} decomposition
\begin{equation}\label{eq:decomp}
\wt M_q^\wedge\defin\wh K_q\otimes_{K_q}\wt M_q=\tbigoplus_{j\in J}(\ccE^{\varphi_j}\otimes\wh R_j),
\end{equation}
where $J$ is a finite set, the $\varphi_j$ are pairwise distinct elements of $\tqm\CC[\tqm]$, and the $R_j$ are meromorphic bundles with a regular connection. A coarser decomposition descends to $\wh {\wt M}=\wh K\otimes_K\wt M$ to give a decomposition $\wt M^\wedge=\wt M^\wedge{}^\rr\oplus\wt M^\wedge{}^\ir$.

\subsection{Meromorphic Higgs bundles}\label{subsec:higgslocalone}

Before analyzing wild twistor $\cD$-modules in dimension one, we need to adapt the well-known results on the formal structure of meromorphic connections in dimension one to Higgs objects.

Let $(\wt M,\theta)$ be a finite dimensional $K$-vector space equipped with an endomorphism~$\theta$. To keep the analogy with connections, we also regard~$\theta$ as a Higgs field, \ie a~morphism $\wt M\to\wt M\otimes dt$.

By a lattice $E$ of $\wt M$, we mean $\CC\{t\}$-submodule of finite type of $\wt M$ such that $\wt M=E[\tm]$. It is equivalent to saying that $E$ is free of finite rank over $\CC\{t\}$ and that $\wt M=E[\tm]$. By a $V_0$-lattice, we mean a $\CC\{t\}[t\theta]$-submodule $U$ of $\wt M$ which is of finite type and such that $U[\tm]=\wt M$.

\begin{lemme}\label{lem:V0}
If $U$ is any $V_0$-lattice, then $U/tU$ is a finite dimensional $\CC$-vector space with an endomorphism induced by $t\theta$.
\end{lemme}

\begin{proof}
Let us fix a $K$-basis of $\wt M$ and let $A$ be the matrix of $t\theta$ in this basis. Denote by $\chi_A$ its characteristic polynomial. Then $\chi_A(t\theta)=0$. This can be written as $b(t\theta)=tQ(t,t\theta)$, for some nonzero polynomial~$b$. If $U$ is any $V_0$-lattice, then $U/tU$ is a $\CC[t\theta]$-module of finite type, on which $b(t\theta)$ vanishes. Hence it is finite dimensional.
\end{proof}

We say that $(\wt M,\theta)$ is \emph{regular} if there exists a $K$-basis of $\wt M$ in which the matrix of $\theta$ has at most a simple pole; in other words, if there exists a lattice $E$ of $\wt M$ on which $\theta$ has a simple pole at $0$. We then say that $(E,\theta)$ is a \emph{logarithmic Higgs bundle}. Clearly, this is equivalent to saying that any $V_0$-lattice is a lattice.

Giving $(\wt M,\theta)$ is equivalent to giving a square matrix of size $d$ with entries in $K$, up to linear equivalence, \ie up to conjugation by an element of $\GL_d(K)$.

The tensor product and the direct sum are well-defined and preserve regular objects (the tensor product of two Higgs fields $\theta$ and $\theta'$ is $(\theta\otimes\nobreak\Id)\oplus(\Id\otimes\theta')$).

\subsubsection{Classification in rank one}
Giving a meromorphic Higgs bundle of rank~one is equivalent to giving a meromorphic differential form $\omega=a(t)dt$ (the equivalence by $\GL_1$ is reduced to identity). We will write $a(t)=\partial_t(\varphi)+a_{-1}/t$ for some meromorphic function $\varphi(t)=t^ku(t)$, where $k\in\ZZ$ and $u$ is a unit. Notice that, if $a_{-1}=0$, the meromorphic Higgs bundle is the restriction to $\hb=0$ of the $\cR_\cX[\tm]$-module $\cE^{\varphi/\hb}$ introduced in \S\ref{subsec:exptwist}.

\subsubsection{Classification in arbitrary rank}
$\wt M_q=K_q\otimes_K\wt M$, equipped with the pull-back $\theta_q$ of $\theta$ as a differential form. If $A(t)$ is the matrix of $t\theta$ in some $K$-basis of $\wt M$, then $A_q(t_q)\defin qA(t_q^{q})$ is that of $t_q\theta_q$ in the corresponding $K_q$-basis of $\wt M_q$. For a suitable $q$, the eigenvalues of $A_q(t_q)$ exist in $K_q$, and $A_q(t_q)$ can be reduced to the Jordan normal form, so $(\wt M_q,t_q\theta_q)$ decomposes as a finite direct sum $\oplus_f(\cE^{\varphi/\hb}_{|\hb=0}\otimes \wt N_\varphi)$, where the matrix of $t_q\theta_q$ in some basis of $\wt N_\varphi$ is a constant matrix and $\varphi$ varies in a finite set of distinct meromorphic (but possibly holomorphic) functions. In particular, $(\wt M_q,t_q\theta_q)$ is an extension of rank-one objects. One can also consider, in a way analogous to that of meromorphic connections, a coarser (unique) decomposition
\begin{equation}\label{eq:decHiggs}
(\wt M_q,t_q\theta_q)\simeq\tbigoplus_\varphi(\cE^{\varphi/\hb}_{|\hb=0}\otimes \wt R_\varphi),
\end{equation}
where $\varphi$ varies in a finite set in $\tqm\CC[\tqm]$ and $\wt R_\varphi$ is regular.

Going back to $(\wt M,\theta)$, we find a unique decomposition
\begin{equation}\label{eq:regirreg}
(\wt M,\theta)=(\wt M,\theta)^\rr\oplus(\wt M,\theta)^\ir,
\end{equation}
coming from the decomposition of $(\wt M_q,\theta_q)$ into regular terms ($\varphi=0$), and purely irregular terms ($\varphi\neq0$). In fact, $(\wt M,\theta)^\rr$ is regular.

\begin{lemme}\label{lem:higgsreg}
Let $(\wt M,\theta)$ be a germ of meromorphic Higgs bundle, and let $U$ be any $V_0$-lattice. Then $\dim_\CC U/tU=\dim_K\wt M^\rr$.
\end{lemme}

\begin{proof}
Given two $V_0$-lattices $U,U'$ of $\wt M$, we consider the filtrations $U^\cbbullet=t^\cbbullet U$ and $U^{\prime\cbbullet}=t^\cbbullet U'$ of $\wt M$. Then, denoting by $U^{\prime\cbbullet}\gr^k_U\wt M$ the filtration induced by $U^{\prime\cbbullet}$ on $\gr^k_U\wt M$, we have
\[
\gr^p_{U'}\gr^0_U\wt M\xrightarrow[\ts\sim]{~\ts t^{-p}}\gr^0_{U'}\gr^{-p}_U\wt M\simeq\gr^{-p}_U\gr^0_{U'}\wt M
\]
and
\begin{align*}
\dim_\CC U/tU=\dim_\CC\gr^0_U\wt M&=\sum_p\dim_\CC\gr^p_{U'}\gr^0_U\wt M\\
&=\sum_p\dim_\CC\gr^{-p}_U\gr^0_{U'}\wt M=\dim_\CC U'/tU'.
\end{align*}
Therefore, to prove the lemma, we can take a $V_0$-lattice which decomposes according to the decomposition \eqref{eq:regirreg}.

Assume that $\wt M$ is purely irregular. We want to show that $\wt M$ has finite type as a $\CC\{t\}[t\theta]$-module. It is enough to show this after ramification, and we can reduce to the rank-one situation corresponding to a form $\omega$ with a pole of order $\geq2$. Then the polynomial $b(t\theta)$ introduced in the proof of Lemma \ref{lem:V0} is constant, so $U=tU$, and therefore $U=U[\tm]=\wt M$.

Assume now that $\wt M$ is regular. Then $U$ is a lattice, and the result is clear.
\end{proof}

\subsection{General properties of strict holonomic $\cR_\cX$-modules}

\begin{lemme}\label{lem:holstrict}
Let $\cM$ be a holonomic $\cR_\cX$-module. Then $\wt\cM\defin\cM[\tm]$ is coherent over $\cO_\cX[\tm]$. If moreover $\cM$ is strict, then, for any $\hbo\in\Omega_0$,
\begin{enumerate}
\item\label{lem:holstrict1}
$\wt\cM_\ohbo$ is free of finite rank over $\cO_{\cX,\ohbo}[\tm]$,
\item\label{lem:holstrict2}
$i^*_{\hbo}\wt\cM=\wt M_{\hbo}\defin M_{\hbo}[\tm]$,
\end{enumerate}
\end{lemme}

\begin{proof}
As $\cM$ is holonomic, its characteristic variety is contained in $(T^*_XX\cup T^*_0X)\times\Omega_0$ if $X$ is small enough. Locally near $\ohbo$, there exists thus an operator with principal symbol $t^j\partiall_t^k$ (for some $j,k$) which annihilates $\cM$. This gives the first assertion. Assume now that $\cM$ is strict.

\begin{enumerate}
\item
Away from $\{t=0\}$ (assuming $X$ small enough), we know by \cite[Prop\ptbl1.2.8]{Bibi01c} that $\cM$ is $\cO_{\cX^*}$-coherent and that it is locally free on $X^*\times\Omega_0^*$; moreover, it has no $\cO_{\cX^*}$-torsion, because any torsion section would be killed by some power of~$\hb$, in contradiction with strictness. As a consequence, $\wt\cM$ has no $\cO_\cX$-torsion.

Let us work locally near $\ohbo$. Let $\cN$ be a coherent $\cO_\cX$\nobreakdash-submodule of $\wt\cM$ such that $\wt\cM=\cN[\tm]$. As~$\cN$ has no $\cO_\cX$\nobreakdash-torsion, it locally free away from $\{t=0\}$ (if $X$ is small enough), and the kernel and the cokernel of the natural morphism $\cN\to\cN^{\vee\vee}$ (where $\cN^\vee=\cHom_{\cO_\cX}(\cN,\cO_\cX)$) are supported on $t=\nobreak0$. Let us also set ${\wt\cM}^\vee=\cHom_{\cO_\cX[\tm]}(\wt\cM,\cO_\cX[\tm])$. Then ${\wt\cM}^\vee=\cN^\vee[\tm]$ and therefore the natural morphism $\wt\cM\to\wt\cM^{\vee\vee}$ is an isomorphism. Consequently, $\cN^{\vee\vee}$ is (isomorphic to) a submodule $\cN'$ of $\wt\cM$. As $\cN^{\vee\vee}$ is reflexive, it is locally free (because $\cO_{\cX,\ohbo}$ is a regular local ring of dimension~$2$). We have hence found a locally free $\cO_\cX$-module $\cN'$ of $\wt\cM$ such that $\wt\cM=\cN'[\tm]$.

\item
As $\cM$ is strict, the sequence
\[
0\to\cM\To{(\hb-\hbo)}\cM\to i^*_{\hbo}\cM\to0
\]
is exact. It remains exact after tensoring with $\cO_\cX[\tm]$.\qedhere
\end{enumerate}
\end{proof}

In dimension one, we will mainly work with the localized modules $\wt\cM$. It is therefore useful to summarize the correspondence $\wt\cM\mto\wt\cM_{\min_t}$.

\begin{corollaire}
Let $\cM$ be a strict holonomic $\cR_\cX$-module which is strictly specializable. Then $\wt\cM$ is a locally free $\cO_\cX[\tm]$-module with a $\hb$-connection (\ie a compatible action of $\cR_\cX$) which is strictly specializable (\cf\S\ref{subsec:strictspe}). Conversely, starting from such a $\wt\cM$, the minimal extension $\wt\cM_{\min_t}$ across $\{t=0\}$ is strict holonomic and strictly S-decomposable (with only one strict component).\qed
\end{corollaire}

\begin{remarque}[Formal coefficients]
Let us denote by $\cO_{\wh\cX}$ the sheaf $\varprojlim_n\cO_\cX/t^n\cO_\cX$, and by $\cR_{\wh\cX}$ the sheaf $\cO_{\wh\cX}\otimes_{\cO_\cX}\cR_\cX$. These sheaves are supported on $\{t=0\}$. Then the previous results extend to holonomic $\cR_{\wh\cX}$-modules in a natural way.
\end{remarque}

\subsection{The regular case}

\begin{proposition}\label{prop:reg}
Let $\wt\cM$ be a strictly specializable $\cR_\cX[\tm]$-module which is $\cO_\cX[\tm]$-locally free of finite rank. The following properties are equivalent:
\begin{enumerate}
\item\label{prop:reg1}
there exists $\hbo\in\Omega_0$ such that $\wt M_{\hbo}$ is regular (as a meromorphic connection if $\hbo\neq0$, and as a meromorphic Higgs bundle if $\hbo=0$),
\item\label{prop:reg2}
for any $\hbo\in\Omega_0$, $\wt M_{\hbo}$ is regular,
\item\label{prop:reg3}
for some (or any) $\hbo\in\Omega_0$, any coherent $V_0\cR_{\cX,\ohbo}$-submodule of $\wt\cM_\ohbo$ generating $\wt\cM_\ohbo$ over $\cO_{\cX,\ohbo}[\tm]$ has finite type over $\cO_{\cX,\ohbo}$.
\item\label{prop:reg4}
$\wt\cM$ is also strictly specializable with ramification and exponential twist and, for any ramification $\rho_q:t_q\mto t=t^q$, any $\varphi\in\tqm\CC[\tqm]\moins\{0\}$, denoting $\wt\cM_q=\rho_q^+\cM$, we have $\psi_{t_q}^\beta(\cE^{-\varphi/\hb}\otimes\nobreak\wt\cM_q)=0$ for any $\beta\in\CC$.
\end{enumerate}
\end{proposition}

We say that $\wt\cM$ is \emph{regular} when one of these equivalent properties is satisfied.

\begin{remarques}\label{rem:reg}\mbox{}
\begin{enumerate}
\item\label{rem:reg1}
Property \ref{prop:reg}\eqref{prop:reg3} is taken as the definition of regularity in \cite[\S3.1.d]{Bibi01c}. Notice that, when applied to the canonical $V$-filtration $V_\phbo^\cbbullet\wt\cM$, it gives its $\cO_{\cX,\ohbo}$-freeness, as each $\gr^b_{V_\phbo}\wt\cM$ is $\cO_{\Omega_0,\hbo}$-free of finite rank, and the sum of the ranks, for $b\in(-1,0]$, is equal to the $\cO_\cX[\tm]$-rank of $\wt\cM$.

\item\label{rem:reg2}(Formal coefficients)
Proposition \ref{prop:reg} also applies to modules defined over~$\cR_{\wh\cX}[\tm]$.
\end{enumerate}
\end{remarques}

\begin{proof}
\eqref{prop:reg2}$\implique$\eqref{prop:reg1} is clear. For the converse, let us note that the assumption of strict specializability implies in particular that, for any~$\hbo$, if $V_\phbo^\cbbullet\wt\cM$ denotes the Kashiwara-Malgrange filtration of $\wt\cM$, then $V_\phbo^0\wt\cM/V_\phbo^1\wt\cM$ is a locally free $\cO_{\Omega_0,\hbo}$-module, whose rank does not depend on~$\hbo$. Under Assumption \eqref{prop:reg1}, its fibre at a given $\hbo$ has dimension equal to the rank of $\wt\cM$ (this is classical if $\hbo\neq0$, and follows from Lemma~\ref{lem:higgsreg} if $\hbo=0$), hence the same property holds for any $\hbo$. This, in turn, implies regularity of $\wt M_\hbo$ for any $\hbo$.

\eqref{prop:reg3}$\implique$\eqref{prop:reg1} is proved in the same way and \eqref{prop:reg4}$\implique$\eqref{prop:reg1} is easy.

\subsubsection*{\proofname\ of \eqref{prop:reg1}$\implique$\eqref{prop:reg3}}
Let $m\in\wt\cM_\ohbo$ and let $\wt\cN_\ohbo$ be the $\cR_\cX[\tm]$-submodule generated by $m$. It is enough to prove the property for any such $\wt\cN$. By Lemma \ref{lem:holstrict}\eqref{lem:holstrict1}, we know that $\wt\cN_\ohbo$ is $\cO_\cX[\tm]$-free, and we denote by $\delta+1$ its rank.

The family $\bmm\defin(m,\dots,(t\partiall_t)^\delta m)$ defines a $\cO_\cX$-linear morphism $(\cO_\cX[\tm])^{\delta+1}\to \wt\cN$ (in some neighbourhood of $\ohbo$). As the rank of $\wt\cN$ is $\delta+1$ and as $\wt\cN$ is generated by $m$ over $\cR_\cX[\tm]$, this morphism is injective. Let us denote by $\cC$ its cokernel. This is a coherent $\cO_\cX[\tm]$-module supported on some curve in $(X\times\Omega_0,\ohbo)$ not included in $\{t=0\}$, which is called the \emph{apparent singularity}. There exists thus $k\in\NN$ and $f(t,\hb)\in\cO_{\cX,\ohbo}$ not divisible by $t$ such that
\[
(\hb-\hbo)^kf(t,\hb)(t\partiall_t)^{\delta+1} m=\sum_{j=0}^\delta a_j(t,\hb)(t\partiall_t)^jm,\quad a_j(t,\hb)\in\cO_{\cX,\ohbo}[\tm]
\]
where $f(t,\hb)=0$ is an (possibly non reduced) equation of the components of the curve distinct from $\{\hb=\hbo\}$, and we assume that $k$ and $f$ are chosen so that none of the irreducible components of $(\hb-\hbo)^kf(t,\hb)$ divides all the $a_j$.

Reducing ${}\bmod(\hb-\hbo)$ implies that $k=0$, otherwise there would exist a non-trivial relation with coefficients in $\cO_{X,0}[\tm]$ between the classes of the elements of the family, and the rank of the submodule it generates would be $<\delta+1$, in contradiction with the freeness of $\wt\cN_\ohbo$.

As we assume that the restriction of $\wt\cN$ to any fixed $\hb$ has a regular singularity at $t=0$, the $a_j$ also belong to $\cO_{\cX,\ohbo}$ and thus the matrix of $t\partiall_t$ in the family $\bmm$ has entries in $\cO_\cX[1/f]$.

The next step will consist in eliminating the apparent singularity, which is clearly an obstruction for proving the $\cO_{\cX}$-coherence of the $V_0\cR_\cX[\tm]$-submodule generated by $m$.

Let $\bmn$ be any $\cO_{\cX,\ohbo}[\tm]$-basis of $\wt\cN_\ohbo$. We have $\bmm=\bmn\cdot S$ where the matrix $S$ has entries in $\cO_{\cX,\ohbo}[\tm]$ but is invertible only after inverting $f$, \ie $S\in\GL_{\delta+1}(\cO_{\cX,\ohbo}[\tm,f^{-1}])$.

\begin{lemme}[\cite{Malgrange91c}]
For $S$ as above, there exist two matrices
\[
S_1\in\nobreak \GL_{\delta+1}(\cO_{\cX,\ohbo}[\tm])\quad\text{and}\quad S_2\in\GL_{\delta+1}(\cO_{\cX,\ohbo}[f^{-1}])
\]
such that $S=S_1S_2$.\qed
\end{lemme}

Let us thus set $\mug=\bmn\cdot S_1$. Then $\mug$ is another $\cO_{\cX,\ohbo}[\tm]$-basis of $\wt\cN_\ohbo$ and the matrix of $t\partiall_t$ in the basis $\mug$ has pole along $\{t=0\}$ at most, as it is so for any $\cO_{\cX,\ohbo}[\tm]$-basis of $\wt\cN_\ohbo$. On the other hand, as $\mug=\bmm\cdot S_2^{-1}$, the matrix of $t\partiall_t$ in the basis $\mug$ has poles at most along $\{f=0\}$. We conclude that the matrix of $t\partiall_t$ in the basis $\mug$ has entries in $\cO_{\cX,\ohbo}$. It follows that the $V_0\cR_{\cX,\ohbo}[\tm]$-submodule generated by $\mug$ is finite over $\cO_{\cX,\ohbo}$, hence so is any finitely generated $V_0\cR_{\cX,\ohbo}[\tm]$-submodule of $\wt\cN_\ohbo$, by a standard argument.

\subsubsection*{\proofname\ of \eqref{prop:reg3}$\implique$\eqref{prop:reg4}}
Firstly, working with $m$ and $\wt\cN_\ohbo$ as above, we get $t\partiall_t\ephi\otimes\mug=\ephi\otimes\mug\cdot (t\varphi'(t)\id+A(t,\hb))$, where $A$ has entries in $\cO_{\cX,\ohbo}$. This can be rewritten as $\ephi\otimes\mug=\ephi\otimes\mug\cdot tB(t,\hb)+(t\partiall_t\ephi\otimes\mug)\cdot tC(t,\hb)$, where $B,C$ have entries in $\cO_{\cX,\ohbo}$. It follows that the $V_0\cR_{\cX,\ohbo}$-submodule generated by $\ephi\otimes\mug$ is equal to $\cE^{-\varphi/\hb}\otimes\wt\cN_\ohbo$. Therefore, there exists a good $V$-filtration which is constant and equal to $\cE^{-\varphi/\hb}\otimes\wt\cN_\ohbo$. It satisfies all the properties characterizing the Kashiwara-Malgrange filtration, and all graded pieces are zero. It follows easily that $\cE^{-\varphi/\hb}\otimes\wt\cM_\ohbo$ satisfies the same properties.

It remains to checking that strict specializability is preserved by ramification, as regularity will be so, applying for instance Property \eqref{prop:reg1}. This is Proposition \ref{prop:strictsperamif}.
\end{proof}

In the following, we will set $\cX^\circ=X\times(\Delta_0\moins\{0\})$.

\begin{proposition}[Regular {$\cR_{\wh\cX}[\tm]$}-modules on $\hb\neq0$]
\label{prop:structreg}
Let $\wt\cM^\wedge$ be a locally free $\cO_{\wh\cX{}^\circ}[\tm]$-module with a compatible $\cR_{\wh\cX{}^\circ}$-action, which is strictly specializable and regular (at $t=0$). Then there exists a strictly specializable regular $\cR_{\cX^\circ}[\tm]$-module $\wt\cR$ such that $\wt\cM^\wedge$ is isomorphic to $\cO_{\wh\cX{}^\circ}\otimes_{\cO_{\cX^\circ}}\wt\cR$.
\end{proposition}

In the following proof, we always assume that $\hbo\neq0$. We will first consider the local situation (w.r.t.\ $\hb$). We will first recall \cite[Remark 5.3.8(3)]{Bibi01c}:

\begin{lemme}\label{lem:localreg}
For any strictly specializable regular $\wt\cM_\ohbo$ or $\wt\cM_\ohbo^\wedge$, there exists a basis in which the matrix of $t\partiall_t$ only depends (holomorphically) on~$\hb$ and is lower-triangular, with eigenvalues of the form $\beta\star\hb$.
\end{lemme}

\begin{proof}
We give the proof in the convergent setting, but it is also valid in the formal setting.

Let us fix $\hbo\neq0$ and let us set\footnote{Recall that the notation $\ell_\hbo(\beta)$ is for $\reel\beta-(\im\beta)(\im\hbo)$, \cf \S\ref{subsec:strictspe}.}
\[
B_{\hbo}=\{\beta\in\CC\mid\ell_\hbo(\beta)\in(-1,0]\text{ and } \psi_t^\beta\wt\cM_\ohbo\neq0\}.
\]
For each $\beta\in B_\hbo$, let us fix a basis of $\psi_t^\beta\wt\cM_\hbo$ for which the matrix of $-\rN$ is constant. Let us lift this basis modulo $V^{>\ell_\hbo(\beta)}_\phbo\wt\cM_\ohbo$. Applying this to any $\beta\in B_\hbo$, we find a basis of $\wt\cM_\ohbo$ for which the matrix of $t\partiall_t$ has entries in $\cO_{\cX,\ohbo}$ and its constant part (with respect to $t$) $R(\hb)$ has the following form: it is block lower triangular, each block corresponding to a value $b$ of $\ell_\hbo(\beta)$, $\beta\in B_\hbo$; moreover, each diagonal block, with corresponding value $b$, is itself block-diagonal, each block corresponding to one $\beta\in B_\hbo$ such that $\ell_\hbo(\beta)=b$, and takes the form $(\beta\star\hb)\id+\rY_\beta$, where $\rY_\beta$ is constant (we can choose the basis such that $\rY_\beta$ has the Jordan normal form).

By a constant (with respect to $t$) base change, we can assume that, if $\beta_1\star\hbo\neq\beta_2\star\hbo$, then the $(\beta_1,\beta_2)$ and the $(\beta_2,\beta_1)$ blocks in the matrix $R(\hb)$ are zero. Therefore, we can reorder the basis in such a way that $R(\hb)$ is block-diagonal, each diagonal block corresponding to a value of $\beta\star\hbo$, $\beta\in B_\hbo$, and each diagonal block is itself lower triangular with respect to the $\ell_\hbo$-order, each diagonal sub-block corresponding to a given $\beta$ and having the form $(\beta\star\hb)\id+\rY_\beta$.

Let us set $\Lambda=B_\hbo+\ZZ$ and let $\Sing\Lambda$ be the set of $\hb\in\Omega_0$ such that there exist $\beta_1\neq\beta_2\in\Lambda\cup\ZZ$ with $\beta_1\star\hb=\beta_2\star\hb$. Then, \cf \cite[p\ptbl17]{Bibi01c}, $\Sing\Lambda$ is contained in~$i\RR$ and has $0$ as its only limit point.

If $\hbo\not\in\Sing\Lambda$ then, for any $\beta_1,\beta_2\in B_\hbo$, we have $\beta_1\star\hbo-\beta_2\star\hbo\not\in\hbo\ZZ^*$, hence, for any $\hb\in\nb(\hbo)$, we also have $\beta_1\star\hb-\beta_2\star\hb\not\in\hb\ZZ^*$. Therefore, the classical arguments of the theory of regular meromorphic differential equations enable us to find a basis of $\wt\cM_\ohbo$ in which the matrix of $t\partiall_t$ is $R(\hb)$.

If $\hbo\in\Sing\Lambda$ (but $\hbo\neq0$), there may exist $\beta_1,\beta_2\in B_\hbo$ with $\beta_1\star\hbo-\beta_2\star\hbo\in\hbo\ZZ^*$. We first apply successive partial rescalings by powers of $t$ to reduce to the case where the constant part $R(\hb)$ (with respect to $t$) of the matrix of $t\partiall_t$ is lower-triangular, and has eigenvalues $\beta\star\hb$ with $\beta\in\Lambda$ and no two eigenvalues differ by an element in $\hb\ZZ^*$. Let us recall this classical reduction:

Let us index by $I$ the set of diagonal blocks of $R(\hb)$, each block corresponding to a value $\gamma_i(\hbo)$ of various $\beta\star\hb$ at $\hbo$. We can order the set $I$ in such a way that, if $\gamma_{i_1}(\hbo)-\gamma_{i_2}(\hbo)\in\hbo\NN^*$, then $i_1<i_2$. We now only retain that the matrix $R(\hb)$ is lower-triangular. Let $i_0\in I$ be such that there exist no $i\in I$ with $\gamma_{i_0}(\hbo)-\gamma_i(\hbo)\in\hbo\NN^*$ (that is, if such a value belongs to $\hbo\ZZ^*$, it must belong to $-\hbo\NN^*$). We apply the base change with diagonal matrix having diagonal blocks $\Id_i$ for $i\neq i_0$ and $t\Id_{i_0}$. After this change, the constant part $R_1(\hb)$ of the matrix of $t\partiall_t$ remains lower-triangular, with the same eigenvalues, except that $\beta\star\hb$ is changed to $(\beta+1)\star\hb$ if $\beta\star\hbo=\gamma_{i_0}(\hbo)$. By an easy induction, we get the assertion.

We can now apply the classical arguments of the theory of regular meromorphic differential equations to find a basis of $\wt\cM_\ohbo$ in which the matrix of $t\partiall_t$ is equal to the previous constant part $R(\hb)$ (with respect to $t$). In particular, it is lower triangular and has eigenvalues of the form $\beta\star\hb$, so we get the assertion of the lemma in the case $\hbo\in\Sing\Lambda$.
\end{proof}

We then have, in the same way, by applying the classical reduction theory with a parameter $\hb$,

\begin{lemme}\label{lem:localreginv}
Near any $\hbo\neq0$, there exists a polynomial $p(\hb)$ which vanishes at most on $\Sing\Lambda\cap\nb(\hbo)$ and an invertible matrix with entries in $\cO_{\Omega_0,\hbo}[1/p(\hb)]$ such that, after the base change given by this matrix, the matrix of $t\partiall_t$ takes the form
\[
\oplus_{\beta\in B_\hbo}\big[(\beta\star\hb)\id+\rY_\beta\big]
\]
where $\rY_\beta$ is a lower triangular nilpotent Jordan matrix. In particular, if we assume that $\hbo\in\bS$, then we can choose $p(\hb)=\hb+1/\hb$.
\end{lemme}

\begin{proof}
In the previous proof, instead of applying shearing transformations near a point $\hbo\in\Sing\Lambda$, we can apply the same arguments as in the case where $\hbo\not\in\Sing\Lambda$. In order to do so, we have to invert near $\hbo$ all functions $p+(\beta_1-\beta_2)\star\hb/\hb$, with $p\in\NN^*$ and $\beta_1,\beta_2\in B_\hbo$. On a fixed neighbourhood $\nb(\hbo)$, only a finite number of such functions vanish at $\hbo$.
\end{proof}

\begin{proof}[\proofname\ of Proposition \ref{prop:structreg}]
We have proved the local existence of $\wt\cR$ in Lemma \ref{lem:localreg}. Given two local bases of $\wt\cM^\wedge_\phbo$ in which the matrix of $t\partiall_t$ has entries in $\cO_{\Omega_0,\hbo}$, any formal base change between both bases has entries in $\cO_{\cX,\ohbo}$ (this is standard, as we work away from $\hb=0$). Therefore, the $\cR_{\cX,\ohbo}[\tm]$-submodule generated by such a local basis is unique, hence can be glued all along $\Omega_0\moins\{0\}$.
\end{proof}

\subsection{Strict specializability with ramification and exponential twist in dimension one}

Let us fix $\hbo\in\Omega_0$ and let $\wt M_\hbo$ be a meromorphic connection or Higgs bundle. Then, there exist $q=q(\hbo)\in\NN^*$ and, denoting by $\wt M_{q,\hbo}$ the pull-back of $\wt M_\hbo$ by the ramification $t_q\mto t=t_q^q$, a finite family $(\varphi_j)_{j\in J(\hbo)}$ of elements of $\tqm\CC[\tqm]$ which satisfy the following properties:
\begin{enumerate}
\item
for some, or any, $V$-lattice $U_{\hbo}$ of $\cE^{-\varphi_j/\hbo}\!\otimes\!\wt M_{q,\hbo}$, $U_{\hbo}/tU_{\hbo}\!\neq\!0$,
\item
for any ramification $t_{rq}\mto t_q=t_{rq}^r$ and any $\varphi\in\trqm\CC[\trqm]$ distinct from any $\varphi_j(t_{rq}^r)$, then $\cE^{-\varphi/\hbo}\otimes\wt M_{rq,\hbo}$ (where $\wt M_{rq,\hbo}$ is the pull-back of $\wt M_{q,\hbo}$) is a $V$-lattice of itself.
\end{enumerate}

Let now $\cM$ be a strict holonomic $\cR_\cX$-module. Then, according to Lemma \ref{lem:holstrict}\eqref{lem:holstrict2}, restricting to $\hb=\hbo$ commutes with localizing away from $\{t=0\}$, so the previous statement applies to the restriction to $\hbo$ of $\wt\cM$.

\begin{proposition}\label{prop:constance}
Let $\wt\cM$ be a $\cR_\cX[\tm]$-module which is $\cO_\cX[\tm]$-locally free and strictly specializable with ramification and exponential twist on~$X$. Then the (smallest) integer $q(\hbo)$ and the set $J(\hbo)$ defined above do not depend on~$\hbo$. We denote by~$q$ and~$J$ their constant value. Then, denoting by $\rho_q:X_q\to X$ the ramification $t_q\mto t=\nobreak t_q^{q}$, and setting $\wt\cM_q=\rho_q^+\wt\cM$,
\begin{enumerate}
\item
for any $j\in J$, there exists $\beta$ with $\reel\beta\in(-1,0]$ such that $\psi_{t_q}^\beta(\cE^{-\varphi_j/\hb}\otimes\nobreak\wt\cM_q)\!\neq\nobreak\!\nobreak0$,
\item
for any ramification mapping $t_{rq}\mto t_q=t_{rq}^r$ and any $\varphi\in\trqm\CC[\trqm]\moins\{\varphi_j(t_{rq}^r)\mid j\in J\}$, for any $\beta$ with $\reel\beta\in(-1,0]$, $\psi_{t_{rq}}^\beta(\cE^{-\varphi/\hb}\otimes\nobreak\wt\cM_{rq})=0$.
\end{enumerate}
\end{proposition}

\begin{proof}
Let us set for convenience $q=q(1)$ and $J=J(1)$. By assumption, for any $r\geq1$, $\psi_{t_{rq}}^\beta(\cE^{-\varphi/\hb}\otimes\wt\cM_{rq})$ is locally free as a $\cO_{\Omega_0}$-module. Therefore, its restriction at some $\hb=\hbo$ vanishes (\resp does not vanish) as soon as its restriction at $\hb=1$ does so. The restriction to $\hb=\hbo$ of $\cE^{-\varphi/\hb}\otimes\wt\cM_{rq}$ is $\cE^{-\varphi/\hbo}\otimes\wt M_{rq,\hbo}$ for any $\hbo$, after Lemma \ref{lem:holstrict}\eqref{lem:holstrict2}. According to \cite[Prop\ptbl3.3.14]{Bibi01c}, the restriction to $\hb=\hbo$ of $\psi_{t_{rq}}^\beta(\cE^{-\varphi/\hb}\otimes\wt\cM_{rq})$ is a direct summand of a graded piece of $\cE^{-\varphi/\hbo}\otimes\wt M_{rq,\hbo}$ with respect to a suitable good $V$-filtration. It is then easily seen that the nonvanishing of such a graded piece is equivalent to the existence of a regular part in the formal decomposition of $\cE^{-\varphi/\hbo}\otimes\wt M_{rq,\hbo}$. The proposition follows.
\end{proof}

\begin{remarques}[Formal coefficients]\label{rem:formal}\mbox{}
\begin{enumerate}
\item\label{rem:formal1}
The same result applies to $\cR_{\wh\cX}$-modules satisfying analogous assumptions.
\item\label{rem:formal2}
Notice also that, if $\wt\cM^\wedge=\cO_{\wh\cX}\otimes_{\cO_\cX}\wt\cM$ is the associated formalized module, then $\wt\cM$ satisfies the assumption of the proposition if and only if $\wt\cM^\wedge$ does so. Moreover, for any ramification, we have $(\wt\cM_q)^\wedge=(\wt\cM^\wedge)_q$ and for any $\beta\in\CC$ and $\varphi\in \tqm\CC[\tqm]$, we have
\[
\psi_{t_q}^\beta(\cE^{-\varphi/\hb}\otimes\wt\cM_q) =\psi_{t_q}^\beta(\cE^{-\varphi/\hb}\otimes(\wt\cM^\wedge)_q).
\]
Indeed, ``only if'' is clear by flatness of $\cO_{\wh\cX}$ over $\cO_\cX$. For the ``if'' part, we note that if $U^\cbbullet\wt\cM$ is a good $V$-filtration of $\wt\cM$, then $U^\bbullet\wt\cM^\wedge\defin\cO_{\wh\cX}\otimes_{\cO_\cX}U^\cbbullet\wt\cM$ is a good $V$-filtration of $\wt\cM^\wedge$. Hence $\wt\cM$ is specializable (in the sense of \cite[Def\ptbl3.3.1]{Bibi01c}) if $\wt\cM^\wedge$ is so. By standard manipulation, we can assume that (near $\hbo$) the Bernstein polynomial of the good filtration $U^\bbullet\wt\cM$ has roots in $\Lambda(\hbo)$ (see \cite[p\ptbl67]{Bibi01c}), and therefore so does $U^\bbullet\wt\cM^\wedge$, which is then the canonical $V$-filtration of $\wt\cM^\wedge$, by \cite[Lemma 3.3.4]{Bibi01c}. In particular, the graded pieces are strict by assumption, hence so are the graded pieces of $U^\bbullet\wt\cM$. This implies that $\wt\cM$ is strictly specializable, with $U^\bbullet\wt\cM$ as canonical good $V$-filtration. The remaining part of the assertion is then easy.
\end{enumerate}
\end{remarques}

\subsubsection{Submodules}
Let $\wt\cM$ be as in Proposition \ref{prop:constance}. Let $m\in\wt\cM_\ohbo$ and let $\wt\cN=\sum_k\cO_\cX[\tm]\partiall_t^km$ be the $\cR_\cX[\tm]$-submodule generated by $m$ (in some small neighbourhood of $\ohbo$). Arguing as in Lemma \ref{lem:holstrict}\eqref{lem:holstrict1}, we see that $\wt\cN$ is $\cO_\cX[\tm]$-free in some neighbourhood of $\ohbo$. Let $\delta+1$ denote its rank.

\begin{lemme}\label{lem:submodules}
The $\cR_\cX[\tm]$-module $\wt\cN$ is strictly specializable with ramification and exponential twist and for any ramification $t_r\mto\nobreak t$, any $\varphi\in \trm\CC[\trm]$ and any $\beta\in\CC$, we have $\psi_{t_r}^\beta(\cE^{-\varphi/\hb}\otimes\wt\cN_r)\subset\psi_{t_r}^\beta(\cE^{-\varphi/\hb}\otimes\wt\cM_r)$.
\end{lemme}

\begin{proof}
As the ramified twisted module $\cE^{-\varphi/\hb}\otimes\wt\cN_r$ is a submodule of $\cE^{-\varphi/\hb}\otimes\wt\cM_r$, it is enough to show that the good $V$-filtration of the latter defined near $\hbo$ (\cf \cite[\S3.4.a]{Bibi01c}) induces a good $V$-filtration of the former near $\hbo$. For the sake of simplicity, we will give the proof for $\wt\cN\subset\wt\cM$, but it will apply similarly to ramified twisted modules.

If we have shown that $V^\cbbullet_\phbo\wt\cN\defin V^\cbbullet_\phbo\wt\cM\cap\wt\cN$ is a good $V$-filtration of $\wt\cN$, then each $\gr^b_{V_\phbo}\wt\cN$ is strict, as a submodule of a strict module, and it has a decomposition as $\oplus_\beta\psi_t^\beta\wt\cN$ (\cf \cite[Lemma~3.3.4]{Bibi01c}), showing that each $\psi_t^\beta\wt\cN$ is itself strict. By construction, we have $t^k\psi_t^\beta\wt\cN=\psi_t^{\beta+k}\wt\cN$ (as this is true for $\wt\cM$ in the place of $\wt\cN$), hence the strict specializability of $\wt\cN$ along $\{t=0\}$.

It remains to showing that, for any $b\in\RR$, $V^b_\phbo\wt\cM\cap\cN$ is $V_0(\cR_\cX)$-coherent. This is standard.
\end{proof}

\subsubsection{Formal structure}
Let $\wt\cM$ be as in Proposition \ref{prop:constance}. We fix the order of ramification as in the proposition, and consider $\wt\cM_q=\rho_q^+\wt\cM$.

\begin{proposition}\label{prop:structformelle}
With these assumptions, for any $j\in J$, there exist regular strictly specializable $\cO_{\wh{\cX_q}}[\tqm]$-modules with $\hb$-connection $\wt\cR_j^\wedge$ and an isomorphism of $\cR_{\wh{\cX_q}}[\tqm]$-modules
\[\tag{DEC$^\wedge$}
\wt\cM_q^\wedge\isom \tbigoplus_{j}(\cE^{\varphi_j/\hb}\otimes\wt\cR_j^\wedge).
\]
\end{proposition}

\begin{remarque}
By uniqueness, the restriction to $\hb=\hbo$ of (DEC$^\wedge$) is nothing but the formal decomposition \eqref{eq:decomp} if $\hbo\neq0$ or that obtained from \eqref{eq:decHiggs} by tensoring with $\cO_{\wh X,0}$ if $\hbo=0$.
\end{remarque}

\begin{proof}
It will be convenient to change notation and forget the $_q$ everywhere, so the coordinate is denoted by $t$, etc. Moreover, we will only work with $\cO_{\wh\cX}[\tm]$-modules, so we will also forget the exponent ${}^\wedge$.

We first prove the local existence of $\wt\cR^\wedge_{j,\ohbo}$ and of the local decomposition (DEC$^\wedge$). We then prove that this decomposition is globally defined along the $\hb$-variable.

\subsubsection*{Local decomposition}
We argue by induction on the rank of $\wt\cM$. Assume moreover that we have proved the existence of a decomposition (DEC$^\wedge$) for any submodule $\wt\cN_\ohbo$ of $\wt\cM_\ohbo$ which is generated by one section~$m$ as in Lemma \ref{lem:submodules}. Then, as there is no nonzero morphism from $\cE^{\varphi/\hb}\otimes\wt\cR_{\varphi,\ohbo}^\wedge$ to $\cE^{\psi/\hb}\otimes\wt\cR_{\psi,\ohbo}^\wedge$ if $\psi\neq\varphi$, and as
$\wt\cM_\ohbo$ is the sum of a finite number of such submodules $\wt\cN_\ohbo$, the decomposition (DEC$^\wedge$) holds for $\wt\cM_\ohbo$.

Let $m\in\wt\cM_\ohbo$ and let $\wt\cN_\ohbo$ be as in Lemma \ref{lem:submodules}. We are thus reduced to proving the decomposition (DEC$^\wedge$) for $\wt\cN_\ohbo$.

Let us denote by $k$ the maximal order of the pole of the family $(\varphi_j)_j$ (associated to $\wt\cN_\ohbo$ by Proposition \ref{prop:constance}). We can assume that there exist two indices $j$ such that the corresponding $\varphi_j$ have a pole of order~$k$ but with distinct coefficients of $t^{-k}$ (otherwise, one would first tensor $\wt\cN$ with $\cE^{-ct^{-k}/\hb}$).

We now argue exactly as in the regular case (Prop\ptbl\ref{prop:reg}, proof of \eqref{prop:reg1}$\implique$\eqref{prop:reg3}), except that we replace $t\partiall_t$ with $t^k(t\partiall_t)$. We find a $\cO_{\wh\cX,\ohbo}$-basis $\mug$ of $\wt\cN_\ohbo$ in which the matrix of $t^k(t\partiall_t)$ has entries in $\cO_{\wh\cX,\ohbo}$. The constant (with respect to~$t$) part of this matrix has eigenvalues in $\CC$ (the coefficients of $t^{-k}$ in the $\varphi_j$) and at least two of them are distinct. A classical argument gives a splitting of $\wt\cN_\ohbo$ corresponding to these eigenvalues. We conclude by the inductive assumption.

\subsubsection*{Globalization}
We will show the existence of a global formal decomposition $\wt\cM=\wt\cM_\reg\oplus\wt\cM_\irr$. If $j_0$ is such that $\varphi_{j_0}=0$, we will set $\wt\cR_{j_0}^\wedge\defin\wt\cM_\reg$. In order to get the other terms of the decomposition, we apply the same argument to any $\cE^{-\varphi_j/\hb}\otimes\wt\cM$.

Given any local good $V$-filtration $U_\bbullet\wt\cM_\ohbo$, we set $\wt\cM_{\irr,\ohbo}=\bigcap_k U_k\wt\cM_\ohbo$. Any two such filtrations define the same $\wt\cM_{\irr,\ohbo}$, hence $\wt\cM_\irr$ is globally defined. Moreover, using the local decomposition (DEC$^\wedge$), one easily checks that $\wt\cM_{\irr,\ohbo}$ corresponds to the sum indexed by $j\neq j_0$. Hence $\wt\cM_\irr$ is a locally free $\wt\cO_{\wh\cX}[\tm]$-module of finite rank. The globally defined exact sequence
\[
0\to\wt\cM_\irr\to\wt\cM\to\wt\cM/\wt\cM_\irr\to0
\]
has local splittings, according to the local (DEC$^\wedge$). However, any two local splittings coincide, as there is no nonzero morphism $\wt\cR^\wedge_{j_0,\ohbo}\to\cE^{\varphi_j/\hb}\otimes\wt\cR^\wedge_{j,\ohbo}$ for $j\neq j_0$. Therefore, there exists a unique global splitting.
\end{proof}

It will be important to lift this result to convergent isomorphisms. However, as in the theory of meromorphic differential equations, one can only expect sectorial (with respect to $t$) liftings, when $\hbo\neq0$. We will also provide holomorphic liftings when $\hbo=0$. There, a sectorial partition with respect to $\hb$ is also needed.

\subsubsection{Sectorial structure ($\hbo\neq0$)}\label{subsec:sectneq0}
Let $\cM$ be as in Proposition \ref{prop:constance}. We also assume, in order to simplify the notation, that no ramification is needed, so that we can apply Proposition \ref{prop:structformelle} to $\wt\cM^\wedge$ and we have an isomorphism
\[\tag{DEC$^\wedge$}
\wt\cM^\wedge\isom \tbigoplus_{\varphi}(\cE^{\varphi/\hb}\otimes\wt\cR_\varphi^\wedge).
\]
Let us first work away from $\hb=0$. Then, by Proposition \ref{prop:structreg}, $\wt\cR_\varphi^\wedge$ is defined over $\cR_\cX$, and thus we can write the previous decomposition as
\begin{equation}\label{eq:decompform}
\wt\cM^\wedge\isom \tbigoplus_{\varphi}(\cE^{\varphi/\hb}\otimes\wt\cR_\varphi)^\wedge.
\end{equation}

Let $Y$ denote the real blow-up of the disc $X$ of radius $r_0$, \ie $Y=[0,r_0)\times S^1$ and let $e:Y\to X$ denote the projection $(r,e^{i\theta})\mto re^{i\theta}$. Let us set $\cY=Y\times\Omega_0$ and let us define the sheaves $\cA_\cY$, $\cA_\cY^{<0}$ and $\cCh{\cY}$ as usual: $\cCh{\cY}$ denotes the sheaf of $C^\infty$ functions on~$\cY$ which are holomorphic with respect to $\hb$, $\cA_\cY$ is the kernel of $\ov{t\partial_t}$ acting on $\cCh{\cY}$ and $\cA_\cY^{<0}$ denotes the subsheaf of $\cA_\cY$ consisting of germs having all their derivatives with respect to $t\partial_t$ vanishing on $\{r=0\}$.

It will be convenient to set
\[
\wt\cM^\sA_\othbo\defin\cA_{\cY,\othbo}\otimes_{\cO_{\cX,\ohbo}}\wt\cM_\ohbo.
\]
We refer to the literature (\eg \cite{Sibuya62}, \cite{Sibuya90bis}, \cite{Malgrange91}) for the results concerning the Hukuhara-Turrittin theorem with a parameter, as well as for the inductive process giving the $\cA$-decomposition.

\begin{proposition}\label{prop:decompA}
With the previous assumptions (and $\hbo\neq0$), for any $\theta_o\in S^1$, the decomposition \eqref{eq:decompform} can be lifted to a decomposition
\[\tag{DEC$^\sA$}
\wt\cM^\sA_\othbo\isom \tbigoplus_{\varphi}\Big[\cA_{\cY,\othbo}\otimes_{\cO_{\cX,\ohbo}}(\cE^{\varphi/\hb}\otimes\wt\cR_\varphi)_\ohbo\Big].
\]
\end{proposition}

\begin{proof}[Sketch of proof]
As we work near $\hbo\neq0$, we can regard $\wt\cM$ as an analytic family of meromorphic connections parametrized by $\nb(\hbo)$. The proof is then analogous to that for meromorphic connections (without parameter $\hb$), as the exponential factors~$\varphi_j$ are independent of~$\hb$. One proceeds exactly as for the theorem of Hukuhara-Turrittin in order to show that the decomposition (DEC$^\wedge$) can be locally lifted to $\wt\cM^\sA_\othbo$. At the end of the process, we get a basis of $\wt\cM^\sA_{\othbo}$ in which the matrix of $t\partiall_t$ is the sum of a matrix $R(\hb)$ depending only on~$\hb$ and a matrix having entries in $\cA^{<0}_{\cY,\othbo}$ (according to Lemma \ref{lem:localreg}). Another application of the same kind of existence theorems gives a base change of the form $\id+Q(t,\hb)$ with~$Q$ having entries in $\cA^{<0}_{\cY,\othbo}$, after which the matrix of $t\partiall_t$ is equal to $R(\hb)$.
\end{proof}

\subsubsection{Sectorial structure ($\hbo=0$)}\label{subsec:sect0}
When $\hbo=0$, the lifting of the decomposition and the definition of $\partiall_t=\hb\partial_t$ lead us to solve successive quasi-linear differential systems like
\[
\hb t^k(t\partial_t)u=L\cdot u+F(u,t,\hb)
\]
where $L$ is diagonal with constant nonzero eigenvalues and where we have information on~$F$ (existence of asymptotic expansions in sectorial domains). Here, $\hb$ plays the role of a small parameter, and taking sectors with respect to $\hb$ as well as for $t$ will be needed. In this direction, the main results have been previously proved by Russell and Sibuya \cite{R-S66,R-S68} ($k=0$), Sibuya \cite{Sibuya68} and then by Majima \cite{Majima84} ($k\geq1$). An account of the results of Majima has been given in \cite[Appendix]{Bibi93}.

We denote by $\cZ$ the space $Y\times \wt\Omega_0$, where $\wt\Omega_0=\RR_+\times S^1$ is the real blow-up (polar coordinates) of $\Omega_0$ at $0$. We will set $\hb=\module\hb\cdot e^{i\arghb}$. We have a natural map $\cZ\to\cY=Y\times\Omega_0$. We denote by $\cC^\infty_\cZ$ the sheaf of $C^\infty$ functions on the manifold with corners $\cZ$, and by $\cA_\cZ$ the subsheaf defined by the Cauchy-Riemann equations with respect to $t$ and $\hb$. In particular, when restricting to $\module\hb\neq0$, we recover the restriction of $\cA_{\cY}$ to this open set.

The next result will not be used in this article, but we give it for the sake of being complete.

\begin{proposition}\label{prop:decompA0}
With the assumptions of \S\ref{subsec:sectneq0}, for any $(\theta_o,\arghbo)\in S^1\times S^1$, the decomposition \textup{(DEC$^\wedge$)} can be lifted to a decomposition
\[\tag{DEC$^\sA_0$}
\wt\cM^\sA_\otahbo\isom \tbigoplus_{\varphi}\big[\cE^{\varphi/\hb}\otimes\wt\cR_{\varphi,\otahbo}^\sA\big]
\]
for some regular $\cA_{\cZ,\otahbo}[\tm]$-free modules $\wt\cR_{\varphi,\otahbo}^\sA$, in such a way that, modulo~$\hb$, this isomorphism is equal to the isomorphism \eqref{eq:decHiggs} (hence does neither depend on~$\theta_o$ nor on~$\arghbo$) and in particular
\[
\wt\cR_{\varphi,\otahbo}^\sA/\hb\wt\cR_{\varphi,\otahbo}^\sA\simeq \cA_{Y,(0,\theta_0)}\otimes_{\cO_{X,0}}\wt R_\varphi.
\]
\end{proposition}

As above, we use the notation
\[
\wt\cM^\sA_\otahbo=\cA_{\cZ,\otahbo}\otimes_{\cO_{\cX,\oO}}\wt\cM_\oO.
\]

\begin{proof}
Let us denote by $\cA_{\wh\cZ}$ the formal completion $\varprojlim_{k}\cA_\cZ/t^k\cA_\cZ$ (formal series in $t$ with coefficients in $\cA_{\wt\Omega_0}$). By a result of Majima \cite{Majima84} (see also \cite[Prop\ptbl1.1.16, p\ptbl44]{Bibi93}), which is a variant of the lemma of Borel-Ritt, we have an exact sequence
\[
0\to\cA_\cZ^{<0}\to\cA_\cZ\To{T}\cA_{\wh\cZ}\to0.
\]

The proof of the proposition is now a variant of a theorem of Majima, generalizing the theorem of Hukuhara-Turrittin in dimension $\geq1$. We assume $J\neq\emptyset$, otherwise there is nothing to prove. Let us denote by $J_{\max}$ the set of $j\in J$ such that the order of the pole of $\varphi_j$ is maximal. Let $J_0\subset J_{\max}$ be a subset corresponding to a fixed principal part of $\varphi_j$, $j\in J_{\max}$. We argue by induction on the rank of $\wt\cM^\sA_\otahbo$. The inductive assumption applies to free $\cA_{\cZ,\otahbo}[\tm]$-modules equipped with a compatible action of $\partiall_t$: we assume that, when tensored with $\cA_{\wh\cZ,\otoO}$, the module is isomorphic to some elementary model $\wt\cM^{\sA,\el}_{\otoO}$ as in the RHS of (DEC$^\cA_0$) and, when restricted to $\hb=0$, is decomposable.

There exists a basis of $\wt\cM^\sA_\otahbo$ in which the matrix of $t\partiall_t$ takes the form $\begin{smallpmatrix}A_{11}&\hb A_{12}\\ \hb A_{21}&A_{22}\end{smallpmatrix}$ such that $\wh A_{12}=0$, $\wh A_{21}=0$ and $\wh A_{11}$ (\resp $\wh A_{22}$) has diagonal blocs of the kind $t\varphi'_j+\wh B_j(t,\hb)$ with $j\in J_0$ (\resp $j\not\in J_0$), where $\wh B_j$ has entries in $\cA_{\wh\cZ}$ (no pole along $t=0$ or $\hb=0$).

In order to diagonalize this matrix, we have to find matrices $Q_{12}$ and $Q_{21}$ with entries in $\cA_{\cZ,\otahbo}$, such that $\wh Q_{12}=0$ and $\wh Q_{21}=0$, and satisfying
\begin{align*}
\hb t\partial_t Q_{12}&=-\hb A_{12}+(A_{11}Q_{12}-Q_{12}A_{22})+\hb Q_{12}A_{21}Q_{12}\\
\hb t\partial_t Q_{21}&=-\hb A_{21}+(A_{22}Q_{21}-Q_{21}A_{11})+\hb Q_{21}A_{12}Q_{21}.
\end{align*}
The existence of such matrices follows from results of Majima \cite{Majima84} (see \eg \cite[Cor\ptbl A.11]{Bibi93}, with a partial system $\Sigma$ consisting of only one equation and no integrability condition). Moreover, applying \cite[Th\ptbl A.12]{Bibi93}, one can choose $Q_{12}=\hb P_{12}$, $Q_{21}=\hb P_{21}$, where $P_{12},P_{21}$ have entries in $\cA_{\cZ,\otahbo}$ and $\wh P_{12}=0$, $\wh P_{21}=0$.
\end{proof}

\section{Local properties of twistor $\cD$-modules in dimension one}
\setcounter{equation}{0}
Using the results of the previous section, we analyze some properties of twistor $\cD$-modules on a disc $X$. The main result of this section is:

\begin{theoreme}\label{th:orbnilp}
Let $\wt\cT=(\wt\cM,\wt\cM,\wt C)$ be an object of $\wtRTriples(X)$ which is strictly specializable with ramification and exponential twist at $t=0$. Assume that:
\begin{enumerate}
\item
away from $t=0$, $\wt\cT$ is smooth and $\wt C$ takes values in $\cCh{\cX^*|\bS}$ (not only in $\cCc{\cX^*|\bS}$),
\item
the sesquilinear duality $\wt\cS=(\id,\id)$ is Hermitian of weight $0$, in other words, $\wt C^*=\wt C$;
\item
the twistor properties (with polarization) are satisfied at $t=0$, \ie for any $q\geq1$, any $\varphi\in\tqm\CC[\tqm]$, any $\beta$ with $\reel\beta\in(-1,0]$ and any $\ell\in\NN$, $P\gr_\ell^\rM\psi_{t_q}^{\varphi,\beta}\rho_q^+\wt\cT(\ell/2)$, with sesquilinear duality $(\id,\id)$, is a polarized pure twistor of weight $0$.
\end{enumerate}
Then the minimal extension $\wt\cT_{\min_t}$ (\cf Definition \ref{def:mint}), equipped with the sesquilinear duality $\cS=(\id,\id)$, is a polarized pure twistor $\cD$-module of weight $0$ on some neighbourhood of~$0$ in~$X$.
\end{theoreme}

In particular, in some punctured neighbourhood $X^*$ of $t=0$, the restriction $\cT_{|X^*}$ is a smooth polarized pure twistor structure of weight $0$. This will prove that the twistor property is open. This is an analogue, in the ``easy'' direction, of the nilpotent orbit theorem in Hodge theory. A similar result, in a more specific situation (TERP structures) has been obtained by Hertling and Sevenheck in \cite[Th\ptbl9.3(2)]{H-S06}.

In \S\ref{subsec:reducramif}, we show how to reduce to the case where \emph{no ramification is needed} (\ie we replace $\wt\cT$ with $\rho_q^+\cT$ for a suitable $q$). Then, by assumption, the twistor properties are satisfied at $t=0$ and the proof consists in constructing, starting from bases of the $P\gr_\ell^\rM\psi_t^{\varphi,\beta}\wt\cM$ which are orthonormal for $P\gr_\ell^\rM\psi_t^{\varphi,\beta}\wt C$, a global frame of $\cM_{|X^*}$ which is orthonormal with respect to $\wt C_{|X^*}$. This is obtained in Corollary \ref{cor:orthonorm}. This will prove that the restriction of $(\wt\cT,\wt\cS)$ to $X^*$ is a smooth twistor structure of weight $0$, corresponding to a flat bundle with harmonic metric, so that the twistor property is satisfied for $(\wt\cT_{\min_t},(\id,\id))$ in some neighbourhood of $t=0$.

In \S\ref{subsec:basis}, we use the decomposition (DEC$^\wedge$) to construct a global basis of $\wt\cM_{|\cX^*}$ with a controlled behaviour when $t\to0$. We will assume that no ramification is needed to get the formal decomposition (this is not a restriction, as the assumptions in the theorem can be lifted to the ramified object). On the other hand, we will need neither the decomposition (DEC$^\sA$) nor the decomposition (DEC$^\sA_0$). The former will be used in \S\ref{subsec:asympt} to give an asymptotic expansion \ref{eq:devasymptC} for the sesquilinear pairing.

\subsection{Ramification}\label{subsec:reducramif}
Let us assume that Theorem \ref{th:orbnilp} is proved when no ramification is needed, that is, when we can set $q=1$ in Proposition \ref{prop:constance}, and let us show how to deduce it in general. Let $\wt\cT$ be an object of $\wtRTriples(X)$ satisfying the assumptions in Theorem \ref{th:orbnilp} and let $\rho_q:X_q\to X$ be the ramification $t_q\mto t=t_q^q$. According to the identification made in Remark~\ref{rem:ramif}, $\rho_q^+\wt\cT$ satisfies the same assumptions. If $q$ is chosen as in Proposition \ref{prop:constance}, we can apply Theorem \ref{th:orbnilp} to $\rho_q^+\wt\cT$ and deduce that $(\rho_q^+\wt\cT)_{\min_{t_q}}$ is a polarized twistor $\cD$-module of weight $0$ on some neighbourhood of $0$ in $X_q$. In particular, it is so on $X_q\moins\{0\}$ (up~to shrinking~$X_q$), and it is then clear that $\wt\cT_{\min_t}$ is so on $X\moins\{0\}$ (up~to shrinking~$X$). By assumption, the polarized twistor property is satisfied at $t=0$ for $\wt\cT_{\min_t}$, proving thus the assertion of Theorem \ref{th:orbnilp} for $\wt\cT_{\min_t}$.

From now on, we assume that no ramification is needed.

\enlargethispage{\baselineskip}%
\subsection{Construction of local bases}
\label{subsec:basis}
We denote by $B$ the subset of $\{\beta\in\CC\mid\reel\beta\in(-1,0]\}$ consisting of $\beta$'s such that some $\psi_t^{\varphi,\beta}\wt\cM\neq\nobreak0$. By assumption, for any $\varphi\in\tm\CC[\tm]$, any $\beta\in B$ and any $\ell\in\NN$, the $\cO_{\Omega_0}$-module $P\gr_\ell^{\rM}\psi_t^{\varphi,\beta}\wt\cM$ is free of finite rank. Moreover, for any $\varphi$ in $\tm\CC[\tm]$ and $\beta\in B$, we can find a global basis $\bme^o_{\varphi,\beta,\ell}$ of $P\gr_\ell^{\rM}\psi_t^{\varphi,\beta}\wt\cM$ which is orthonormal with respect to the specialized sesquilinear form $P\gr_\ell^{\rM}\psi_t^{\varphi,\beta}C(\ell/2)$.

The family $(\bme^o_{\varphi,\beta,\ell})_{\beta\in B,\, \ell\in\NN}$ generates a basis $\bme^o_{\varphi,\beta,\ell,\ell-2k}=(-\rN)^k\bme^o_{\varphi,\beta,\ell}$ ($k\in\NN$) of $\gr_\bullet^{\rM}\psi_t^{\varphi,\beta}\wt\cM$. We will denote by $\rY_{\varphi,\beta}$ the matrix of $-\rN$ in this basis, and by $\rH_{\varphi,\beta}$ the matrix defined by $\rH_{\varphi,\beta}=w\id$ on $\gr_w^{\rM}\psi_t^{\varphi,\beta}\wt\cM$. Lastly, for any $w$ and any $j\in\ZZ$, we set $\bme^o_{\varphi,\beta+j,\ell,w}=t^j\bme^o_{\varphi,\beta,\ell,w}$.

For any $\hbo\in\Omega_0$ and $\beta\in B$, we denote by $q_{\beta,\hbo}\in\ZZ$ the integer such that $\ell_\hbo(\beta+\nobreak q_{\beta,\hbo})\in[0,1)$, and we will also consider the set $B_\hbo=\{\beta+\nobreak q_{\beta,\hbo}\mid\beta\in B\}$. We use the partial order on $B_\hbo$ coming from the order on $\ell_\hbo(\beta)$'s.
We denote by $\bme^{o,\phbo}$ the basis generated by $(\bme^o_{\varphi,\beta,\ell})_{\beta\in B_\hbo,\, \ell\in\NN}$.

\subsubsection{Construction of a formal basis}\label{subsec:formalbasis}
We say that a $\cO_{\wh\cX,\ohbo}[\tm]$-basis $\wh\bme^\phbo$ of $\wt\cM^\wedge_\ohbo$ is \emph{admissible} with respect to $\bme^o$ if $\wh\bme^\phbo=(\wh\bme^\phbo_{\varphi,\beta,\ell,w})$ with $\beta\in B_{\hbo}$, $\ell\in\NN$ and moreover $w\in\ZZ\cap[-\ell,\ell]$ and:
\begin{itemize}
\item
for any $\varphi$, $\wh\bme^\phbo_{\varphi,\beta,\ell,w}\in(\cE^{\varphi/\hb}\otimes \wt\cR^\wedge_\varphi)_\ohbo$ and, if $\hbo=0$, its restriction at $\hb=0$ belongs to $(\cE^{\varphi/\hb}_{|\hb=0}\otimes \wt R_\varphi)_\oO$,
\item
$\wh\bme^\phbo_{\varphi,\beta,\ell,\ell}$ induces $\bme^{o,\phbo}_{\varphi,\beta,\ell,\ell}$ on $P\gr_\ell^{\rM}\psi_t^{\varphi,\beta}\cM_\ohbo$,
\item
for any $k=0,\dots,\ell$, $\wh\bme^\phbo_{\varphi,\beta,\ell,\ell-2k}=[t\partiall_t-\beta\star\hb-t\varphi']^k\cdot\wh\bme^\phbo_{\varphi,\beta,\ell,\ell}$.
\end{itemize}

Let us be more precise on the word ``induce'': it means that the class of $\bme^\phbo_{\varphi,\beta,\ell,\ell}$ modulo $V^{>\ell_\hbo(\beta)}\wt\cM_\ohbo$ has a component on $\psi_t^{\varphi,\beta}\wt\cM_\ohbo$ only (and a zero component on any $\psi_t^{\varphi,\gamma}\wt\cM_\ohbo$ with $\gamma\neq\beta$ and $\ell_\hbo(\gamma)=\ell_\hbo(\beta)$); moreover this class belongs to $\rM_\ell\psi_t^{\varphi,\beta}\wt\cM_\ohbo$ and is primitive modulo $\rM_{\ell-1}\psi_t^{\varphi,\beta}\wt\cM_\ohbo$.

Below, we will not indicate the index $\ell$, which only refers to the weight of the primitive element $\wh\bme^\phbo_{\varphi,\beta,\ell,w}$ comes from.

\begin{lemme}\label{lem:baseadmform}
The matrix of $t\partiall_t$ in any admissible formal basis $\wh\bme^\phbo$ takes the following form:
\begin{equation}\tag*{(\protect\ref{lem:baseadmform})($*$)}\label{eq:baseadmform}
\tbigoplus_{\varphi}\Big[\tbigoplus_{\beta\in B_\hbo}\Big([\beta\star\hb+t\varphi']\id+\rY_{\varphi,\beta}\Big)+\wh P_\varphi(t,\hb)\Big],
\end{equation}
where $\wh P_\varphi$ satisfies:
\begin{itemize}
\item
if $\beta\neq\gamma$,
\begin{itemize}
\item
if $\ell_{\hb_o}(\gamma)\leq \ell_{\hb_o}(\beta)$, then $\wh P_{\varphi,\gamma,\beta}\in t\cO_{\wh\cX,\ohbo}$;
\item
if $\ell_{\hb_o}(\gamma)> \ell_{\hb_o}(\beta)$, then $\wh P_{\varphi,\gamma,\beta}\in\cO_{\wh\cX,\ohbo}$;
\end{itemize}
\item
if $\beta=\gamma$,
\begin{itemize}
\item
if $w'\geq w-2$, then $\wh P_{\varphi,\beta;w',w}\in t\cO_{\wh\cX,\ohbo}$;
\item
if $w'\leq w-3$, then $\wh P_{\varphi,\beta;w',w}\in\cO_{\wh\cX,\ohbo}$.
\end{itemize}
\end{itemize}
\end{lemme}

The notation $\wh P_{\varphi,\gamma,\beta}$ (\resp $\wh P_{\varphi,\beta;w',w}$) is chosen such that
\[
t\partiall_t\wh\bme_{\varphi,\beta}^\phbo=\wh\bme_{\varphi,\beta}^\phbo\cdot \Big([\beta\star\hb+t\varphi']\id+\rY_{\varphi,\beta}\Big)
+ \sum_\gamma\wh\bme_{\varphi,\gamma}^\phbo\cdot
\wh P_{\varphi,\gamma,\beta}
\]
and
\[
\Big(\wh\bme_{\varphi,\beta}^\phbo\cdot\wh P_{\varphi,\beta,\beta}\Big)_w=\sum_{w'}\wh\bme_{\varphi,\beta,w'}^\phbo\cdot\wh P_{\varphi,\beta;w',w}
\]

\begin{proof}
This is a direct consequence of the definition of admissibility.
\end{proof}

By construction, we also have:
\begin{lemme}\label{lem:bmeowedge}
Let $\wh\bme^\phbo_1,\wh\bme^\phbo_2$ be two local admissible formal bases (with respect to~$\bme^o$). Then we have a relation $\wh\bme^\phbo_2=\wh\bme^\phbo_1\cdot(\id+\wh Q(t,\hb))$, where $\wh Q$ satisfies
\begin{itemize}
\item
if $\psi\neq\varphi$, $\wh Q_{\varphi,\psi}=0$,
\item
if $\gamma\neq\beta$, $\wh Q_{\varphi,\gamma,\beta}(0,\hb)\not\equiv0\implique \ell_\hbo(\gamma)> \ell_\hbo(\beta)$,
\item
$\wh Q_{\varphi,\beta,\beta}(0,\hb)$ has weight $\leq-1$ with respect to $\rH_{\varphi,\beta}$,
\item
$\wh Q(t,0)$ has entries in $\cO_{X,0}$ (with $\wh Q(t,0)_{\varphi,\psi}=0$ if $\psi\neq\varphi$).\qed
\end{itemize}
\end{lemme}

\subsubsection{Admissible local holomorphic bases}\label{subsec:localholbasis}
We keep the notation as above. According to \cite[Prop\ptbl2.1]{Malgrange04}, given any basis $\wh\bme^\phbo$ of $\wt\cM_\ohbo^\wedge$, the $\cO_{\cX,\ohbo}$-module $(\cO_{\wh\cX,\ohbo}\cdot\wh\bme^\phbo)\cap\wt\cM_\ohbo$ is free of finite rank. Given any holomorphic basis $\bmm^\phbo$ of it, we consider the invertible matrix $\wh \ccP$ with entries in $\cO_{\wh\cX,\ohbo}$ such that $\bmm^\phbo=\wh\bme^\phbo\cdot \wh \ccP$. We can assume that, if $\hbo=0$, this basis is a lifting of the restriction of $\wh\bme^{(0)}$ to $\hb=0$. If $\ccP=\ccP_0+t\ccP_1+\cdots$, then $\ccP_0$ is invertible. Replacing $\bmm^\phbo$ with $\bmm^\phbo\cdot \ccP_0^{-1}$, we can assume that $\ccP_0=\id$. Similarly, replacing then $\bmm^\phbo$ with $\bmm^\phbo\cdot (\id-t\ccP_1)$, we can assume that $\ccP_1=0$. More generally, given any fixed integer $k$, we can assume that $\ccP_j=0$ for any $j\leq k$. Lastly, we can assume that the restriction at $\hb=0$ is equal to the identity matrix.

We say that a basis $\bme^\phbo$ is \emph{admissible} with respect to $\bme^o$ if
\begin{itemize}
\item
it comes from an admissible formal basis after a base change $(\id+\hb t^k\wh Q(t,\hb))$ where $k$ is strictly bigger than the maximal order of the pole of the $\varphi_j$'s and $\wh Q$ has entries in $\cO_{\wh\cX,\ohbo}$.
\end{itemize}
We can decompose such a basis into subfamilies $\bme^\phbo_{\varphi,\beta}$, $\bme^\phbo_{\varphi,\beta,w}$, $\bme^\phbo_{\varphi,\beta,\ell,w}$, and we also assume that
\begin{itemize}
\item
for any $k=0,\dots,\ell$, $\bme^\phbo_{\varphi,\beta,\ell,\ell-2k}=[t\partiall_t-\beta\star\hb-t\varphi']^k\cdot\bme^\phbo_{\varphi,\beta,\ell,\ell}$.
\end{itemize}
(This can be achieved starting from a basis satisfying the first point only, by replacing $\bme^\phbo_{\varphi,\beta,\ell,\ell-2k}$ with $[t\partiall_t-\beta\star\hb-t\varphi']^k\cdot\bme^\phbo_{\varphi,\beta,\ell,\ell}$, without breaking the first point, as the formal basis satisfies the same property.)

In any admissible local holomorphic basis $\bme^\phbo$, the matrix of $t\partiall_t$ takes the form
\begin{equation}\label{eq:baseadm}
\tbigoplus_{\varphi}\Big[\tbigoplus_{\beta\in B_\hbo}\Big([\beta\star\hb+t\varphi']\id+\rY_{\varphi,\beta}\Big)+ P_\varphi(t,\hb)\Big]\oplus\tbigoplus_{\varphi\neq\psi}P_{\varphi,\psi},
\end{equation}
with $P_\varphi$ (instead of $\wh P_\varphi$) satisfying the properties given in Lemma \ref{lem:baseadmform}, with $\cO_{\cX,\ohbo}$ instead of $\cO_{\wh\cX,\ohbo}$ and, for $\psi\neq\varphi$, $P_{\varphi,\psi}$ has entries in $\hb t^\ell\cO_{\cX,\ohbo}$ for some $\ell\geq1$. Moreover, given any such $\ell$, one can find an admissible local holomorphic basis such that, for any $\varphi\neq\psi$, $P_{\varphi,\psi}$ has entries in $\hb t^\ell\cO_{\cX,\ohbo}$.

\begin{remarque}\label{rem:Vfil}
Any element of $\bme^\phbo_{0,\beta}$ is a local section of $V_\phbo^{\ell_\hbo(\beta)}\wt\cM_\ohbo$ and its class in $\gr_{V_\phbo}^{\ell_\hbo(\beta)}\wt\cM_\ohbo$ is annihilated by a power or $t\partiall_t-\beta\star\hb$. Moreover, the admissible local holomorphic basis $\bme^\phbo$ can be chosen such that, for any $\varphi\neq0$, $\bme^\phbo_{\varphi,\beta}$ is contained in $V_\phbo^{\ell_\hbo(\beta)+\ell}\wt\cM_\hbo$, where $\ell$ is a given arbitrary integer.

\begin{proof}
By definition, a similar result holds at the formal level for $\wh\bme^\phbo_{0,\beta}$, and if $\varphi\neq0$, $\wh\bme^\phbo_{\varphi,\beta}$ is a local section of $V_\phbo^{\ell_\hbo(\beta)+k}\wt\cM^\wedge_\hbo$ for any $k$. For $b\in[0,1)$, let us denote by $\wh\cL^b$ the $\cO_{\wh\cX,\ohbo}$-free module generated by the $\wh\bme^\phbo_{\varphi}$ ($\varphi\neq0$), the $\bme^\phbo_{0,\beta}$ with $\ell_\hbo(\beta)\geq b$, and the $t\bme^\phbo_{0,\beta}$ with $\ell_\hbo(\beta)<b$. Then $V_\phbo^b\wt\cM^\wedge_\hbo=\sum_\ell(t\partiall_t)^\ell\wh\cL^b$.

As indicated at the beginning of \S\ref{subsec:localholbasis}, $\cL^b\defin\wh\cL^b\cap\wt\cM$ is $\cO_{\cX,\ohbo}$-free of finite rank and it follows that $V_\phbo^b\wt\cM_\hbo=\sum_\ell(t\partiall_t)^\ell\cL^b$ (\cf the argument of Remark \ref{rem:formal}\eqref{rem:formal2}). This gives the first point, and the second one follows because $\bme^\phbo_{0,\beta}$ and $\wh\bme^\phbo_{0,\beta}$ have the same class. The last point is then clear by chosing a base change $(\id+\hb t^k\wh Q(t,\hb))$ with $k$ large enough.
\end{proof}
\end{remarque}

\subsubsection{Admissible local $\cA$-bases}\label{subsec:Abasis}
By a variant of Borel-Ritt, for any $\hbo\in\Omega_0$ and any $\theta_o\in S^1$, we have a surjective morphism (Taylor expansion)
\[
\cA_{\cY,\othbo}\to\cO_{\wh\cX,\ohbo}\to0,
\]
the kernel of which consists of functions which are infinitely flat along $S^1\times\nb(\hbo)$ near $\theta_o$. Let us set
\[
\wt\cM^\sA_\othbo\defin \cA_{\cY,\othbo}\otimes_{\cO_{\cX,\ohbo}}\wt\cM_\ohbo.
\]
For any $\theta_o\in S^1$, we therefore have a surjective morphism
\[
\wt\cM^\sA_\othbo\to\wt\cM^\wedge_\ohbo\to0.
\]
We say that a $\cA_{\cY,\othbo}$-basis $\bmea^\phbo$ of $\wt\cM^\sA_\othbo$ is \emph{admissible} (with respect to~$\bme^o$) if its Taylor expansion $\wh\bme^\phbo$ at $\theta_o$ is a local admissible formal basis and, if $\hbo=0$, if the restriction of $\bmea^{(0)}$ to $\hb=0$ is the pull-back of a $\cO_{X,0}$-basis of $\wt M_0$.

\begin{lemme}\label{lem:bmeoA}
Two admissible local $\cA$-bases $\bmea^\phbo_1,\bmea^\phbo_2$ are related by a base change $\bmea^\phbo_2\!=\!\bmea^\phbo_1{\cdot}(\id+\QA(t,\hb))$, also written $\bmea^\phbo_{2,\psi}\!=\!\bmea^\phbo_{1,\psi}+\sum_\varphi\!\bmea^\phbo_{1,\varphi}\QA_{\varphi,\psi}(t,\hb))$, where $\QA$ has entries in $\cA_{\cY,\othbo}$ and satisfies
\begin{itemize}
\item
if $\psi\neq\varphi$, then $\QA_{\varphi,\psi}$ is infinitely flat when $t\to0$, \item
if $\gamma\neq\beta$, $\QA_{\varphi,\gamma,\beta}(0,\hb)\not\equiv0\implique \ell_\hbo(\gamma)> \ell_\hbo(\beta)$,
\item
$\QA_{\varphi,\beta,\beta}(0,\hb)$ has weight $\leq-1$ with respect to $\rH_{\varphi,\beta}$,
\item
$\QA(t,0)$ has entries in $\cO_{X,0}$ (in particular, $\QA(t,0)_{\varphi,\psi}=0$ if $\psi\neq\varphi$).
\end{itemize}
\end{lemme}

\begin{proof}
The Taylor expansion at $\theta_o$ of $\QA$ must satisfy the properties of Lemma \ref{lem:bmeowedge}.
\end{proof}

\begin{lemme}\label{lem:Ahol}
Let $\wh\bme^\phbo$ be an admissible formal basis, $\bmea^\phbo$ an admissible local $\cA$-basis with respect to $\wh\bme^\phbo$ at $\othbo$ and $\bme^\phbo$ be an admissible holomorphic basis of order $k\geq1$ with respect to $\wh\bme^\phbo$ at $\ohbo$. Then there exists a matrix $\QA$ with entries in $\cA_{\cY,\othbo}$ such that $\bmea^\phbo=\bme^\phbo\cdot(\id+\hb t^k\QA)$.
\end{lemme}

\begin{proof}
We have a relation $\wh\bme^\phbo=\bme^\phbo\cdot(\id+\hb t^k\wh Q)$ by assumption. Lifting this relation by the Taylor map gives the relation for $\bmea^\phbo$.
\end{proof}

\enlargethispage{\baselineskip}%
\subsubsection{Untwisted admissible local $C^\infty$-bases}\label{subsec:cinfbasis}
For each $\varphi$ and $\beta\in B$, let us set
\begin{equation}\label{eq:A}
A(t,\hb)=\oplus_{\varphi,\beta\in B}A_{\varphi,\beta},\quad A_{\varphi,\beta}(t,\hb)=\mt^{\beta'+i\hb\beta''}\Lt^{\rH_{\varphi,\beta}/2},
\end{equation}
with $\Lt\defin\vert\log\vert t\vert^2\vert$. Given an admissible local $\cA$-basis $\bmea^\phbo$, we set
\begin{equation}\label{eq:untwist}
\varepsilong^\phbo=\bmea^\phbo\cdot A^{-1}(t,\hb).
\end{equation}
We say that $\varepsilong^\phbo$ is an untwisted admissible local $\cCh{\cY^*,\othbo}$-basis. Let us notice that, as $A$ is diagonal with respect to the $(\varphi,\beta)$ decomposition, we also have $\varepsilong^\phbo=(\varepsilong_{\varphi,\beta}^\phbo)_{\varphi,\beta\in B}$ with $\varepsilong_{\varphi,\beta}^\phbo=\bmea^\phbo_{\varphi,\beta}\cdot A_{\varphi,\beta}^{-1}(t,\hb)$.

\begin{lemme}\label{lem:untwist}
For $\hbo\in\Omega_0$ and $\theta_o\in S^1$, if $\varepsilong^\phbo_1$ and $\varepsilong ^\phbo_2$ are two untwisted admissible local $\cCh{\cY^*,\othbo}$-bases of $\wt\cM^\sA_{|\cY^*,\othbo}$, and if we set $\varepsilong^\phbo_2=\varepsilong^\phbo_1\cdot(\id+R(t,\hb))$, then, for any $\delta\in(0,1/2)$, we have, uniformly for $\hb$ in some neighbourhood $\nb(\hbo)$ of~$\hbo$,
\[
\lim_{t\to0}\Lt^\delta \vert R(t,\hb)\vert=0.
\]
Moreover, $R(t,0)_{\varphi,\psi}=0$ if $\psi\neq\varphi$.
\end{lemme}

Let us emphasize that we index the bases $\varepsilong^\phbo_\varphi$ by $\beta\in B$, as they will be globalized, so the indexing set should not depend on~$\hbo$. But, up to a phase factor, we have, for $\beta\in B$, $\varepsilong_{\varphi,\beta}^\phbo=\bmea^\phbo_{\varphi,\beta+q_{\beta,\hbo}}\cdot A_{\varphi,\beta+q_{\beta,\hbo}}^{-1}(t,\hb)$. Therefore, in the proof below, we assume that the index set is $B_\hbo$.

\begin{proof}
Let $\QA$ be as in Lemma \ref{lem:bmeoA}. If $\psi\neq\varphi$, we find that
$A_\varphi\QA_{\varphi,\psi}A_\psi^{-1}$ is infinitely flat along $\mt=0$. Therefore, $R_{\varphi,\psi}(t,\hb)$ is infinitely flat when $t\to0$, uniformly in some neighbourhood of $\hbo$, and $R_{\varphi,\psi}(t,0)=0$ if $\psi\neq\varphi$.

Let us now consider the case where $\varphi=\psi$. If $\gamma\neq\beta$, let us remark that, as $\ell_\hbo(\gamma)$ and $\ell_\hbo(\beta)$ belong to $[0,1)$ by definition, we have \hbox{$1+\ell_\hbo(\gamma)-\ell_\hbo(\beta)>0$} and there exists $\epsilon>0$ and $\nb(\hbo)$ small enough so that, for any $\hb\in\nb(\hbo)$, we have $1+\ell_\hb(\gamma)-\ell_\hb(\beta)>\epsilon$. It follows that there exists $c>0$ such that, for any $\hb\in\nb(\hbo)$,
\[
\vert A_{\varphi,\gamma}(\QA_{\varphi,\gamma,\beta}-\QA_{\varphi,\gamma,\beta}(0,\hb))A_{\varphi,\beta}^{-1}\vert\leq c\mt^{\epsilon}.
\]
By the second point in Lemma \ref{lem:bmeoA}, if $\QA_{\varphi,\gamma,\beta}(0,\hb)\neq\nobreak0$, we also have such an estimate for $\vert A_{\varphi,\gamma}\cdot\QA_{\varphi,\gamma,\beta}(0,\hb)\cdot A_{\varphi,\beta}^{-1}\vert$.

When $\gamma=\beta$, we only have to estimate $\vert A_{\varphi,\beta}{\cdot}\QA_{\varphi,\beta,\beta}(0,\hb){\cdot} A_{\varphi,\beta}^{-1}\vert$ or, what amounts to the same, $\vert \Lt^{\rH_{\varphi,\beta}/2}\cdot \QA_{\varphi,\beta,\beta}(0,\hb) \cdot\Lt^{-\rH_{\varphi,\beta}/2}\vert$. By the third point in Lemma~\ref{lem:bmeoA}, this is smaller than $c\cdot\Lt^{-1/2}$.
\end{proof}

Recall that $\Delta_0=\{\hb\mid \vert\hb\vert\leq1\}$. We will now globalize the construction above.

\begin{lemme}[Globalization of untwisted local bases]\label{lem:glob}
For any $\delta\in(0,1/2)$, there exists an open neighbourhood $\nb(\Delta_0)$ of $\Delta_0$ and a basis $\varepsilong$ of $\cCch{X^*\times\nb(\Delta_0)}\otimes_{\cO_{X^*\times\nb(\Delta_0)}}\wt\cM_{|X^*\times\nb(\Delta_0)}$ satisfying the following property:
\begin{itemize}
\item
if $\varepsilong^\phbo$ is any admissible untwisted local basis, then the base change $\varepsilong=\varepsilong^\phbo\cdot(\id+R^\phbo(t,\hb))$ satisfies
\[
\lim_{t\to0}\Lt^\delta\vert R^\phbo(t,\hb)\vert=0
\]
uniformly for $\hb\in\nb(\hbo)$.
\end{itemize}
\end{lemme}

Of course, a basis $\varepsilong$ constructed starting from some $\delta\in(0,1/2)$ is convenient for any $\delta'\in(0,\delta]$.

\begin{proof}
We first use a partition of the unity with respect to $\theta\in S^1$ to globalize with respect to $\theta$ (and constant with respect to $\hb$). By Lemma \ref{lem:untwist}, the base change from the $\theta$-global frame to the local frame satisfies the desired property. The globalization with respect to $\hb$ is then similar to that of \cite[Lemma~5.4.6]{Bibi01c}.
\end{proof}

\subsection{Asymptotic expansion of sesquilinear pairings}\label{subsec:asympt}

In this paragraph, we wish to generalize Lemma~5.3.12 of \cite{Bibi01c} to objects of $\wtRTriples(X)$ which are strictly specializable with ramification and exponential twist. Although the result will not be strong enough for our purpose, it gives a first taste of the kind of expansions one can obtain.

In the following, we assume that $\hbo\in\bS$. Let $\wt\cT=(\wt\cM',\wt\cM'',\wt C)$ be an object of $\wtRTriples(X)$ which is strictly specializable with ramification and exponential twist. We will assume that no ramification is needed for $\wt\cM'$ and $\wt\cM''$, so that we can apply Proposition \ref{prop:decompA}.

For any $\varphi\in\tm\CC[\tm]$, denote by $\ephi$ the section~$1$ of $\cE^{-\varphi/\hb}$. For any $m\in\wt\cM_\ohbo$, we denote by $\Phi(m)\subset\tm\CC[\tm]$ the set of $\varphi$ such that $\psi_t^\gamma(\cE^{-\varphi/\hb}\otimes\nobreak\wt\cN)_\ohbo\neq\nobreak0$ for some $\gamma$, where $\wt\cN_\ohbo$ is the submodule generated by $m$ in $\wt\cM_\ohbo$ (by assumption and Lemma \ref{lem:submodules}, $\Phi(m)$ does not depend on~$\hbo$). We then denote by $L_{\varphi,\gamma}(m)$ the nilpotency order of $\rN=-(t\partiall_t-\gamma\star\hb)$ on $\psi_t^\gamma(\cE^{-\varphi/\hb}\otimes\wt\cN)_\ohbo$. For any $\varphi\in\Phi(m)$, there exists a minimal finite set $B_\varphi(m)$ such that, for any $j\in\NN$ there exists an integer $k(j)$ and an operator $P_j\in V_0(\cR_{\cX,\ohbo})$ such that
\begin{equation}\label{eq:Bernsteinj}
\bigg[\prod_{k=0}^{k(j)}\prod_{\beta\in B_\varphi(m)}\big[-(t\partiall_t-(\beta+k)\star\hb)\big]^{L_{\varphi,\beta}(m)}-t^jP_j\bigg]\cdot (\ephi\otimes m)=0.
\end{equation}
We also define $A_\varphi(m)=\{\alpha\mid\exists\beta\in B_\varphi(m),\,\alpha=-\beta-1\}$.

\begin{exemple}\label{exem:localholbasis}
Let $\bme^\phbo$ be an admissible local holomorphic basis as in \S\ref{subsec:localholbasis}. From Remark \ref{rem:Vfil} we deduce that, for any $\varphi$ and any $\beta\in B_\hbo$, $B_\varphi(\bme^\phbo_{\varphi,\beta})$ consists of $\beta$ and of complex numbers~$\gamma$ such that $\ell_\hbo(\gamma)>\ell_\hbo(\beta)$. Moreover, if $\psi\neq\varphi$, $\gamma\in B_\psi(\bme^\phbo_{\varphi,\beta})$ implies that $\ell_\hbo(\gamma)>\ell_\hbo(\beta)$, and, given any positive number $k$, the basis can be chosen such that the difference $\ell_\hbo(\gamma)-\ell_\hbo(\beta)$ can be made larger than~$k$.
\end{exemple}

For $m'\in\wt\cM'_\ohbo$ and $m''\in\wt\cM''_\ohbo$, we define \begin{align*}
\Phi(m',m'')&=\Phi(m')\cap\Phi(m''),
\\
B_\varphi(m',m'')&=\Big[(B_\varphi(m')-\NN)\cap B_\varphi(m'')\Big]\cup\Big[B_\varphi(m')\cap (B_\varphi(m'')-\NN)\Big]
\\
L_{\varphi,\beta}(m',m'')&=\min \{L_{\varphi,\beta}(m'),L_{\varphi,\beta}(m'')\}.
\end{align*}

Lastly, if $f\in \cCc{\cX,\ohbo}$, one can expand $f$ with respect to $t,\ovt$ and one can associate to this expansion a minimal set $E(f)\subset\NN^2$ such that $f=\sum_{(\nu',\nu'')\in E(f)}t^{\nu'}\ovt^{\nu''}f_{(\nu',\nu'')}$ with $f_{(\nu',\nu'')}\in \cCc{\cX,\ohbo}$. By convention, $E(f)=\emptyset$ if $f\equiv0$.

\begin{proposition}\label{prop:devasymptC}
With these assumptions and notation, let $m'\in\wt\cM'_\ohbo$ and $m''\in\wt\cM''_\mohbo$. For any $\varphi\in\Phi(m',m'')=:\Phi$, any $\beta\in B_\varphi(m',m'')=:B_\varphi$ and any $\ell=0,\dots,L_{\varphi,\beta}(m',m'')-1=:L_{\varphi,\beta}-1$, there exists $f_{\varphi,\beta,\ell}\in\cCc{\cX,\ohbo}$ and $N\in\NN$ such that, in $\Dbh{\cX}[\tm]_\ohbo$ and hence in $\cCc{}(U^*\times(\nb(\hbo)\cap \bS))$ for $U$ and $\nb(\hbo)\cap \bS$ small enough,
\begin{multline}\tag*{(\protect\ref{prop:devasymptC})($*$)}\label{eq:devasymptC}
(\hb+1/\hb)^N \wt C(m',\ovv{m''})_\ohbo\\[-3pt]
=\sum_{\varphi\in\Phi}\sum_{\beta\in B_\varphi}\sum_{\ell=0}^{L_{\varphi,\beta}-1}f_{\varphi,\beta,\ell}(t,\hb)\,e^{-\hb\ov\varphi+\varphi/\hb}\,\mt^{2\beta\star\hb/\hb}\frac{\Lt^\ell}{\ell!}.
\end{multline}
Moreover, if $f_{\varphi,\beta,\ell}\neq0$ and if the point $(k',k'')\in\NN^2$ belongs to $E(f_{\varphi,\beta,\ell})$, then $\beta+k'\in B_\varphi(m')+\NN$ and $\beta+k''\in B_\varphi(m'')+\NN$.
\end{proposition}

\begin{remarques}\label{rem:devasymptC}\mbox{}
\begin{enumerate}
\item\label{rem:devasymptC1}
As a consequence of the last part of the proposition, we can also write \ref{eq:devasymptC} as a finite sum with terms
\[
g_{\varphi,\beta,\ell}^{(k',k'')}(t,\hb)\,e^{-\hb\ov\varphi+\varphi/\hb}\,t^{k'}\ovt^{k''}\,\mt^{\beta'+i\beta''\hb}\mt^{\beta'+i\beta''/\hb}\frac{\Lt^\ell}{\ell!},
\]
where $k',k''\in\NN$, $\beta+k'$ runs in $B_\varphi(m')+\NN$ and $\beta+k''$ in $B_\varphi(m'')+\NN$, and $g_{\varphi,\beta,\ell}^{(k',k'')}(t,\hb)\in\cCc{\cX,\ohbo}$ is such that $g_{\varphi,\beta,\ell}^{(k',k'')}(0,\hb)\neq0$. (Recall that we denote by~$\beta'$ (\resp $\beta''$) the real part (\resp the imaginary part) of $\beta$.) Let us also notice that
\begin{align*}
\reel(\beta'+k'+i\beta''\hb)&=\ell_\hb(\beta+k')\quad\text{(by definition)},\\
\reel(\beta'+k''+i\beta''/\hb)&=\ell_{-\hb}(\beta+k'')\quad\text{(as $\hb\in\bS$).}
\end{align*}

\item\label{rem:devasymptC2}
We know that, when restricted to $X^*=\{t\neq0\}$, $\wt C$ takes values in $\cCc{\cX^*|\bS}$. If we moreover assume that it takes values in $\cCh{\cX^*|\bS}$, it will be clear from the end of the proof below that the $f_{\varphi,\beta,\ell}$ also belong to $\cCh{\cX|\bS}$. Then the functions $f_{\varphi,\beta,\ell}$ (the number of which is finite) can be regarded as $C^\infty$ functions from $X$ into the Banach space $\rH(\nb(\hbo))$ of continuous functions on the closure of a small disc $\nb(\hbo)$ which are holomorphic in its interior.

Similarly, if $\wt C$ takes values in $\cCh{\cX^*|\bS}$, it is holomorphic with respect to $\hb$ in a set like $V\ohbo\cap\cX^*$, where $V\ohbo$ is a neighbourhood of $\ohbo$ in $\cX$. However, the intersection of $V\ohbo$ with $\{t\}\times\Omega_0$ could a~priori have a radius tending to $0$ when $t\to0$. But from \ref{eq:devasymptC} and the previous remark, we note that this cannot happen, as $(\hb+1/\hb)^N\wt C$ is then holomorphic on $X^*\times\nb(\hbo)$. Therefore, $\wt C$ is holomorphic with respect to~$\hb$ on $X^*\times\nb(\hbo)$.
\end{enumerate}
\end{remarques}

\enlargethispage{\baselineskip}%
\begin{proof}[\proofname\ of Proposition \ref{prop:devasymptC}]\mbox{}
\subsubsection{Computing the set of indices}\label{subsec:proofcomputing}
Let us first assume that we have proved the existence of a formula like \ref{eq:devasymptC}, without being precise on the sets of indices. We will now show how to recover the information on these sets. We will use the Mellin transform to treat each coefficient of the expansion \ref{eq:devasymptC}. Let $\chi(t)$ be a $C^\infty$ function on~$X$ with compact support contained in an open set over which $m',m''$ are defined, identically equal to $1$ near $0$. We denote in the same way the form $\chi\tfrac{i}{2\pi}\,dt\wedge d\ovt$. Let us denote by $p$ the order of $\wt C(m',\ovv{m''})$ on the support of $\chi$ (recall that there exists $r\in\NN$ such that $t^r\wt C(m',\ovv{m''})$ is a distribution which is continuous with respect to $\hb\in\bS$, hence has a well-defined order $q$, and we take $p=q+r$). We will first consider the coefficients for which $\varphi=0$. Let us set $\wt v=(\hb+1/\hb)^N\wt C(m',\ovv{m''})$ and $v=t^r\wt v$, which is a local section of $\Dbh{X}$.

For all $k',k''\in\NN$, the function $s\mto\langle v,\mt^{2s}t^{-k'}\ovt^{-k''}\chi\rangle$ is defined and holomorphic on the half plane $2\reel s>q+k'+k''$. Let $Q_j$ be the operator appearing in \eqref{eq:Bernsteinj} when $\varphi=0$. Then $Q_j\cdot v$
is supported at $\{t=0\}$. It follows that, on some half plane $\reel s\gg0$, if we relate $\alpha$ and~$\beta$ by $\alpha=-\beta-1$, the function
\[
\bigg[\prod_{k=0}^{k(j)}\prod_{\alpha\in A_0(m)}\big[\hb(s-k'+k)-\alpha\star\hb\big]^{L_{0,\alpha}(m')}\bigg]\langle v,\mt^{2s}t^{-k'}\ovt^{-k''}\chi\rangle
\]
coincides with a function which is holomorphic on $2\reel s>q+k'+k''-\nobreak j$. By applying the anti-holomorphic argument we find that, for any $k',k''\in\NN$, the function $s\mto\langle v,\mt^{2s}t^{-k'}\ovt^{-k''}\chi\rangle$ extends as a meromorphic function on~$\CC$ with poles contained in the sets $s=\alpha\star\hb/\hb$ with $\alpha\in (A_{0}(m')+k'-\NN)\cap (A_{0}(m'')+k''-\NN)$ and the order along $s=\alpha\star\hb/\hb$ is bounded by $L_{0,\alpha}(m',m'')$. Moreover, this function only depends on~$\wt v$.

Let us now compute the Mellin transform of the expansion \ref{eq:devasymptC} for $\wt v$. Let us first remark that, if $\varphi\neq0$, the Mellin transform of $e^{-\hb\ov\varphi+\varphi/\hb}\mt^{2\beta\star\hb/\hb}\Lt^\ell$ is an entire function: one argues as above, noticing that the term between brackets can be chosen equal to~$1$.

Let us then consider the terms for which $\varphi=0$. It is not restrictive to assume that two distinct elements of the set of indices $B_0$ do not differ by an integer and that any element $\beta$ in $B_0$ is maximal, that is, the set $\bigcup_\ell E(f_{0,\beta,\ell})$ is contained in $\NN^2$ and in no $(m,m)+\NN^2$ with $m\in\NN^*$. Let $\beta\in B_0$. We will use that, for any $\nu',\nu''\in\ZZ$ not both negative and any function $g\in \cCc{\cX,\ohbo}$ such that $g(0,\hb)\not\equiv0$, the meromorphic function $s\mto\langle g(t,\hb)\mt^{2\beta\star\hb/\hb}\Lt^\ell,\mt^{2s}t^{\nu'}\ovt^{\nu''}\chi\rangle$ has poles at most along the sets $s=\alpha\star\hb/\hb-\NN$ (with $\alpha=-\beta-1$), and has a pole along $s=\alpha\star\hb/\hb$ if and only if $\nu'=0$ and $\nu''=0$, this pole having precisely order $\ell+1$.

For $\beta\in B_0$, let $E_\beta\subset\NN^2$ be minimal such that $E_\beta+\NN^2=\bigcup_\ell (E(f_{0,\beta,\ell})+\NN^2)$. From the previous argument we can conclude that, for any $(k',k'')\in E_\beta$, the function $s\mto\langle \wt v,\mt^{2s}t^{-k'}\ovt^{-k''}\chi\rangle$ has a non trivial pole along $s=\alpha\star\hb/\hb$; on the other hand, from the first part of the proof it follows that $\alpha-k'\in A_{0}(m')-\NN$ and $\alpha-k''\in A_{0}(m'')-\NN$, that is $\beta+k'\in B_{0}(m')+\NN$ and $\beta+k''\in B_{0}(m'')+\NN$. As we assume that $\beta$ is maximal, there exists $(k',k'')\in E_\beta$ with $k'=0$ or $k''=0$. It follows that $\beta\in B_{0}(m',m'')+\NN$ and that the condition given in the proposition is satisfied by the elements of $E_\beta$. It is then trivially satisfied by the elements of all the $E(f_{0,\beta,\ell})$.

In order to obtain the result for the $f_{\varphi,\beta,\ell}$, one applies the previous result to the moderate distribution $e^{\hb\ov\varphi-\varphi/\hb}\wt v$.\qed

\subsubsection{\proofname\ of the existence of a formula}\label{proof:devasymptCb}
We now prove the existence of a formula like \ref{eq:devasymptC}, without being precise on the set of indices.

\def\theparagraph{\textup{\thesubsubsection(\arabic{paragraph})}}
\refstepcounter{paragraph}
\subsubsection*{\theparagraph. Sectorial formula}\label{proof:devasymptC1}
We will work on the real blow up $Y=[0,r_0)\times S^1$ of the origin in $X$, with blowing-up map $e:Y\to X$, and on the corresponding space $\cY=Y\times\Omega_0$, as in \S\ref{subsec:sectneq0}. We consider the sheaves $\cA_\cY$, etc.\ as in \S\ref{subsec:sectneq0}.

We can extend in a unique way the pairing $\wt C$ to a pairing $\wt C^\sA:\wt\cM_{|\bS}^{\prime\sA}\otimes_{\cOS}\ovv{\wt\cM_{|\bS}^{\prime\prime\sA}}\to\Dbh{Y}[\tm]$ in a way compatible with the $\cR_\cX$ and the $\cR_{\ovvv\cX}$-action (one uses the local freeness of $\wt\cM',\wt\cM''$ given by Lemma \ref{lem:holstrict}).

We will fix $\theta_o\in S^1$ and $\hbo\in\bS$ and we work with germs at $\othbo$. According to (DEC$^\sA$) in Proposition \ref{prop:decompA} and Lemma \ref{lem:localreginv}, for $N\gg0$, $(\hb+1/\hb)^Nm'$ (\resp $(\hb+\nobreak1/\hb)^Nm''$) is a linear combination with coefficients in $\cA_{\cY,\othbo}$ of sections $\mu'$ (\resp $\mu''$) which satisfy equations of the form
\[
\begin{cases}
(t\partiall_t-\beta\star\hb-t\varphi'(t))^\ell \mu'=0,\\
\ovv{t\partiall_t}\mu'=0
\end{cases}
\qquad
\begin{cases}
(t\partiall_t-\gamma\star\hb-t\psi'(t))^n\mu''=0,\\
\ovv{t\partiall_t}\mu''=0
\end{cases}
\]
On $(Y^*\times\bS)\cap\nb\othbo$, $\wt C^\sA(\mu',\ovv{\mu''})$ takes values in $\cCc{Y^*\times\bS}$ (\resp $\cCh{\cY^*|\bS}$ if we assume that $\wt C$ takes values in $\cCh{\cX^*|\bS}$). It follows that, on such a set, the $C^\infty$ function of $t$ (with a given branch of $\log t$ chosen near $\theta_o$)
\[
t^{-\beta\star\hb/\hb}e^{-\varphi/\hb}\ovv{t^{-\gamma\star\hb/\hb}}\ovv{e^{-\psi/\hb}}\wt C^\sA(\mu',\ovv{\mu''})
\]
is annihilated by $(t\partiall_t)^\ell$ and $\ovv{t\partiall_t}{}^n$, hence is a linear combination, with coefficients depending on $\hb$ only (in a $C^0$ way, or an analytic way if $\wt C$ takes values in $\cCh{\cX^*|\bS}$) of terms $(\log t)^a(\log\ovt)^b$. Hence $\wt C^\sA(\mu',\ovv{\mu''})$ is a linear combination with such coefficients of terms $e^{-\hb\ov\psi+\varphi/\hb}t^{\beta\star\hb/\hb}\ovv{t^{\gamma\star\hb/\hb}}(\log t)^a(\log\ovt)^b$. It will be convenient to write $e^{-\hb\ov\psi+\varphi/\hb}=e^{-\hb\ov\psi+\psi/\hb}e^{(\varphi-\psi)/\hb}$ and to note that, when $\hb\in\bS$, $|e^{-\hb\ov\psi+\psi/\hb}|=1$.

If $\varphi\neq\psi$, then, as $\wt C^\sA(\mu',\ovv{\mu''})$ has moderate growth along $\mt=0$ uniformly in $(Y^*\times\bS)\cap\nb\othbo$, we can assume that $\nb\othbo$ is small enough so that
\begin{itemize}
\item
either $\reel[(\varphi-\psi)/\hb]<0$ all over this neighbourhood and then $\wt C^\sA(\mu',\ovv{\mu''})$ is infinitely flat along $\mt=0$ (locally uniformly with respect to $\hb\in\nb(\hbo)$),
\item
or, whatever the size of $\nb\othbo$ is, $\reel[(\varphi-\psi)/\hb]$ takes positive values on some nonempty open subset of $(Y^*\times\bS)\cap\nb\othbo$ and $\wt C^\sA(\mu',\ovv{\mu''})\equiv0$ on this neighbourhood (the only possibility in order to extend as a temperate distribution all over this neighbourhood).
\end{itemize}

If $\varphi=\psi$, then, as $\gamma\star\hb/\hb$ is ``real'', we have $\ovv{t^{\gamma\star\hb/\hb}}=\ovt^{\gamma\star\hb/\hb}$, and one can rewrite $e^{-\hb\ov\varphi+\varphi/\hb}t^{\beta\star\hb/\hb}\ovv{t^{\gamma\star\hb/\hb}}(\log t)^a(\log\ovt)^b$ as an expansion like \ref{eq:devasymptC}, with the $\varphi$-terms only, and $\wt f_{\varphi,\beta,\ell}\in\cCc{Y\times\bS,\othbo}$ (\resp $\wt f_{\varphi,\beta,\ell}\in\cCh{\cY|\bS,\othbo}$).

Taking a partition of unity with respect to $\theta$, we obtain \ref{eq:devasymptC} on~$Y^*\times\bS$ for $(\hb+\nobreak1/\hb)^N\wt C(m',\ovv{m''})$, that is, with coefficients $\wt f_{\varphi,\beta,\ell}$ in $e_*\cCc{Y\times\bS,(0;\hbo)}$ (\resp in $e_*\cCh{\cY|\bS,(0;\hbo)}$).

\refstepcounter{paragraph}
\subsubsection*{\theparagraph. Globalizing the sectorial formula}\label{proof:devasymptC2}
It remains to showing that the expansion can be rewritten with coefficients $f_{\varphi,\beta,\ell}$ in $\cCc{X\times\bS,(0;\hbo)}$ (\resp in $\cCh{\cX|\bS,(0;\hbo)}$). We will once more use a Mellin transform argument, as in \S\ref{subsec:proofcomputing}.

\begin{lemme}\label{lem:devasymptY}
Let $\wt B_0\subset\CC$ be a finite subset. A function
\[
\wt f(t,\hb)=\sum_{\beta\in\wt B_0}\sum_{\ell=0}^L \wt f_{\beta,\ell}(t,\hb)\mt^{2\beta\star\hb/\hb}\Lt^\ell
\]
with $\wt f_{\beta,\ell}\in e_*\cCc{Y\times\bS,(0;\hbo)}$ (\resp in $e_*\cCh{\cY|\bS,(0;\hbo)}$), can be written as
\[
\sum_{\beta\in B_0}\sum_{\ell=0}^L f_{\beta,\ell}(t,\hb)\mt^{2\beta\star\hb/\hb}\Lt^\ell
\]
for some finite set $B_0$, with $f_{\beta,\ell}\in \cCc{X\times\bS,(0;\hbo)}$ (\resp in $\cCh{\cX|\bS,(0;\hbo)}$) if and only if there exists a finite set $A_0\subset\CC$ such that, for any $k',k''\in\hNN$ with $k'+k''\in\NN$, the poles of $s\mto\langle\wt f,\mt^{2s}t^{-k'}\ovt^{-k''}\chi\rangle$ are contained in the sets $s=\alpha\star\hb/\hb$, with $\alpha\in(A_0+k'-\NN)\cap(A_0+k''-\NN)$.
\end{lemme}

We apply the lemma to the $\varphi=0$ part of the expansion obtained above: arguing as in \S\ref{subsec:proofcomputing} and using that the Mellin transform of the sum of terms with $\varphi\neq0$ is an entire function of $s$, we conclude that the condition of Lemma \ref{lem:devasymptY} is fulfilled.
\end{proof}

\begin{proof}[\proofname\ of Lemma \ref{lem:devasymptY}]
A function $\wt g\in e_*\cCc{Y\times\bS,(0;\hbo)}$ has a Taylor expansion $\sum_{m\geq0}\wt g_m(e^{i\theta})\mt^m$, where each $\wt g_m$ is continuous on $S^1\times\bS$ and $C^\infty$ with respect to~$S^1$ and can be developed in Fourier series $\sum_n\wt g_{m,n}e^{in\theta}$. Then $\wt g$ can be written as $h_0\mt^{-k_0}+h_1\mt^{-k_0+1}$ for some $k_0\in\NN$ and $h_i\in\cCc{X\times\bS,(0;\hbo)}$ if and only if
\begin{equation}\tag*{(\protect\ref{lem:devasymptY})($*$)}\label{eq:devasymptY}
\wt g_{m,n}\not\equiv0\implique m\pm n\geq-k_0.
\end{equation}
Indeed, one direction is clear. Now, if \ref{eq:devasymptY} is fulfilled, we have
\[
\wt g=\mt^{-k_0}\hspace*{-2mm}\sum_{\substack{p,q\in\NN\\p+q\geq k_0}}\hspace*{-2mm}\wt g_{p+q-k_0,p-q-k_0}t^p\ovt^q
+\mt^{-k_0+1}\hspace*{-4mm}\sum_{\substack{p,q\in\NN\\p+q\geq k_0-1}}\hspace*{-3mm}\wt g_{p+q-k_0+1,p-q-k_0+1}t^p\ovt^q.
\]
Borel's lemma gives us two functions $h_0,h_1\in\cCc{X\times\bS,(0;\hbo)}$ such that $\wt g-(h_0\mt^{-k_0}+\nobreak h_1\mt^{-k_0+1})$ is in $e_*\cCc{Y\times\bS,(0;\hbo)}$ and infinitely flat along $\mt=0$, hence belongs to $\cCc{X\times\bS,(0;\hbo)}$. We include it in one of the two terms, to get the assertion.

Condition \ref{eq:devasymptY} can be expressed in terms of Mellin transform: indeed, one can check that, for any $j',j''\in\hNN$ such that $j'+j''\in\NN$, the Mellin transform $s\mto\langle\wt g,\mt^{2s}t^{-j'}\ovt{}^{-j''}\chi\rangle$, which is holomorphic for $\reel s\gg0$, extends as a meromorphic function on $\CC$ with simple poles contained in $\hZZ$. Condition \ref{eq:devasymptY} is equivalent to the existence of $k_0\in\NN$ such that for any $j',j''\in\hNN$ with $j'+j''\in\NN$, the poles are contained in the intersection of the sets $\frac12k_0-1+j'-\NN$ and $\frac12k_0-1+j''-\NN$. This gives the lemma when $\wt B_0$ has only one element and $L=0$.

Arguing by decreasing induction on $L$ (hence on the maximal order of the poles), one then shows the lemma when $\wt B_0$ is reduced to one element. The general case is then clear.
\end{proof}

Let us denote by $E$ the block-diagonal matrix, with blocks indexed by~$\varphi$, such that the $\varphi$-diagonal block is $e^{-\hb\ov\varphi+\varphi/\hb}\id$. Recall that we have $\vert e^{-\hb\ov\varphi+\varphi/\hb}\vert=1$ when $\hb\in\bS$. Notice also that, as $-\hb\ov\varphi+\varphi/\hb$ is ``real'', we have $E^*=E$. Recall that the matrix $A$ is defined in \eqref{eq:A}.

\begin{corollaire}\label{cor:admissibleholC}
Let $\hbo\in\bS$, let $\bme^\pmhbo$ be a pair of admissible local holomorphic bases (\cf\S\ref{subsec:localholbasis}) and let $\bC^\phbo$ be the matrix of $\wt C$ in these bases. Then, for $N$ large enough,
\[
(\hb+1/\hb)^N\cdot{}^t\!A^{-1}\bC^\phbo\ovv{A}{}^{-1}=(\hb+1/\hb)^NE+R^\phbo(t,\hb),
\]
where the entries of $R^\phbo$ are $C^\infty$ and holomorphic with respect to $\hb$ on $X^*\times\nb(\hbo)$, and $\lim_{t\to0}\Lt^\delta|R^\phbo(t,\hb)|=0$ for some $\delta>0$, uniformly on $\nb(\hbo)\cap\bS$.
\end{corollaire}

\begin{proof}
We consider the setting of Example \ref{exem:localholbasis}. Then, for any $\varphi_1,\varphi_2$, any $\beta_1\in B_\hbo$, $\beta_2\in B_\mhbo$, any $w_1,w_2\in\ZZ$, and any $e_1^\phbo\in\bme_{\varphi_1,\beta_1,w_1}^\phbo$, $e_2^\mphbo\in\bme_{\varphi_2,\beta_2,w_2}^\mphbo$, $(\hb+\nobreak1/\hb)^N\wt C(e_1^\phbo,\ovv{e_2^\mphbo})$ is a sum of terms as in Remark \ref{rem:devasymptC}\eqref{rem:devasymptC1} above, such that, on some neighbourhood $\nb(\hbo)\cap\nobreak\bS$, for any $\varphi,\beta,\ell,k',k''$ such that $g_{\varphi,\beta,\ell}^{(k',k'')}(0,\hb)\neq0$,
\begin{enumerate}
\item\label{exem:localholbasisC1}
if $\varphi_1\neq\varphi_2$, then $\ell_\hb(\beta+k')+\ell_{-\hb}(\beta+k'')\gg\ell_\hb(\beta_1)+\ell_{-\hb}(\beta_2)$,
\item\label{exem:localholbasisC2}
if $\varphi_1=\varphi_2$ and $\beta_1-\beta_2\not\in\ZZ$, then $\ell_\hb(\beta+k')+\ell_{-\hb}(\beta+k'')\gg\ell_\hb(\beta_1)+\ell_{-\hb}(\beta_2)$ or $\varphi=\varphi_1=\varphi_2$ and $\ell_\hb(\beta+k')+\ell_{-\hb}(\beta+k'')>\ell_\hb(\beta_1)+\ell_{-\hb}(\beta_2)$,
\item\label{exem:localholbasisC3}
if $\varphi_1=\varphi_2$, $\beta_1-\beta_2\in\ZZ$ and $w_1\neq w_2$, then $\ell_\hb(\beta+k')+\ell_{-\hb}(\beta+k'')>\ell_\hb(\beta_1)+\ell_{-\hb}(\beta_2)$ or $\varphi=\varphi_1=\varphi_2$, $\beta+k'=\beta_1$, $\beta+k''=\beta_2$ and $\ell<(w_1+w_2)/2$.
\end{enumerate}

Indeed, let us consider the second case for instance. If $\varphi\neq\varphi_1$, we apply the estimate given in Example \ref{exem:localholbasis} for $\ell_\hbo(\gamma)$. If $\varphi=\varphi_1=\varphi_2$ and $g_{\varphi,\beta,\ell}^{(k',k'')}(0,\hb)\neq0$, then $\beta$ satisfies both properties:
\begin{align*}
\beta+k'&=\beta_1\quad\text{or}\quad\ell_\hbo(\beta+k')>\ell_\hbo(\beta_1),\\
\beta+k''&=\beta_2\quad\text{or}\quad\ell_\mhbo(\beta+k'') >\ell_\mhbo (\beta_2).
\end{align*}
Recall that $\ell_\hb(\gamma)=\reel(\gamma'+i\hb\gamma'')$ and that, if $\hb\in\bS$, $\ell_{-\hb}(\gamma)=\ell_{1/\hb}(\gamma)$. The assumption implies that we cannot have simultaneously $\beta+k'=\beta_1$ and $\beta+ k''=\beta_2$. Therefore, in any case,
\[
\ell_\hbo(\beta+k')+\ell_\mhbo(\beta+k'')>\ell_\hbo(\beta_1)+\ell_\mhbo (\beta_2).
\]
One can choose $\nb(\hbo)$ and $\epsilon>0$ such that, for any $\hb\in\nb(\hbo)$,
\[
\ell_\hb(\beta+k')+\ell_{1/\hb}(\beta+k'')> \ell_\hb(\beta_1)+\ell_{1/\hb} (\beta_2)+\epsilon.
\]

In conclusion, when one of the assumptions of \eqref{exem:localholbasisC1}, \eqref{exem:localholbasisC2} or \eqref{exem:localholbasisC3} is fulfilled, the $[(\varphi_1,\beta_1,w_1),(\varphi_2,\beta_2,w_2)]$-block of $R^\phbo$ satisfies the property in the corollary (and even much better in the first two cases).

Let us now fix $\varphi$ and $\beta_1\in B_\hbo$, and let $\beta_2$ be the unique element of $B_\mhbo$ such that $\beta_2\in\beta_1+\ZZ$. Let $\ell\in\NN$ and let us consider the families $\bme^\phbo_{\varphi,\beta_1,\ell,\ell},\bme^\mphbo_{\varphi,\beta_2,\ell,\ell}$ lifting primitives elements. Assume first that $\varphi=0$. According to the definition of $P\psi_t^{0,\beta,\ell}C(\ell/2)$ as a residue, and arguing as above, we conclude that,  for $\hb\in\nb(\hbo)\cap\bS$,
\begin{multline*}
(\hb+1/\hb)^N\wt C((t\partiall_t-\beta\star\hb)^\ell\bme^\phbo_{0,\beta_1,\ell,\ell},\bme^\mphbo_{0,\beta_2,\ell,\ell})=(\hb+1/\hb)^N\mt^{2\beta\star\hb/\hb}(-\hb)^\ell\id\\
+\text{terms as in \eqref{exem:localholbasisC3}}.
\end{multline*}
Also according to the general form of $\wt C(\bme^\phbo_{0,\beta_1,\ell,\ell},\bme^\mphbo_{0,\beta_2,\ell,\ell})$ given by Proposition \ref{prop:devasymptC}, we conclude that
\begin{multline*}
(\hb+1/\hb)^N\wt C(\bme^\phbo_{0,\beta_1,\ell,\ell},\bme^\mphbo_{0,\beta_2,\ell,\ell})=(\hb+1/\hb)^N\mt^{2\beta\star\hb/\hb}\Lt^\ell\id/\ell!\\
+\text{terms as in \eqref{exem:localholbasisC3}}.
\end{multline*}
A similar result holds for $(\hb+1/\hb)^N\wt C(\bme^\phbo_{0,\beta_1,\ell,\ell-2k},\bme^\mphbo_{0,\beta_2,\ell,\ell-2k})$ for any $k$, as $\wt C$ is $\cR_\cX\otimes\cR_{\ovvv\cX}$-linear (with the convention that $\Lt^{\ell-2k}$ is replaced by $0$ if $\ell-2k<0$). This implies that the $[(0,\beta_1,w),(0,\beta_2,w)]$-block of $R^\phbo$ satisfies the property in the corollary for any $w\in\ZZ$.

When $\varphi\neq0$, we argue similarly by tensoring first by $\cE^{-\varphi/\hb}$.
\end{proof}

\subsection{Construction of an orthonormal basis}
We now go back to the situation of \S\ref{subsec:basis}. Therefore $\wt\cT=(\wt\cM,\wt\cM,\wt C)$ satisfies the twistor properties at $t=0$. We also assume that $\wt C$ takes values in $\cCh{\cX^*|\bS}$. By Remark \ref{rem:devasymptC}\eqref{rem:devasymptC2}, this means that, for any $\hbo\in\bS$, there exists an open neighbourhood $\nb(\hbo)$ such that $\wt C$ can be extended to $\cCh{X^*\times\nb(\hbo)}$.

Let $\varepsilong$ be the frame obtained by Lemma \ref{lem:glob}. It can be decomposed into subfamilies $\varepsilong_\varphi$.

\begin{proposition}\label{prop:Eorthbasis}
If $X$ is small enough, there exists a matrix $S(t,\hb)$ which is continuous on $X^*\times\nb(\Delta_0)$ and is holomorphic with respect to $\hb$, such that $\lim_{t\to0}S(t,\hb)=0$ uniformly on any compact set of the interior of $\Delta_0$, such that, if we set $\varepsilong'=\varepsilong\cdot(\id+S(t,\hb))$, the matrix $\wt C(\varepsilong',\ovv{\varepsilong'})$ is equal to $E$.
\end{proposition}

In fact, one can be more precise concerning the limit: the $L^2$-norm of $S(t,\cbbullet)$ on~$\Delta_0$ (or, equivalently, on~$\bS$), with respect to the standard Euclidean metric, has a~limit equal to $0$ when $t\to0$.

Similarly, the basis $\varepsilong'$ can be decomposed into subfamilies $\varepsilong'_\varphi$. We define the basis~$\wt\varepsilong$ by a rescaling depending on~$\varphi$:
\begin{equation}\label{eq:tildeeps}
\wt\varepsilong_\varphi=\varepsilong'_\varphi\cdot (e^{\hb\ov\varphi}\id).
\end{equation}

\begin{corollaire}\label{cor:orthonorm}
The frame $\wt\varepsilong$ is orthonormal with respect to $\wt C$.\qed
\end{corollaire}

\begin{proof}[\proofname\ of Theorem \ref{th:orbnilp}]
It follows from Corollary \ref{cor:orthonorm} that the restriction of $\cT$ to any small nonzero~$t$ is a polarized pure twistor of weight $0$. Therefore, according to \cite{Simpson97} (see also \cite[Lemma 2.2.2]{Bibi01c}), $\cT_{|X^*}$ defines $C^\infty$ bundle $H$ on~$X^*$ with a flat connection and a harmonic metric.
\end{proof}

\begin{proof}[\proofname\ of Proposition~\ref{prop:Eorthbasis}]
It relies on the following lemma:

\begin{lemme}\label{lem:globalC}
Let $\varepsilong$ be any global basis as constructed in Lemma \ref{lem:glob}. Then the matrix of $\wt C$ in this basis satisfies
\[
\wt C(\varepsilong,\ovv{\varepsilong})=(\id+R(t,\hb))\cdot E,
\]
with
\begin{equation}
\tag*{(\protect\ref{lem:globalC})$(*)$}\label{eq:globalC}
\begin{cases}
R(t,\hb)\text{ continuous on $X^*\times\nb(\bS)$ and holomorphic w.r.t. $\hb$},\\
\exists\delta>0,\; \lim_{t\to0}\Lt^\delta\vert R(t,\hb)\vert=0,\text{ uniformly for $\hb\in\bS$}.
\end{cases}
\end{equation}
\end{lemme}

The proof of the proposition then proceeds as follows. Let us denote by $H'\oplus H''$ the usual decomposition of $L^2(\bS)$: the functions in~$H'$ (\resp $H''$) extend holomorphically on $\{\module\hb<1\text{ \resp}>1\}$. Arguing as in \cite[Lemme~4.5]{Malgrange83db}, we find that there exist continuous mappings
\[
S':X\to \Mat_d(H'),\quad S'':X\to \Mat_d(H''),\quad S_0:X\to\Mat_d(\CC),
\]
where $d$ is the size of the matrices we are working with, uniquely determined by the following properties:
\begin{itemize}
\item
$S'(0,\hb)\equiv0$, $S''(0,\hb)\equiv0$, $S_0(0)=0$,
\item
$S'(t,0)\equiv0$, $S''(t,\infty)\equiv0$,
\item
$\id+R(t,\hb)=(\id+S'(t,\hb))(\id+S_0(t))(\id+S''(t,\hb))$.
\end{itemize}
We notice then that $S'$ (\resp $S''$) is continuous and holomorphic with respect to $\hb$ on $X^*\times\nb(\Delta_0)$ (\resp on $X^*\times\nb(\Delta_\infty)$).

For $S=S',S_0,S''$, we set $T=E^{-1}SE$. By assumption, the sesquilinear pairing $\wt C$ satisfies $\wt C^*=\wt C$. Therefore, $(\id+R)E=E(\id+R^*)$, as $E^*=E$ and, using uniqueness in the previous statement, we find
\[
T'=S^{\prime\prime*},\quad T_0=S_0^*,\quad T''=S^{\prime*}.
\]
Consequently, we get
\[
\wt C(\varepsilong,\ovv\varepsilong)=(\id+S')(\id+S_0)E(\id+S^{\prime*}).
\]
Moreover, as $S_0$ is independent of~$\hb$, the relation $ES_0=S_0^*E$ implies that $S_0$ is diagonal with respect to the $\varphi$-decomposition and we have $S_{0,\varphi,\varphi}^*=S_{0,\varphi,\varphi}$ for any~$\varphi$. If we set $\id+S_0=(\id+U_0)(\id+U_0^*)$ with $U_0$ diagonal with respect to the $\varphi$-decomposition, we then have
\[
(\id+S_0)E=(\id+U_0)E(\id+U_0^*).
\]
Now, the proposition is clear, by taking $\varepsilong'=\varepsilong\cdot(\id+{}^t\!S')^{-1}(\id+{}^tU_0)^{-1}$.
\end{proof}

\begin{proof}[\proofname\ of Lemma \ref{lem:globalC}]
Let us first indicate that we could have constructed a frame~$\varepsilong$ starting from admissible local holomorphic bases $\bme^\phbo$, by the same procedure (similar to that used in \cite{Bibi01c}). From Corollary \ref{cor:admissibleholC} we would only get an expression for $(\hb+1/\hb)^N\wt C(\varepsilong,\varepsilong)$, for $\hb\in\bS$. In order to get a good expression for $\wt C(\varepsilong,\varepsilong)$, we have to use the maximum principle, and thus work holomorphically in some neighbourhood of $\bS$, where the expression for $(\hb+1/\hb)^N\wt C(\varepsilong,\varepsilong)$ does not remains good because of the many terms $e^{-\hb\ov\psi+\psi/\hb}$. On the other hand, working with admissible $\cA$-bases will give us a better control of $(\hb+1/\hb)^N\wt C(\varepsilong,\varepsilong)$ on some neighbourhood of $\bS$.

In order to prove Lemma \ref{lem:globalC}, it is enough to check the property locally. Let us fix $\theta_o\in S^1$ and $\hbo\in\bS$. Given a pair $\bmea^\pmhbo$ of local admissible $\cA$-bases at $\pm\hbo$, we denote by $\varepsilong^\pmhbo$ the corresponding untwisted bases, by $\CA^\phbo$ the matrix $\wt C(\bmea^\phbo,\ovv{\bmea^\mphbo})$, by $\infC^\phbo$ the matrix $\wt C(\varepsilong^\phbo,\ovv{\varepsilong^\mphbo})$, and by $\infC$ the matrix $\wt C(\varepsilong,\ovv{\varepsilong})$.

Let us notice that, as the entries of $E$ are continuous on $X^*\times\Omega_0^*$ and holomorphic with respect to $\hb$, and have a modulus equal to $1$ when $\hb\in\bS$, we have, for $R^\phbo$ satisfying \ref{eq:globalC}, $E+R^\phbo(t,\hb)=(\id+R^{\prime\phbo}(t,\hb))E$ with $R^{\prime\phbo}$ satisfying \ref{eq:globalC}. We will then show $\infC=E+R^\phbo$, locally near $\hbo$. Moreover, according to Lemma \ref{lem:glob}, it is enough to show $\infC^\phbo=E+R^\phbo$ for some $R^\phbo$ satisfying \ref{eq:globalC}. Note also that, by definition, $\infC^\phbo={}^t\!A^{-1} \CA^\phbo\ovv A{}^{-1}$, where $A$ is defined by \eqref{eq:A}.

According to Lemma \ref{lem:bmeoA}, we can assume that the $\cA$-bases are compatible with the decomposition (DEC$^\sA$). Moreover, according to Lemma \ref{lem:Ahol}, we can also assume that Corollary \ref{cor:admissibleholC} applies to the chosen local admissible $\cA$-bases. We will follow the same steps.

\begin{lemme}\label{lem:caszero}
If $\varphi\neq\psi$, there exists a neighbourhood $\nb\othbo$ of $\othbo$ such that, for any $p\!\geq\!0$, $\lim_{t\to0}\mt^{-p}\,\big\vert\wt C(\bmea^\phbo_\varphi,\ovv{\bmea^\mphbo_\psi})\big\vert\!=\nobreak\!\nobreak0$ uniformly for $(t,\hb)\in X^*\times\bS\cap\nb\othbo$.
\end{lemme}

\begin{proof}
Arguing as in \S\ref{proof:devasymptC1}, we find that there exists $N$ such that the function $(\hb+\nobreak1/\hb)^Ne^{\hb\ov\psi-\varphi/\hb}\wt C(\bmea^\phbo_\varphi,\ovv{\bmea^\mphbo_\psi})$ has moderate growth, uniformly with respect to $(t,\hb)\in\nb\othbo$, and is zero if $\reel[(\varphi-\psi)/\hbo]\geq0$ on some open subset of $\nb\othbo$. As $e^{\hb\ov\psi-\varphi/\hb}\wt C(\bmea^\phbo_\varphi,\ovv{\bmea^\mphbo_\psi})$ is holomorphic with respect to~$\hb$ on $X^*\times\nb(\hbo)$ (\cf Remark \ref{rem:devasymptC}\eqref{rem:devasymptC2}), it also has moderate growth uniformly with respect to $\hb\in\nb(\hbo)$ by the maximum principle. The lemma follows from the rapid decay, uniformly with respect to $\hb\in\nb(\hbo)$, of $e^{\varphi/\hb-\hb\ov\psi}$ when \hbox{$\reel[(\varphi-\psi)/\hbo]<0$}.
\end{proof}

\begin{lemme}\label{lem:premiercas}
Let $\beta_1\in B_\hbo$, $\beta_2\in B_{-\hbo}$ be such that $\beta_1-\beta_2\not\in\ZZ$. Then there exists a neighbourhood $\nb(\hbo)$ of $\hbo$ and $\epsilon>0$ such that
\begin{equation}\tag*{(\protect\ref{lem:premiercas})($*$)}\label{eq:premiercas}
\lim_{t\to0}\mt^{-\epsilon}\,\mt^{-(\beta'_1+i\beta''_1\hb)}\mt^{-(\beta'_2+i\beta''_2/\hb)}\,\big\vert\wt C(\bmea^\phbo_{\varphi,\beta_1},\ovv{\bmea^\mphbo_{\varphi,\beta_2}})\big\vert=0
\end{equation}
uniformly for $\hb\in\nb(\hbo)\cap\bS$.
\end{lemme}

\begin{proof}
It follows from Corollary \ref{cor:admissibleholC} applied to the local $\cA$-admissible basis that \ref{eq:premiercas} holds after multiplication by $(\hb+\nobreak1/\hb)^N$ uniformly with respect to $\hb\in\nb(\hbo)\cap\bS$ and, as $|e^{\hb\ov\varphi-\varphi/\hb}|=1$ for $\hb\in\bS$, it also holds after multiplication by $(\hb+\nobreak1/\hb)^Ne^{\hb\ov\varphi-\varphi/\hb}$.

Arguing now as in \S\ref{proof:devasymptC1}, $(\hb+\nobreak1/\hb)^Ne^{\hb\ov\varphi-\varphi/\hb}\wt C(\bmea^\phbo_{\varphi,\beta_1},\ovv{\bmea^\mphbo_{\varphi,\beta_2}})$ can be written, for $\hb\in\nb(\hbo)$, as
\begin{multline*}
t^{\beta_1\star\hb/\hb}\ovt^{\beta_2\star\hb/\hb}\sum_{a,b\in\NN}g_{a,b}(\mt,\theta;\hb)(\log t)^a(\log\ovt)^b\\
=t^{\beta_1\star\hb/\hb}\ovt^{\beta_2\star\hb/\hb}\sum_{\ell\in\NN}h_\ell(\mt,\theta;\hb)\Lt^\ell,
\end{multline*}
where the entries of the matrices $g_{a,b}$ (hence $h_\ell$) are in $\cCh{\cY,\othbo}$. Condition \ref{eq:premiercas} translates then as the vanishing of suitable derivatives $\partial_{\mt}^kh_\ell$ of $h_\ell$ with respect to $\mt$ along $\mt=0$, for $\hb\in\nb(\hbo)\cap\bS$. As $\partial_{\mt}^kh_\ell(0,\theta;\hb)$ is holomorphic with respect to $\hb$, this vanishing holds for any $\hb\in\nb(\hbo)$, and in turn this implies that \ref{eq:premiercas} holds (up to changing $\epsilon$) for $(\hb+\nobreak1/\hb)^Ne^{\hb\ov\varphi-\varphi/\hb}\wt C(\bmea^\phbo_{\varphi,\beta_1},\ovv{\bmea^\mphbo_{\varphi,\beta_2}})$ uniformly on $\nb(\hbo)$.

As $e^{\hb\ov\varphi-\varphi/\hb}\wt C(\bmea^\phbo_{\varphi,\beta_1},\ovv{\bmea^\mphbo_{\varphi,\beta_2}})$ is holomorphic with respect to $\hb$ on $X^*\times\nb(\hbo)$ (\cf Remark \ref{rem:devasymptC}\eqref{rem:devasymptC2}), we get \ref{eq:premiercas} for it by the maximum principle. If we now restrict to $\hb\in\nb(\hbo)\cap\bS$, we get \ref{eq:premiercas}.
\end{proof}

In a similar way we get:
\begin{lemme}\label{lem:deuxiemecas}
Let $\beta\in B$ and $\beta_1\in B_\hbo\cap(\beta+\ZZ)$, $\beta_2\in B_{-\hbo}\cap(\beta+\ZZ)$. Let $w_1,w_2\in\ZZ$ be such that $w_1\neq w_2$. Then there exists a neighbourhood $\nb(\hbo)$ of $\hbo$ and $\delta>0$ such that
\begin{multline}\tag*{(\protect\ref{lem:deuxiemecas})($*$)}\label{eq:deuxiemecas}
\lim_{t\to0}\Lt^{\delta}\,\Lt^{-(w_1+w_2)/2}\mt^{-(\beta'_1+i\beta''_1\hb)}\mt^{-(\beta'_2+i\beta''_2/\hb)}\cdot{}\\
{}\cdot\big\vert \wt C(\bmea^\phbo_{\varphi,\beta_1,w_1},\ovv{\bmea^\mphbo_{\varphi,\beta_2,w_2}})\big\vert=0
\end{multline}
uniformly for $\hb\in\nb(\hbo)\cap\bS$.
\end{lemme}

\begin{proof}
Let us set $\beta_1=\beta+k_1$, $\beta_2=\beta+k_2$. The power of $\mt$ in \ref{eq:deuxiemecas} reads $\mt^{-2\beta\star\hb/\hb}\mt^{-(k_1+k_2)}$. On the other hand, arguing as in \S\ref{proof:devasymptC1}, $(\hb+\nobreak1/\hb)^N e^{\hb\ov\varphi-\varphi/\hb}\wt C(\bmea^\phbo_{\varphi,\beta_1,w_1},\ovv{\bmea^\mphbo_{\varphi,\beta_2,w_2}})$ can be expanded as
\[
\mt^{2\beta\star\hb/\hb}t^{k_1}\ovt^{k_2}\sum_{\ell\geq0}h_\ell(\mt,\theta,\hb)\Lt^\ell.
\]
We end the proof as for Lemma \ref{lem:premiercas}.
\end{proof}

Lastly, the case $w_1=w_2$ is treated similarly, concluding the proof of Lemma \ref{lem:globalC}, and hence of Proposition \ref{prop:Eorthbasis}.
\end{proof}

\subsection{A characterization by growth conditions}
Let $\wt\cT=(\wt\cM,\wt\cM,C)$ be an object of $\wtRTriples(X)$ which satisfies the assumptions in Theorem \ref{th:orbnilp}. If $\wt\cM$ is regular at $t=0$, then the matrix $E$ above is equal to identity. Moreover, the estimate for the limit of $\Lt^\delta R(t,\hb)$ holds in some neighbourhood of~$\bS$, not only on~$\bS$. It follows that the matrix $S(t,\hb)$ is continuous on $X\times\nb(\Delta_0)$ and holomorphic with respect to $\hb$ on this set. As a consequence, we have a good estimate on $X^*\times\nb(\Delta_0)$ for the norm of the global basis~$\varepsilong$ with respect to the harmonic metric defined by the twistor object on~$X^*$ and, using the base changes of \oldS\S\ref{subsec:formalbasis}, \ref{subsec:localholbasis} and \ref{subsec:Abasis}, we get an estimate for the norm of any local admissible holomorphic basis. In particular, we can characterize $\wt\cM$ from $\wt\cM_{|\cX^*}$ as the subsheaf consisting of local sections whose norms, with respect to the harmonic metric, have moderate growth near the origin. This was one of the main points in \cite[\S5.4]{Bibi01c}. It was also the basic tool for computing $L^2$ cohomology in \cite[\S6.2]{Bibi01c}, in a way analogous to that of~\cite{Zucker79} (see also \cite[Chap\ptbl20]{Mochizuki07} for a slightly different approach).

In the irregular case, we do not get a good estimate for the norm of the basis $\varepsilong$ when $\hb=1$, as Proposition \ref{prop:Eorthbasis} does not give enough information on the matrix~$S$ when $\hb\in\bS$. Moreover, even on the interior of $\Delta_0$ where we have a good estimate, the exponential terms prevent us to get a convenient estimate for the norm of any admissible local holomorphic basis. As a consequence, we do not get a characterization of the meromorphic extension~$\wt\cM$ of $\wt\cM_{|\cX^*}$ in terms of growth with respect to the harmonic metric.

Nevertheless, let us note that Corollary \ref{cor:orthonorm} implies that the restriction $\wt\varepsilon_0$ to $\hb=0$ of the frame $\wt\varepsilon$ is an orthonormal frame for the harmonic metric $h$ on the Higgs bundle $(\wt\cM/\hb\wt\cM)_{|X^*}$. Therefore, the frame $\varepsilong'_0$ is also an orthonormal frame and the frame $\varepsilong_0$, which is the restriction at $\hb=0$ of the frame obtained in Lemma \ref{lem:glob} is asymptotically orthonormal, according to Proposition \ref{prop:Eorthbasis}. Restricting \eqref{eq:untwist} to $\hb=0$ and using the definition of an admissible local $\cA$-basis gives a meromorphic frame of $\wt M_0$ whose $h$-norm has moderate growth at the origin. As a consequence we get:

\begin{corollaire}
Let $\wt\cT=(\wt\cM,\wt\cM,\wt C)$ be an object of $\wtRTriples(X)$ which satisfies the assumptions in Theorem \ref{th:orbnilp}. Then the meromorphic extension $\wt\cM/\hb\wt\cM$ of the Higgs bundle $(\wt\cM/\hb\wt\cM)_{|X^*}$ is characterized as the subsheaf of $j_*(\wt\cM/\hb\wt\cM)_{|X^*}$ consisting of sections whose $h$-norm has moderate growth at the origin. Similarly, the parabolic filtration defined by the metric is identified with the filtration induced by $V_{(0)}^\cbbullet\wt\cM$.\qed
\end{corollaire}

However, if we accept coefficients which are $C^\infty$ with respect to $\hb$, it is reasonable to expect that we can recover a characterization by moderate growth for $\cM$ (this question has now been completely solved by T\ptbl Mochizuki in \cite{Mochizuki08}). Such a procedure was used in a simpler situation in \cite{Bibi04} that we describe now.

We will denote by $\cO_\cX^\infty$ the subsheaf of $\cC_\cX^\infty$ consisting of functions which are holomorphic with respect to $t$, but possibly $C^\infty$ with respect to $\hb$. We can then consider $\cR_\cX^\infty=\cO_\cX^\infty\langle\hb\partial_t\rangle$. We will also need to consider the sheaf $\cO_\cX^0$ consisting of functions which are continuous, $C^\infty$ away from $\hb=0$, and holomorphic with respect to~$t$ (the notation is therefore not precise enough, but we keep it for simplicity). We define $\cR_\cX^0$ similarly.

Let $j:X^*\hto X$ or $\cX^*\hto\cX$ denote the open inclusion. Let $\pi:\cX\to X$ denote the projection.

\subsubsection*{A simple example}
Let us start from a polarized regular twistor $\cD[\tm]$-module $\wt\cT^\reg=(\wt\cR,\wt\cR,\wt C^\reg)$ on~$X$, and let us assume that $\wt\cT=(\wt\cM,\wt\cM,\wt C)$ is obtained from $\wt\cT^\reg$ by twisting with $\cE^{\varphi/\hb}$ (recall that it is a $\cO_\cX[\tm]$-free module of rank one, with the action of $\partiall_t$ defined on its generator $1$ by $\partiall_t1=\hbm\varphi\cdot 1$). In particular, we have $\wt C=e^{-\hb\ov\varphi+\varphi/\hb}\wt C^\reg$. Then $\cT^\reg$ defines a harmonic bundle $H$ on~$X^*$, with harmonic metric $h^\reg$ and, if $\pi:\cX\to X$ denotes the projection, we know that $\wt\cR$ is characterized by the moderate growth condition with respect to $\pi^*h^\reg$. As $\wt\cM$ is isomorphic, as a $\cO_\cX[\tm]$-module, to $\wt\cR$, it is also characterized by a moderate growth condition, but not with respect to the harmonic metric $h$ defined by~$\wt\cT$. Let us notice that, using the isomorphism
\[
\cCh{\cX}\otimes\wt\cR_{|\cX^*}\To{{}\cdot e^{-\hb\ov\varphi}}\cCh{\cX}\otimes\wt\cM_{|\cX^*}
\]
the object $\wt\cT$ restricted to $X^*$ also defines a polarized smooth twistor structure of weight $0$, with the same associated $C^\infty$-bundle~$H$ (but with distinct metric $h$ and distinct flat connection $D_V$; in fact, $h=e^{-2\reel\varphi}h^\reg$ and $D_V=D_V^\reg+d\varphi$, see \cite[\S2]{Bibi04}).

On the other hand, there is no isomorphism $\wt\cT_{|X^*}\to\wt\cT^\reg_{|X^*}$, as the flat global section $e^{-\varphi/\hb}\cdot 1$ of $\cE^{\varphi/\hb}$ only exists on $\hb\neq0$; as a matter of fact, $h$ and $h^\reg$ do differ.

In order to circumvent this difficulty, we extend the coefficients of $\wt\cM$ by tensoring with $\cO_\cX^\infty[\tm]$. Restricting to $\cX^*$, we have an isomorphism of $\cO_\cX^\infty[\tm]$-modules
\begin{equation}\label{eq:echb}
\cO_{\cX^*}^\infty[\tm]\otimes_{\cO_{\cX^*}[\tm]}\wt\cM_{|\cX^*}\To{e^{\ov\hb\varphi}}\cO_{\cX^*}^\infty[\tm]\otimes_{\cO_{\cX^*}[\tm]}\wt\cR_{|\cX^*},
\end{equation}
where $\ov\hb$ denotes the usual conjugate of~$\hb$. This isomorphism is not compatible with the $\cR_\cX^\infty$-structure (\ie the action of~$\partiall_t$), but it is compatible with $\wt C$ and $\wt C^\reg$: indeed, on~$\bS$, we have $\ov\hb=1/\hb$.

Under this isomorphism, the subsheaf of $j_*\big(\cO_{\cX^*}^\infty[\tm]\otimes_{\cO_{\cX^*}[\tm]}\wt\cM_{|\cX^*}\big)$ consisting of sections having moderate growth with respect to $\pi^*h$ is identified with the subsheaf of $j_*\big(\cO_{\cX^*}^\infty[\tm]\otimes_{\cO_{\cX^*}[\tm]}\wt\cR_{|\cX^*}\big)$ of sections having moderate growth with respect to $\pi^*h^\reg$ (use \cite[(2.2) and (2.3)]{Bibi04} with $d\varphi$ instead of $-dt$).

\subsubsection*{A generalization of the construction}
We generalize below the previous construction, although it does not seem to be enough to recover a characterization by moderate growth. We proceed a little differently, as we do not have at hand a regular $\cR_\cX[\tm]$-module like $\wt\cR$ in the previous example.

\begin{enumerate}
\item\label{prop:oxinfstep1}
In the first step, we construct an auxiliary locally free $\cO_\cX^0[\tm]$-module $\wt\cM^\aux$ with a compatible action of $\cR_\cX^0$. In the simple example above, we would have $\wt\cM^\aux=\cE^{\ov\hb\varphi}\otimes\wt\cM$.
\item\label{prop:oxinfstep2}
In the second step, we show that there is an isomorphism
\begin{equation}\label{eq:isoaux}
\cO_{\cX^*}^0\otimes_{\cO_{\cX^*}}\wt\cM_{\cX^*}\isom \wt\cM_{\cX^*}^\aux
\end{equation}
compatible with the $\cR_{\cX^*}^0$-actions. In the simple example above, it is hidden behind the isomorphism \eqref{eq:echb}. As a consequence, we get, by transporting $\wt C$ through the previous isomorphism, a $\cR_{\cX|\bS}^\infty\otimes_{\cC^\infty_\bS}\cR_{\ovvv\cX|\bS}^\infty$-sesquilinear pairing
\[
\wt C^\aux:\wt\cM_{|\bS}^\aux\otimes_{\cC^\infty_\bS}\ovv{\wt\cM_{|\bS}^\aux}\to\Dbh{X}[\tm],
\]
and we get an isomorphism of triples
\[
\cO_{\cX^*}^0\otimes_{\cO_{\cX^*}}(\wt\cM_{\cX^*},\wt\cM_{\cX^*},\wt C)\isom (\wt\cM_{\cX^*}^\aux,\wt\cM_{\cX^*}^\aux,\wt C^\aux).
\]
\end{enumerate}

\subsubsection*{Step one: construction of $\wt\cM^\aux$}
Let us denote by $\wt\cM^{\wedge\aux}$ the $\cO_{\wh\cX}^\infty[\tm]$-module
\[
\tbigoplus_j(\cE^{(\ov\hb+1/\hb)\varphi_j}\otimes\wt\cR_j^\wedge),
\]
with the natural $\cR_{\wh\cX}^\infty$-structure. We will construct $\wt\cM^\aux$ equipped with an isomorphism
\[\tag{DEC$^{\wedge\aux}$}
\wt\cM^{\aux\wedge}\isom\wt\cM^{\wedge\aux}.
\]
We first recall the well-known Malgrange-Sibuya Theorem in the present situation.

We use the notation of \S\ref{subsec:sect0} for $\cZ$ (real blow up of $\hb=0$ and of $t=0$ in $X$) and $\cA_\cZ$, and we denote by $p:\cZ\to\cX$ the natural map. We define the sheaf $\cA_\cZ^\infty$ by relaxing the holomorphy condition with respect to $\hb$. If $\hbo\neq0$, $p^{-1}\ohbo$ consists of a circle $S^1$ with coordinate $\theta$, while if $\hbo=0$, then $p^{-1}\oO=S^1\times S^1$ with coordinates $(\theta,\arghb)$. Let us denote by $\GL_d(\cA_\cZ^\infty)$ the sheaf of groups consisting of invertible matrices of size $d$ with entries in $\cA_\cZ^\infty$, and by $\GL_{d,\{\hb=0\}}^{<\{t=0\}}(\cA_\cZ^\infty)$ the subsheaf of matrices of the form $\id_d+M$, where the entries of $M$ are infinitely flat along $\{t=0\}$ and vanish when restricted to $\hb=0$.

\begin{theoreme}[Malgrange-Sibuya]
For any $\hbo\in\Omega_0$, the image of the pointed set $H^1\big(p^{-1}\ohbo,\GL_{d,\{\hb=0\}}^{<\{t=0\}}(\cA_\cZ^\infty)\big)$ in the pointed set $H^1\big(p^{-1}\ohbo,\GL_{d,\{\hb=0\}} (\cA_\cZ^\infty)\big)$ is the identity.
\end{theoreme}

\begin{proof}
We recall the classical proof (\cf \cite[Appendix]{Malgrange83bb}) for the sake of completeness. Let $\lambda$ be an element of $H^1\big(p^{-1}\ohbo,\GL_{d,\{\hb=0\}}^{<\{t=0\}}(\cA_\cZ^\infty)\big)$. Denote by $\cE_\cZ$ the sheaf of $C^\infty$-functions on~$\cZ$. Then one checks that the previous set, where $\cA_\cZ^\infty$ is replaced with $\cE_\cZ$, reduces to the identity. If $\lambda$ is defined on a covering $(U_i)$ of $p^{-1}\ohbo$, then there exist sections $\mu_i\in\Gamma\big(U_i,\GL_{d,\{\hb=0\}}^{<\{t=0\}}(\cE_\cZ)\big)$ such that $\lambda_{ij}=\mu_i\mu_j^{-1}$. Moreover, $\ov\partial_t$ is well-defined on $\cE_\cZ^{<\{t=0\}}$, and the $\mu_i^{-1}\ov\partial_t\mu_i$ glue together as a section of $\Mat_{d,\{\hb=0\}}^{<\{t=0\}}(\cE_\cZ)$ on $p^{-1}\ohbo$. This is then also a germ of section~$\nu$ of $\Mat_{d,\{\hb=0\}}^{<\{t=0\}}(\cE_{X\times\wt\Omega_0})$ at $\ohbo$ if $\hbo\neq0$, or along the circle $S^1$ with coordinate $\arghb$ if $\hbo=0$. According to \cite[Chap\ptbl IX, Th\ptbl1]{Malgrange58}, we can solve $\ov\partial_t\eta=-\eta\nu$ with
\begin{itemize}
\item
$\eta\in \GL_d(\cE_{\cX,\ohbo})$ if $\hbo\neq0$,
\item
$\eta\in\Gamma\big(S^1,\GL_d(\cE_{X\times\wt\Omega_0})\big)$ if $\hbo=0$ and $S^1$ is the circle where $\arghb$ varies,
\end{itemize}
and, in both cases, $\eta(0,\hb)\equiv\id$. Moreover, as $\nu_{|\hb=0}=0$, the function $\eta(t,0)$ is holomorphic and, replacing $\eta$ with $\eta(t,0)^{-1}\eta$ gives $\eta_{|\hb=0}=\id_d$. We now replace $\beta_i$ with $\beta_i\eta$, which belongs to $\Gamma\big(U_i,\GL_{d,\{\hb=0\}} (\cA_\cZ^\infty)\big)$.
\end{proof}

Let us first work at $\hbo\neq0$. According to the decomposition (DEC$^\sA$) of Proposition \ref{prop:decompA}, if we denote by $\wt\cM^\el_\ohbo$ the sum $\oplus_\varphi(\cE^{\varphi/\hb}\otimes\wt\cR_{\varphi,\ohbo})$, then $\wt\cM_\ohbo$ determines a Stokes cocycle $\lambda$ relative to $\wt\cM^\el_\ohbo$. If we fix bases of $\wt\cR_{\varphi,\ohbo}$ and if $d$ denotes the rank of $\wt\cM$, this cocycle belongs to $H^1\big(p^{-1}\ohbo,\GL_d^{<\{t=0\}}(\cA_\cZ)\big)$ (the condition on the restriction to $\{\hb=0\}$ is now empty).

We now consider $\wt\cM^{\aux,\el}_\ohbo\defin\oplus_\varphi(\cE^{(\ov\hb+1/\hb)\varphi}\otimes\wt\cR_{\varphi,\ohbo})$ equipped with its connection $\nabla^{\aux,\el}$ and the same $\cO_{\cX,\ohbo}[\tm]$-basis as $\wt\cM^\el_\ohbo$ (in other words, we only change the connection on $\wt\cM^\el_\ohbo$). If we denote by $\partiall_t^\reg$ the action of $\partiall_t$ on $\wt\cR_{\varphi,\ohbo}$, then the action of $\partiall_t$ on $\cE^{\varphi/\hb}\otimes\wt\cR_{\varphi,\ohbo}$ is $\partiall_t^\reg+\varphi'$, and that on $\cE^{(\ov\hb+1/\hb)\varphi}\otimes\wt\cR_{\varphi,\ohbo}$ is $\partiall_t^\reg+(1+\module\hb^2)\varphi'$.

Let us note that the argument of $\ov\hb+1/\hb$ is the same as the argument of $1/\hb$. Let us denote by $E_c$ the block-diagonal matrix having diagonal $\varphi$-blocks $e^{\ov\hb\varphi}\id_{d_\varphi}$. Then, if we denote by $\lambda^\aux$ the cocycle $E_c\lambda E_c^{-1}$, we see that $\lambda^\aux$ is $\cR_\cY^\infty[\tm]$-linear (\ie is compatible with the action of $\partiall_t$ on $\wt\cM^{\aux,\el}_\ohbo$) and still belongs to $H^1\big(p^{-1}\ohbo,\GL_d^{<\{t=0\}}(\cA_\cZ^\infty)\big)$. If for some suitable covering $(U_i)$ of $p^{-1}\ohbo$ we set $\lambda^\aux=\mu_i\mu_j^{-1}$ according to the theorem of Malgrange-Sibuya, we conjugate on~$U_i$ the matrix of $\nabla^{\aux,\el}$ by $\mu_i$ and we get a new connection $\nabla^\aux$ which does not depend on~$U_i$. It therefore defines a connection $\nabla^\aux$ on the $\cO_{\cX,\ohbo}^\infty[\tm]$-module $\wt\cM^{\aux,\el}_\ohbo$, that we now call $\wt\cM^\aux_\ohbo$. Moreover, we have a fixed identification $\wt\cM^{\aux,\el\wedge}_\ohbo=\wt\cM^{\wedge\aux}_\ohbo$. Last, if we have two such constructions (depending on the choice of $\mu_i$ or of the covering $(U_i)$), we get isomorphic $\cR_{\cX,\ohbo}^\infty[\tm]$-modules; the isomorphism induces the identity at the formal level, therefore the isomorphism is unique. This enables us to glue the various $\wt\cM^\aux_\ohbo$ when $\hbo$ varies in $\Omega_0^*$.

Let us now work at $\hbo=0$. We consider a covering $(U_i)$ of $p^{-1}\oO$ on which the decomposition (DEC$^\sA_0$) holds. Let us fix bases of the $\cA_\cZ[\tm]$-modules $\wt\cR^\sA_{i,\varphi}$ (the index $i$ refers to the open set $U_i$ where $\wt\cR^\sA_{i,\varphi}$ is defined). We assume that these bases lift the same $\cO_{\wh\cX,\oO}[\tm]$-basis of $\wt\cR_{\varphi,\oO}^\wedge$ and, when restricted to $\hb=0$, come from a given basis of $\wt R_\varphi$. We identify then the various $\cA_\cZ[\tm]$-modules $\wt\cR^\sA_i$ to $\cA_\cZ[\tm]^d$, equipped with the connection $\nabla_i^\reg$. Twisting $\wt\cR^\sA_{i,\varphi}$ with $\cE^{\varphi/\hb}$ (\resp $\cE^{(\ov\hb+1/\hb)\varphi}$) gives the connection $\nabla_i$ (\resp $\nabla_i^{\aux,\el}$).

With respect to these identifications and using the isomorphisms (DEC$^\sA_0$), we get a cocycle $\lambda\in H^1\big(p^{-1}\oO,\GL_{d,\{\hb=0\}}^{<\{t=0\}}(\cA_\cZ)\big)$. We conjugate it as above to get $\lambda^\aux$. Using the same argument as above, and as $\ov\hb$ vanishes at $\hb=0$, we have $\lambda^\aux\in H^1\big(p^{-1}\oO,\GL_{d,\{\hb=0\}}^{<\{t=0\}}(\cA_\cZ^\infty)\big)$. If we set $\lambda^\aux_{ij}=\mu_j\mu_i^{-1}$ then, conjugating $\nabla_i^{\aux,\el}$ by $\mu_i$ for each $i$, we get a connection $\nabla^\aux$ with matrix (in the fixed bases) having entries which are continuous on $\cX$ near $\oO$, meromorphic with respect to~$t$, and which become $C^\infty$ when expressed in polar coordinates with respect to $\hb$ (in particular, they are $C^\infty$ away from $\hb=0$). We thus get $\wt\cM^\aux_\oO$. By uniqueness, it coincides with $\wt\cM^\aux$ constructed previously, when restricted to $\hb\neq0$.

\subsubsection*{Step two: the isomorphism \eqref{eq:isoaux}}
It is now straightforward. We have a formal isomorphism $\wt\cM^\wedge\to\wt\cM^{\aux\wedge}$ induced by $E_c$. For any $\hbo$, this isomorphism exchanges the Stokes cocycles along $p^{-1}(\hbo)$. Therefore, it induces a local isomorphism \eqref{eq:isoaux} on $X^*\times\nb(\hbo)$. These local isomorphisms glue together as they have the same underlying formal isomorphism.\qed

\backmatter
\providecommand{\bysame}{\leavevmode\hbox to3em{\hrulefill}\thinspace}
\providecommand{\MR}{\relax\ifhmode\unskip\space\fi MR }
\providecommand{\MRhref}[2]{%
  \href{http://www.ams.org/mathscinet-getitem?mr=#1}{#2}
}
\providecommand{\href}[2]{#2}

\end{document}